\documentclass[11pt]{article}
\usepackage{mydef2col}
\usepackage{markArticle}

\theoremstyle{definition}
\newtheorem{definition}{Definition}[section]
\newtheorem{example}{Example}

\newcommand{\Krylov}{\mathcal{K}}
\DeclareMathOperator{\spn}{span}
\newcommand{\bvec}[1]{\mathbf{#1}}
\DeclareMathOperator{\re}{\Ree}

\title{Diagonal Frog: High-order positivity-preserving FD schemes \\ for anisotropic Fokker-Planck equations}
\subtitle{From Financial engineering to Active Matter}

\shorttitle{Positivity-preserving FD schemes}

\author{
    \authorstyle{ Andrey Itkin}
    \newline \newline
    \institution{FREE department, Tandon School of Engineering, New York University, email: \url{aitkin@nyu.edu}} \\
}

\date{\today}

\begin{document}

\maketitle

\lettrineabstract{The Fokker–Planck equation is fundamental to statistical mechanics, yet in settings with multiple state variables, anisotropic (cross-) diffusion, and jumps, conventional discretizations frequently produce non-physical negative probability densities. Building on the operator approach of \cite{ItkinBook}, we introduce a family of "Diagonal Frog" discretizations whose spatial operators are eventually M-matrices (EM-matrices). Although these operators lack a local M-matrix structure, positivity of the directional sub-operators emerges in the spirit of Zeno's paradox: the matrix exponential, assembled as the limit of infinitely many ever-smaller substeps, is provably nonnegative after a short transient even though no single substep is. For the mixed-derivative block, whose generator is not eventually nonnegative, positivity instead rests on a factorized resolvent solver and holds conditionally, on an explicit step-size window; discrete mass is conserved exactly by the splitting for every step size. The resulting schemes are second-order accurate in time and space and require $O(m^2 N + m^3)$ operations per time step, where $m$ is the dimension of the Krylov subspace used to apply the exponential. As stress tests, we solve a two-dimensional anisotropic Fokker--Planck equation in the strong cross-diffusion regime against an exact Gaussian reference, a Kramers escape problem in a double-well potential, and an advection-dominated problem, and observe that the schemes remain stable, nonnegative, and mass-conservative for a wide range of P\'ecklet numbers (so, don't need any flux limiter). Finally, we extend the construction to multidimensional processes and to the backward Kolmogorov equation with jumps.
}

\section{Introduction}

% full introduction is in file fdPositiveLight5.tex

The Fokker–Planck equation (FPE) models the time evolution of probability density functions (PDFs) $p(t,x)$ in non-equilibrium stochastic systems \cite{Risken1996,Kwok2018}. A fundamental physical requirement is that $p(t,x) \ge 0$ for all coordinates and times. However, modern applications involving anisotropy (cross-diffusion) and non-locality (Lévy jumps) frequently cause standard finite-difference (FD) schemes to produce unphysical negative values. These negative "ghost densities" violate mass conservation and render thermodynamic quantities, like the Gibbs entropy, mathematically undefined.

The struggle to preserve positivity is pervasive across quantitative finance, physics, and biology. In multi-variable physical systems or active matter biology (e.g., Run-and-Tumble particles undergoing Motility-Induced Phase Separation), cross-diffusion and steep gradients often lead to numerical undershoots \cite{ParkPetrosian1996,CamporealeEtAl2013a}. Similarly, when modeling anomalous diffusion or Lévy flights via Partial Integro-Differential Equations (PIDEs), standard quadrature or FFT solvers frequently introduce positivity-breaking Gibbs phenomena \cite{gao2015second,ContTankov,Cherubinietal2010,hirsa2013,Duffy2006}. Identical bimodal and oscillatory artifacts plague financial models dealing with correlated jumps and stochastic volatility \cite{ItkinLipton2014,ItkinCarrBarrierR3}.

To suppress these numerical artifacts, researchers across various disciplines have developed specialized techniques, summarized in Table~\ref{tab:scheme_comparison}. These approaches can be categorized into three primary classes
\begin{table}[!htbp]
\centering
\begin{tabular}{|l|l|l|}
\toprule
\textbf{Scheme Type} & \textbf{Positivity} & \textbf{Common Issues} \\
\midrule
Standard FD          & No   & Negative $p$ near gradients or cross-diffusion. \\
Chang--Cooper        & Yes  & Hard to generalize to $d > 1$. \\
Log-transformation   & Yes  & Oversensitive to near-zero densities. \\
Finite Volume (NTPFA)& Yes  & Complex to implement on non-orthogonal meshes. \\
\bottomrule
\end{tabular}
\caption{Comparison of numerical schemes (spatial order greater than one) for the Fokker--Planck equation.}
\label{tab:scheme_comparison}
\end{table}

These approaches can be categorized into three primary classes:

\begin{enumerate}
\item \textbf{Exponential Fitting and Transformations:} Positivity can be
enforced via the Chang-Cooper method \cite{ChangCooper1970} or
exponential-fitting schemes \cite{Duffy2006}. However, these are often restricted
to one dimension or suffer from reduced first-order accuracy to satisfy M-matrix
stability requirements \cite{ParkPetrosian1996}. Alternatively, logarithmic
transformations (e.g., the SILL scheme) guarantee positivity by substituting $p =
e^f$; while effective, this approach increases computational complexity by
introducing non-linearities into the FPE \cite{YuanShu2005,Gong2021SILL}.

\item \textbf{Limiters and Non-Standard FD:} This class includes Total Variation Diminishing (TVD) limiters \cite{Osher1983, OsherChakravarthy1984, Harten1983,
PareschiRusso2005}, non-linear flux limiters \cite{QiSu2025}, and non-standard
finite difference methods that employ denominator adjustments
\cite{ChapwanyaLubuma2015}. While these schemes maintain high-order accuracy in
smooth regions, they typically revert to first-order stable schemes near
discontinuities to preserve the TVD property \cite{PareschiRusso2005}.

\item \textbf{High-Order and ML-Enhanced Schemes:} Recent developments include high-order FD schemes that necessitate strict spatial mesh constraints to
preclude long-term oscillations \cite{HuZhang2023, LiuGaoZhang2024}.
Additionally, machine-learning-enhanced schemes, such as WENO-DS, have been
trained to suppress unphysical oscillations in financial modeling contexts
\cite{KossaczkaEhrhardt2024,HoangEhrhardt2026}.

\end{enumerate}

Despite these advances, robustly handling multi-dimensional cross-diffusion ($D_{ij} \neq 0$) without violating the discrete maximum principle remains computationally difficult \cite{CamporealeEtAl2013a, CamporealeEtAl2013b, Albert2013, CamporealeEtAl2013c, mcfarland2025efficient}.

To systematically handle mixed-derivative terms and non-local Lévy operators while guaranteeing positivity, the author previously introduced an FD framework for financial PIDEs based on pseudo-differential operators and M-matrix (and EM-matrix) theory \cite{Itkin2014,Itkin2014b,Itkin3D,ItkinBook}. The core mechanism is an operator-splitting technique built from a discretization that represents operators in coordinate space using EM-matrices, preserving non-negativity for both backward and forward equations.

Although this M-matrix framework has demonstrated success in financial Lévy models \cite{ItkinBook}, it has yet to be fully extended to complex physical systems. In this paper, we bridge that gap by applying this approach to the multivariate FPE (also those that could govern active Lévy flyers). We focus on three main objectives:

\begin{enumerate}
\item A detailed derivation of the positivity-preserving splitting for two-dimensional phase spaces.
\item A rigorous proof that the scheme stays stable even under strong cross-diffusion.
\item A numerical benchmark comparing our results with standard solvers, showing where traditional FD methods fail (especially for heavy-tailed distributions).
\end{enumerate}

The rest of the paper is organized as follows. \Cref{sec:disc1d} establishes the mathematical foundation of the FPE and M- and EM-matrix discretization and introduces the 1D Diagonal Frog (DF) scheme, stability and convergence analysis and includes the "Zeno's paradox" interpretation. \Cref{sec:disc2d} generalizes these results for the two-dimensional case, elaborates on a version of Strang's splitting that preserves conditional positivity. The most important part is a derivation of the mixed derivatives iterative scheme of \cite{Itkin3D} for the FPE equation. \Cref{sec:schemes} compares various time integration schemes including the exponential integrator which is a part of the DF scheme. In \cref{sec:krylov} we describe Krylov methods of computing the matrix exponential for a banded matrix. \Cref{sec:numerics} presents various numerical experiments to justify our theoretical findings. In \cref{sec:BKE} in a spirit of \cite{ItkinBook} we extend the proposed method to Backward Kolmogorov equations and jump-diffusion models for some \LY processes with known characteristic function. Finally, \cref{sec:conclusion} offers concluding remarks.

\section{Spatial Discretisation of the 1D Fokker--Planck Equation} \label{sec:disc1d}

The numerical schemes proposed in this paper rely on properties of M-matrices and eventually nonnegative (EM) matrices to preserve positivity. For brevity, a rigorous discussion of the definitions and properties of M-matrices and EM-matrices, along with the proofs supporting our analysis, is provided in \cref{app0}. Readers already familiar with the topic (see, e.g., \cite{BermanPlemmons94, NoutsosTsatsomeros2008, OleskyEtAl2009, ItkinBook}) may find this appendix supplemental.

The discretisation schemes that exploit these matrix properties are constructed for the continuous FPE. In its general form for a $d$-dimensional continuous state space, the FPE reads
\begin{equation} \label{FP_def}
\frac{\partial p(t,x)}{\partial t} = -\sum_{i=1}^d \frac{\partial}{\partial x_i}\big[\mu_i(t,x)p(t,x)\big] + \frac12 \sum_{i,j=1}^d \frac{\partial^2}{\partial x_i\partial x_j}\big[D_{ij}(t,x)p(t,x)\big],
\end{equation}
where $\bm{\mu}$ is the drift vector and $\bm{D}$ is the diffusion tensor. The right-hand side of \eqref{FP_def} can be written compactly by introducing the Fokker–Planck operator $\calL$:
\begin{equation} \label{FP_comp}
\frac{\partial p}{\partial t} = \calL\,p,\qquad
\calL\,\cdot\, = -\sum_{i=1}^d\frac{\partial}{\partial x_i}\big[\mu_i(t,x)\,\cdot\,\big] + \frac12\sum_{i,j=1}^d\frac{\partial^2}{\partial x_i\partial x_j}\big[D_{ij}(t,x)\,\cdot\,\big].
\end{equation}

In a one-dimensional (1D) case, we re-write \eqref{FP_comp} as
\begin{equation}
  \frac{\partial p}{\partial t} = -\frac{\partial}{\partial x}[\mu(x,t)\,p]
  + \frac{\partial^2}{\partial x^2}[D(x,t)\,p],
  \label{eq:fpe}
\end{equation}
where $D(x,t) = \tfrac{1}{2}\sigma^2(x,t) \geq 0$ is the diffusion coefficient.
This is a parabolic PDE with the fundamental conservation property
\begin{equation}
  \int_{-\infty}^{\infty} p(x,t)\,dx = 1, \qquad p(x,t) \geq 0,
  \quad \text{for all } t \geq 0,
\end{equation}
i.e., $p(\cdot,t)$ is a probability density for all time.

Consider a uniform spatial grid $\Omega_h\colon x_1 < x_2 < \cdots < x_n$ with
mesh spacing $h$. After spatial discretization on this grid, the FPE reduces to the linear ODE system
\begin{equation}   \label{eq:ode}
\dot{\bvec{p}}(t) = L(t)\,\bvec{p}(t), \qquad \bvec{p}(0) = \bvec{p}_0,
\end{equation}
where $L(t) \in \R^{n \times n}$ is the discretized Fokker--Planck operator, and $\bvec{p}(t) \in \R^n$ approximates the PDF at the grid points.
The two paramount numerical properties that any time-integration scheme must
preserve are:
\begin{enumerate}[label=(\roman*)]
\item \textbf{Positivity}: $p^n_i \geq 0$ for all $i,n$, so that the discrete PDF remains nonnegative.

\item \textbf{Mass conservation}: $\sum_i p^n_i = 1$ (or constant) for all $n$.
\end{enumerate}
Below we will analyse several time-integration strategies with respect to these properties, culminating in the Krylov subspace exponential integrator as the
method of choice.

\subsection{Divergence form and flux}

We write the FPE \eqref{eq:fpe} in conservative (divergence) form
\begin{equation}   \label{eq:consform}
\frac{\partial p}{\partial t} = -\frac{\partial J}{\partial x},
\end{equation}
where the probability flux $J(x,t)$ is
\begin{equation}   \label{eq:flux}
J(x,t) = \mu(x,t)\,p - \frac{\partial}{\partial x}[D(x,t)\,p].
\end{equation}
This representation makes the conservation structure explicit: integrating
\eqref{eq:consform} over any interval and applying the fundamental theorem of
calculus shows that the total probability changes only through boundary fluxes,
and is exactly preserved under zero-flux or absorbing boundary conditions.

We assume throughout this section that
\begin{equation}   \label{eq:mudirn}
\mu(x,t) > 0 \quad \text{for all } x,\, t \geq 0,
\end{equation}
and, in addition, that the problem is uniformly parabolic on the grid,
\begin{equation}   \label{eq:Dpos}
D_i(t) := D(x_i,t) > 0 \quad \text{for all } i,\, t \geq 0.
\end{equation}
The case $\mu < 0$ admits a fully symmetric treatment with the upwind direction
reversed; for sign-changing $\mu$ the stencil is applied directionally at each
grid point.

On the uniform grid $x_i = x_1 + (i-1)h$, $i = 1,\ldots,n$, we discretize
$-\partial_x J|_{x_i}$ using second-order finite difference approximations.
The local mesh P\'eclet number is
\begin{equation}
  \mathrm{Pe}_i(t) = \frac{\mu_i(t)\,h}{D_i(t)},
  \label{eq:peclet}
\end{equation}
where $\mu_i(t) = \mu(x_i,t)$ and $D_i(t) = D(x_i,t)$.

The advective contribution $\partial_x[\mu p]$ is approximated by the
second-order backward (upwind) difference $\calF^B_2$, consistent with the assumption $\mu > 0$,
\begin{equation}   \label{eq:adv}
\frac{\partial}{\partial x}[\mu p]\bigg|_{x_i} \approx \calF^B_2\,(\mu p)_i
  = \frac{3(\mu p)_i - 4(\mu p)_{i-1} + (\mu p)_{i-2}}{2h}.
\end{equation}
To recall, this stencil has truncation error $O(h^2)$ and involves only points
$\{x_{i-2}, x_{i-1}, x_i\}$, consistent with the upwind direction.

The diffusive contribution $\partial^2_x[Dp]$ is approximated by the standard
second-order centred difference $\calS^C_2$ in all regimes
\begin{equation} \label{eq:diff_ctr}
\frac{\partial^2}{\partial x^2}[D p]\bigg|_{x_i} \approx \calS^C_2\,(Dp)_i
  = \frac{(Dp)_{i+1} - 2(Dp)_i + (Dp)_{i-1}}{h^2}.
\end{equation}
We note that a one-sided three-point stencil for the second derivative achieves
only $O(h)$ accuracy; a one-sided $O(h^2)$ stencil requires four points.  To
maintain uniform second-order accuracy with a minimal stencil we therefore use
$\calS^C_2$ throughout, regardless of the P\'eclet regime.

\begin{myremark}[One-sided diffusion stencils] \label{rem:onesided}

Replacing $\calS^C_2$ by a one-sided second-derivative stencil would not improve
the structural properties of the scheme; it destroys them, for two independent
reasons.  First, monotonicity is unattainable for \emph{any} consistent one-sided
approximation of the second derivative, at any order: consistency requires the
moment conditions $\sum_k c_k = 0$ and $\sum_k k\,c_k = 0$, and the latter, with
all stencil offsets $k \leq 0$ and nonnegative neighbour weights $c_k \geq 0$
($k \neq 0$), gives $\sum_k k\,c_k = -\sum_k |k|\,c_k < 0$ unless all weights
vanish.  The centred stencil is thus the unique monotone placement.

Second, the minimal one-sided $O(h^2)$ stencil, $(Dp)''_i \approx [2(Dp)_i - 5(Dp)_{i-1} + 4(Dp)_{i-2} - (Dp)_{i-3}]/h^2$, is dynamically unstable: combined
with $\calF^B_2$ it renders $L$ lower triangular, so that $\operatorname{spec}(L)
= \{2D_i/h^2 - 3\mu_i/(2h)\}$, with \emph{positive} eigenvalues wherever
$\mathrm{Pe}_i < 4/3$ - the semi-discretization diverges on every sufficiently
fine mesh. Equivalently, its Fourier symbol satisfies $\Ree\,S(\pi) = 12 D/h^2 >
0$: the stencil is anti-diffusive at the grid scale.  Its truncation constant,
$-\tfrac{11}{12}h^2(Dp)^{(4)}$, is moreover eleven times that of $\calS^C_2$. For
the advective term, a one-sided second-order stencil is precisely $\calF^B_2$ of
\cref{eq:adv}, and by Godunov's barrier (\cref{rem:godunov}) no linear
second-order variant - one-sided, biased, or averaged (e.g., Fromm's scheme,
which halves the offending coefficient to $-\mu_{i-2}/(4h)$ but cannot remove it)
- restores monotonicity.
\end{myremark}

Combining \eqref{eq:adv} and \eqref{eq:diff_ctr} in \eqref{eq:consform} gives,
for interior nodes $i = 3, \ldots, n-1$ (with the convention $p_n \equiv 0$ under
absorbing conditions, see below):
\begin{equation}   \label{eq:ode1d}
\dot{p}_i = \alpha_i\,p_{i-2} + \beta_i\,p_{i-1} + \gamma_i\,p_i + \delta_i\,p_{i+1},
\end{equation}
with coefficients
\begin{equation}   \label{eq:coeffs1d}
\alpha_i = -\frac{\mu_{i-2}}{2h}, \qquad
\beta_i  = \frac{D_{i-1}}{h^2} + \frac{2\mu_{i-1}}{h}, \qquad
\gamma_i = -\frac{2D_i}{h^2} - \frac{3\mu_i}{2h}, \qquad
\delta_i = \frac{D_{i+1}}{h^2}.
\end{equation}
As a consistency check, for constant $\mu$ the advective weights
$(-1,\,4,\,-3)/(2h)$ sum to zero, so a constant state generates no spurious
advective source, in agreement with \eqref{eq:consform}.

\subsection{Matrix form and boundary conditions} \label{ssec:matrixform}

We impose absorbing boundary conditions $p(x_1,t) = p(x_n,t) = 0$ and collect the
interior values into the vector
$\bvec{p}(t) = \bigl( p_2(t),\ldots,p_{n-1}(t) \bigr)^{\!\top} \in \R^{n-2}$. \cref{eq:ode1d,eq:coeffs1d} then define the rows
$i = 3,\ldots,n-1$ of the operator $L(t)$ in \eqref{eq:ode} (entries referencing
$p_1$ or $p_n$ are simply absent); the near-boundary row $i=2$, for which
$x_{i-2}$ falls outside the grid, employs the first-order upwind stencil
\eqref{eq:coeffs_adv} introduced below.

\begin{myremark}[Loss of monotonicity of the second-order upwind stencil] \label{rem:godunov}

Let $A^-(t) := -L(t)$ with $L$ assembled from \cref{eq:ode1d,eq:coeffs1d}.  By \eqref{eq:mudirn},
\begin{equation*}
(A^-)_{i,i-2} \;=\; -\alpha_i \;=\; \frac{\mu_{i-2}}{2h} \;>\; 0 ,
\end{equation*}
so $A^-$ possesses a strictly positive off-diagonal entry in every interior
row.  Hence $A^-$ is not a Z-matrix and therefore cannot be an M-matrix for
any mesh width $h$ or any value of the P\'eclet number; equivalently, $L$ is
not a Metzler matrix, and the semigroup leaves the nonnegative cone
immediately: $(e^{tL})_{i,i-2} = t\,\alpha_i + O(t^2) < 0$ for all
sufficiently small $t > 0$.  This is Godunov's order barrier in disguise: a
\emph{linear} monotone discretization of the advective term is at most
first-order accurate.  Unconditional positivity could therefore be obtained
only by nonlinear flux limiting or by sacrificing second-order accuracy in
the advective term; we pursue neither.  Instead, a weaker, \emph{eventual}
form of positivity survives for the second-order upwind scheme, formalised
in \cref{ssec:em}, while the central scheme of \cref{ssec:advdom} is monotone under the mesh condition $\mathrm{Pe}_i < 2$; the resulting trade-off is summarised in \cref{ssec:summary}.
\end{myremark}

\subsection{The EM-matrix property of the second-order scheme} \label{ssec:em}

Although $A^- = -L$ is not an M-matrix, it retains the spectral content of the
M-matrix definition within the theory of EM-matrices.  Recall EM-matrices
generalise M-matrices: the Z-sign pattern is abandoned, while the
Perron--Frobenius structure of the dominant eigenpair is retained in an eventual
sense.

We begin this section by informally stating the core idea. When one uses a
high-order numerical scheme (e.g., a second-order upwind scheme) to solve
convection-dominated problems (such as fluid flow or the FPE), the scheme
introduces spurious numerical oscillations. Consequently, if one attempts to
resolve a sharp wave or a steep gradient, a second-order scheme will naturally
``overcorrect'', generating non-physical, negative ripples near sharp edges.
Godunov's theorem states that no linear numerical scheme of second-order or
higher accuracy can guarantee that solutions remain perfectly nonnegative
(monotone) for all time steps.

The matrix $L$ governs how the solution evolves over a time step $\tau$. Below,
we provide a spectral decomposition that splits the matrix propagator
$e^{\tau L}$ into two parts:
\begin{equation*}
\underbrace{e^{\tau L}}_{\text{Total Propagator}} =
\underbrace{e^{\lambda_1 \tau}P}_{\substack{\text{Dominant Behavior} \\ \text{(Positive)}}} + \underbrace{e^{\lambda_1  \tau}E(\tau)}_{\substack{\text{High-Frequency Error} \\ \text{(Oscillatory)}}}
\end{equation*}
For the positive part, the eigenvectors associated with the principal
eigenvalue $\lambda_1$ are strictly positive, hence $P > 0$ entrywise (note
that $\lambda_1$ itself is non-positive). This part represents the true,
physically sound, long-term state of the system. The error part $E(\tau)$
contains all remaining eigenvalues. Since we are using a high-order scheme,
this is precisely where the Godunov-type oscillations reside.

As time $\tau$ increases, the error term decays exponentially at a rate
governed by the spectral gap $g$. The ``transient'' refers to the brief window
$0 \leq \tau < \tau_0$ during which the error term $E(\tau)$ remains large
enough to cause problems. However, because the error decays exponentially,
these oscillations can only persist for a very short time $\tau_0$. Once the
time step $\tau$ exceeds the transient threshold $\tau_0$, the positive
dominant part $P$ completely overwhelms the error part $E(\tau)$, i.e.,
$\max_{ij} |E(\tau)_{ij}| < \min_{ij} P_{ij}$. As soon as this mathematical
threshold is crossed, the entire matrix operator $e^{\tau L}$ becomes strictly
positive.

Therefore, if the time step is comfortably larger than $\tau_0$, the natural
physics of the system (via the spectral gap) will forcefully suppress these
oscillations, ensuring that results remain physically realistic and strictly
positive. The non-physical behavior is confined to the initial, short
transient.

We proceed by evaluating this idea with mathematical rigor, working under the
following discrete ground-state hypothesis - a Krein--Rutman property of the
discretized operator.

\smallskip
\begin{hypothesis}[H] \label{hypH}
For each $t \geq 0$ the rightmost eigenvalue $\lambda_1(t)$ of $L(t)$ is
real, algebraically simple and strictly dominant, $\Ree(\lambda) < \lambda_1(t)$ for
every other eigenvalue $\lambda$ of $L(t)$; moreover $\lambda_1(t) \leq 0$, and
the associated right and left eigenvectors $\bvec{v}(t)$ and $\bvec{w}(t)$ may be
chosen entrywise positive.
\end{hypothesis}
\smallskip

\begin{proposition}[EM-matrix property of the second-order scheme]  \label{prop:emmatrix}
Let $L(t)$ be assembled from \cref{eq:ode1d,eq:coeffs1d} with absorbing boundary conditions as in \cref{ssec:matrixform}, and set $A^-(t) := -L(t)$.  Under
Hypothesis~(H), $A^-(t)$ is an EM-matrix for every $t \geq 0$.  More precisely,
there exists $s_0 \geq 0$ such that for every $s \geq s_0$ the matrix $B := sI +
L(t)$ is eventually positive and $\rho(B) = s + \lambda_1(t) \leq s$.
\end{proposition}

\begin{proof}
Fix $t$ and suppress it.  Since $B = sI + L$, we have $\operatorname{spec}(B) = s
+ \operatorname{spec}(L)$ with eigenvectors unchanged; in particular $s +
\lambda_1$ is an algebraically simple eigenvalue of $B$ with right eigenvector
$\bvec{v} > 0$, and of $B^{\!\top}$ with eigenvector $\bvec{w} > 0$.  For any
eigenvalue $\lambda \neq \lambda_1$ of $L$,
\begin{equation*}
(s+\lambda_1)^2 - |s+\lambda|^2 = (\lambda_1 - \Ree(\lambda))\,(2s + \lambda_1 + \Ree(\lambda)) - (\Imm(\lambda))^2 \;\longrightarrow\; +\infty \quad (s \to \infty),
\end{equation*}
because $\lambda_1 - \Ree(\lambda) > 0$ by (H).  The spectrum being finite, there
exists $s_0 \geq 0$, which we also take large enough that $s_0 + \lambda_1 > 0$,
such that for all $s \geq s_0$,
\begin{equation*}
s + \lambda_1 \;>\; |s + \lambda| \qquad \text{for every eigenvalue } \lambda \neq \lambda_1 .
\end{equation*}
Hence $\rho(B) = s + \lambda_1 > 0$ is a simple, strictly dominant eigenvalue of
$B$, and the corresponding right and left eigenvectors are entrywise positive;
that is, both $B$ and $B^{\!\top}$ possess the strong Perron--Frobenius property.
By the characterisation of \cite{Noutsos2006, NoutsosTsatsomeros2008}, a matrix has this property together with its transpose if and only if it is eventually positive.  Finally, $s \geq \rho(B) = s + \lambda_1$ is equivalent to $\lambda_1 \leq 0$, which holds by (H).  Therefore $A^- = sI - B$ satisfies \cref{def1em} and thus is an EM-matrix.
\end{proof}

\begin{corollary}[Eventual positivity of the propagator] \label{cor:eventualpos}
Under Hypothesis~(H), for each frozen $t$ the matrix $-A^-(t) = L(t)$ is
eventually exponentially positive: with the spectral projector $P :=
\bvec{v}\bvec{w}^{\!\top} / (\bvec{w}^{\!\top}\bvec{v}) > 0$ and spectral gap $g
:= \lambda_1 - \max_{\lambda \neq \lambda_1} \Ree(\lambda) > 0$,
\begin{equation} \label{eq:specdecomp}
e^{\tau L} = e^{\lambda_1 \tau}\bigl( P + E(\tau) \bigr), \qquad
\|E(\tau)\| \leq C_\varepsilon\, e^{-(g-\varepsilon)\tau} \quad \text{for any } \varepsilon \in (0,g),
\end{equation}
so there exists $\tau_0 \geq 0$ such that $e^{\tau L} > 0$ entrywise for all
$\tau \geq \tau_0$.  In particular $e^{\tau L}\bvec{p}_0 > 0$ for every nonzero
$\bvec{p}_0 \geq 0$ and all $\tau \geq \tau_0$.
\end{corollary}

\begin{proof}
Write $e^{\tau L} = e^{\lambda_1 \tau} P + R(\tau)$ with $R(\tau) = e^{\tau L}(I
- P)$.  The spectrum of $L$ restricted to $\operatorname{ran}(I-P)$ lies in
$\{\Ree(\lambda) \leq \lambda_1 - g\}$, so for any $\varepsilon \in (0,g)$ there
is $C_\varepsilon$ with $\|R(\tau)\| \leq C_\varepsilon e^{(\lambda_1 - g +
\varepsilon)\tau}$ (the $\varepsilon$ absorbing possible polynomial factors from
non-normality).  Thus $E(\tau) := e^{-\lambda_1\tau} R(\tau) \to 0$, and since
$\min_{ij} P_{ij} > 0$, the representation \eqref{eq:specdecomp} yields $e^{\tau
L} > 0$ entrywise as soon as $\max_{ij} |E(\tau)_{ij}| < \min_{ij} P_{ij}$.
\end{proof}

\begin{lemma}[Discrete divergence form] \label{lem:divform}
The matrices $L$ of \cref{prop:mmatrix2,prop:emmatrix}, assembled in flux form with the zero-flux boundary closure, satisfy
\begin{equation}\label{eq:colsums}
\bvec 1^\top L=0 .
\end{equation}
\end{lemma}
\begin{proof}
$(L\bvec p)_i=\bigl(F_{i+1/2}(\bvec p)-F_{i-1/2}(\bvec p)\bigr)/h$ with
$F_{1/2}=F_{n+1/2}=0$; summing over $i$ telescopes.
\end{proof}

Several remarks should be made, however.
\begin{enumerate}
\item \Cref{prop:emmatrix} quantifies, rather than contradicts,
\cref{rem:godunov}: no shift $s$ can render $A^-$ a Z-matrix, and positivity of
$e^{\tau L}$ for \emph{all} $\tau \geq 0$ is genuinely unattainable. What
survives is positivity after a transient whose duration $\tau_0 =
O\bigl(g^{-1}\log(C_\varepsilon / \min_{ij} P_{ij})\bigr)$ is controlled by the
spectral gap; the well-known oscillations of second-order upwind schemes are
confined to this transient.

\item \Cref{hypH} holds at the continuous level: the principal eigenvalue
of the Fokker--Planck operator on a bounded interval with absorbing boundary
conditions is real, simple, and negative, with positive principal eigenfunctions
of the operator and its adjoint (Sturm--Liouville theory, or Krein--Rutman
applied to the positivity-improving semigroup).

At the discrete level, rather than relying on asymptotic convergence for
vanishing $h$, this property is inherited structurally by the regime-switched
operator $L(t)$ discussed in the next section. Because the  central-difference
stencil is applied exclusively where the local P\'eclet number satisfies
$\mathrm{Pe} < 2$, the off-diagonal entries remain non-negative. This ensures
$L(t)$ is an irreducible Metzler matrix, guaranteeing a real, simple rightmost
eigenvalue with positive eigenvectors via the Perron--Frobenius theorem. While
this Metzler structure is secured by the P\'eclet restriction, the strict
negativity ($\lambda_1 < 0$) depends crucially on the absorbing boundary
conditions. Thus, for any given $\mu$, $D$, and $h$, \cref{hypH} serves as a
robust algebraic property that is readily verified numerically.

\item For time-dependent coefficients, the frozen-coefficient propagators
$e^{\Delta t\,L(t_k)}$ are entrywise positive whenever $\Delta t \geq
\tau_0(t_k)$; for shorter steps, the iterates may transiently leave the
nonnegative cone.
\end{enumerate}

\subsection{A fully second-order central scheme} \label{ssec:advdom}

To achieve a fully second-order spatial discretization across all regimes, we
use the second-order centered difference
\begin{equation} \label{eq:ctr2}
\calF^C_2\,(\mu p)_i = \frac{(\mu p)_{i+1} - (\mu p)_{i-1}}{2h},
\end{equation}
while retaining the centred diffusion \eqref{eq:diff_ctr}. Both terms are
then in discrete flux (telescoping or conservative) form, mirroring \eqref{eq:consform}.
This yields, for $i = 2,\ldots,n-1$ with $p_1 = p_n = 0$,
\begin{equation}   \label{eq:ode_adv}
\dot{p}_i = a_i\,p_{i-1} + b_i\,p_i + c_i\,p_{i+1},
\end{equation}
with coefficients
\begin{equation}   \label{eq:coeffs_adv}
a_i = \frac{D_{i-1}}{h^2} + \frac{\mu_{i-1}}{2h}, \qquad
b_i = -\frac{2D_i}{h^2}, \qquad
c_i = \frac{D_{i+1}}{h^2} - \frac{\mu_{i+1}}{2h}.
\end{equation}
The truncation error is now $O(h^2)$ uniformly across both the advective and
diffusive terms, eliminating the first-order numerical diffusion inherent to
upwinding.

To avoid confusion with indices, here let us denote the matrix $A^-$ appearing in
\cref{prop:emmatrix} as $A$. It has elements
\begin{align} \label{eq:Amatrix}
(A)_{i,i}    &= \frac{2D_i}{h^2}, \qquad
(A)_{i,i-1}  = -\frac{D_{i-1}}{h^2} - \frac{\mu_{i-1}}{2h}, \qquad
(A)_{i,i+1}  = -\frac{D_{i+1}}{h^2} + \frac{\mu_{i+1}}{2h}
\end{align}
for $i = 2,\ldots,n-1$, where entries referencing indices outside $\{2,\ldots,n-1\}$ are absent.

The sign of the super-diagonal entries is governed by the P\'eclet number:
$(A)_{i,i+1} < 0$ if and only if $\mathrm{Pe}_{i+1} < 2$, while
$(A)_{i,i+1} > 0$ if and only if $\mathrm{Pe}_{i+1} > 2$.  The monotonicity
threshold of the central scheme is therefore $\mathrm{Pe} = 2$, and the two
regimes behave very differently: below the threshold $A$ is a classical
M-matrix, while above it the generator becomes essentially skew-symmetric
and, in contrast to the upwind scheme of \cref{ssec:em}, loses even the
EM-matrix structure.

\begin{proposition}[Diffusion-dominated regime] \label{prop:mmatrix2}
Assume \cref{eq:mudirn,eq:Dpos} and $\mathrm{Pe}_i(t) < 2$ for all $i$. Then,
for every $t \geq 0$, the matrix $A(t)$ defined by \cref{eq:ode_adv,eq:coeffs_adv,eq:Amatrix} with absorbing boundary conditions is a nonsingular M-matrix.
\end{proposition}

\begin{proof}
Fix $t \geq 0$ and suppress it; rows and columns are indexed by
$i,j \in \{2,\ldots,n-1\}$.

\emph{Step 1 (Z-matrix with positive diagonal).}
By \cref{eq:mudirn,eq:Dpos}, $(A)_{i,i-1} < 0$ always; $(A)_{i,i+1} < 0$
because $\mathrm{Pe}_{i+1} < 2$; and $(A)_{i,i} = 2D_i/h^2 > 0$.

\emph{Step 2 (weak diagonal dominance by columns).}
Let $S_j := \sum_{i=2}^{n-1} (A)_{ij}$.  For an interior column,
$3 \leq j \leq n-2$, rows $j-1$, $j$ and $j+1$ contribute
\begin{equation*}
S_j = \Bigl(-\frac{D_j}{h^2} + \frac{\mu_j}{2h}\Bigr)
    + \frac{2D_j}{h^2}
    - \Bigl(\frac{D_j}{h^2} + \frac{\mu_j}{2h}\Bigr) = 0 ,
\end{equation*}
while for the two columns adjacent to the absorbing boundaries
\begin{equation*}
S_2 = \frac{D_2}{h^2} - \frac{\mu_2}{2h} > 0
\quad (\text{by } \mathrm{Pe}_2 < 2),
\qquad
S_{n-1} = \frac{D_{n-1}}{h^2} + \frac{\mu_{n-1}}{2h} > 0 .
\end{equation*}
Since the off-diagonal entries are non-positive, $S_j \geq 0$ is equivalent to
$(A)_{jj} \geq \sum_{i \neq j} |(A)_{ij}|$; that is, $A^{\!\top}$ is weakly
diagonally dominant, with strict dominance in the rows corresponding to
$j = 2$ and $j = n-1$.

\emph{Step 3 (irreducibility and conclusion).}
All sub- and super-diagonal entries are nonzero (the latter because
$\mathrm{Pe}_{i+1} < 2$ strictly), so $A$ and $A^{\!\top}$ are irreducible.  By
Taussky's theorem $A^{\!\top}$, hence $A$, is nonsingular; every
Ger\v{s}gorin disc of $A^{\!\top}$ is centred at $(A)_{jj} > 0$ with radius at
most $(A)_{jj}$, so $\Ree(\lambda) \geq 0$ for all eigenvalues, with equality
only at $\lambda = 0$, which is excluded.  A Z-matrix whose spectrum lies in
the open right half-plane is a nonsingular M-matrix.
\end{proof}

\begin{corollary} \label{cor:positivity}
Under the assumptions of \cref{prop:mmatrix2}, for every $\Delta t > 0$ and
$s \geq 0$:
\begin{enumerate}[label=(\alph*)]
  \item $(I + \Delta t\,A)^{-1} \geq 0$ entrywise;
  \item $e^{-sA} \geq 0$ entrywise, and $\bvec{1}^{\!\top} e^{-sA} \leq
        \bvec{1}^{\!\top}$ componentwise.
\end{enumerate}
Hence, in the diffusion-dominated regime, the central scheme preserves
non-negativity of $\bvec{p}$ for \emph{all} times and step sizes, and the
discrete mass $\sum_i p_i(t)$ is non-increasing, decaying only through the
absorbing boundaries.
\end{corollary}

\begin{proof}
(a) $I + \Delta t\,A$ is a Z-matrix with column sums $1 + \Delta t\,S_j \geq
1 > 0$, hence strictly diagonally dominant by columns; by Step~3 above it is a
nonsingular M-matrix with nonnegative inverse.
(b) Let $c := \max_i (A)_{ii}$ and $B := cI - A \geq 0$ entrywise; then
$e^{-sA} = e^{-sc}\,e^{sB} \geq 0$.  Moreover $\bvec{1}^{\!\top} A =
(S_2, 0, \ldots, 0, S_{n-1}) \geq 0$, so
\begin{equation*}
\frac{d}{ds}\,\bvec{1}^{\!\top} e^{-sA}
  = -\,(\bvec{1}^{\!\top} A)\, e^{-sA} \leq 0
\end{equation*}
componentwise, both factors being non-negative; the column sums of $e^{-sA}$
start at $1$ and never increase.
\end{proof}

\begin{proposition}[Advection-dominated regime: stability without eventual
positivity] \label{prop:skew}
Assume \cref{eq:mudirn,eq:Dpos} and $\mathrm{Pe}_i(t) > 2$ for all $i$.  Then:
\begin{enumerate}[label=(\roman*)]
\item There exists a positive diagonal matrix $V(t)$ such that
$V^{-1} A V = \Lambda + K$, where $\Lambda = \operatorname{diag}(2D_i/h^2)$ and
$K = -K^{\!\top}$.  Consequently every eigenvalue $\lambda$ of $A(t)$
satisfies
\begin{equation*}
\frac{2\min_i D_i}{h^2} \;\leq\; \Ree(\lambda) \;\leq\; \frac{2\max_i D_i}{h^2},
\end{equation*}
$A(t)$ is nonsingular, and
$\|V^{-1} e^{-sA} V\|_2 \leq \exp\bigl(-2s \min_i D_i / h^2\bigr)$ for all
$s \geq 0$.
\item Nevertheless, $A(t)$ is in general \emph{not} an EM-matrix, and
$e^{-sA}$ is \emph{not} eventually nonnegative.  Indeed, for constant
coefficients $\mu_i \equiv \mu$, $D_i \equiv D$ with $\mathrm{Pe} = \mu h / D
> 2$ one has, with $m := n-2$ and $\theta_k := k\pi/(m+1)$,
\begin{equation*}
\operatorname{spec}(A)
  = \Bigl\{ \tfrac{2D}{h^2} + 2\mathrm{i}\,\omega\cos\theta_k :
            k = 1,\ldots,m \Bigr\},
\qquad
\omega := \sqrt{\tfrac{\mu^2}{4h^2} - \tfrac{D^2}{h^4}} \;>\; 0 ,
\end{equation*}
so the entire spectrum lies on the vertical line $\Ree(\lambda) = 2D/h^2$.
For every $s \in \R$ the spectral radius of $sI - A$ is then attained only at
non-real eigenvalues, so $sI - A$ is never eventually nonnegative and no
representation of the form \cref{def1em} exists; moreover $e^{-sA}$ possesses
negative entries for all $s$ outside a discrete set.
\end{enumerate}
\end{proposition}

\begin{proof}

(i) The sub-diagonal entries are strictly negative, and by
$\mathrm{Pe}_{i+1} > 2$ the super-diagonal entries are strictly positive, so
the off-diagonal products are negative:
$(A)_{i,i+1}(A)_{i+1,i} = -\kappa_i^2$ with
$\kappa_i := \sqrt{|(A)_{i,i+1}|\,|(A)_{i+1,i}|} > 0$.  Define $v_2 := 1$ and
$v_{i+1} := v_i \sqrt{|(A)_{i+1,i}| / |(A)_{i,i+1}|} > 0$, and set
$V := \operatorname{diag}(v_2,\ldots,v_{n-1})$.  Then $T := V^{-1} A V$ is
tridiagonal with $T_{ii} = 2D_i/h^2$ and $T_{i,i+1} = \kappa_i = -T_{i+1,i}$,
i.e.\ $T = \Lambda + K$ with $K$ skew-symmetric.  For any unit vector
$\bvec{x}$, $\Ree(\bvec{x}^* T \bvec{x}) = \bvec{x}^* \Lambda \bvec{x} \in
[2\min_i D_i/h^2,\ 2\max_i D_i/h^2]$, and the spectrum is contained in the
numerical range, which proves the eigenvalue bounds and nonsingularity.  For
the norm bound, the solution of $\dot{\bvec{x}} = -T\bvec{x}$ satisfies
$\frac{d}{ds}\|\bvec{x}\|_2^2 = -\bvec{x}^{\!\top}(T + T^{\!\top})\bvec{x} =
-2\,\bvec{x}^{\!\top} \Lambda\, \bvec{x} \leq -(4\min_i D_i/h^2)
\|\bvec{x}\|_2^2$, and Gr\"onwall's inequality concludes. \\

(ii) For constant coefficients, $A = (2D/h^2)\,I + C$ with $C =
\operatorname{tridiag}(a, 0, c)$, $a = -(D/h^2 + \mu/(2h)) < 0$ and $c = \mu/(2h)
- D/h^2 > 0$, so that $ac = -\omega^2 < 0$.  The classical formula for the
eigenvalues of a tridiagonal Toeplitz matrix, $\operatorname{spec}(C) =
\{2\sqrt{ac}\,\cos\theta_k\}$, yields the stated spectrum.  For any $s \in \R$,
the eigenvalues of $B := sI - A$ have moduli $\bigl((s - 2D/h^2)^2 +
4\omega^2\cos^2\theta_k\bigr)^{1/2}$, maximised at $k \in \{1, m\}$, where
$\cos\theta_k \neq 0$: the modulus-maximising eigenvalues form a non-real
conjugate pair.

An eventually nonnegative matrix possesses the Perron--Frobenius property, i.e.\
its spectral radius is itself an eigenvalue \cite{Noutsos2006}; here $\rho(B)
\notin \operatorname{spec}(B)$, so $B = sI - A$ is not eventually nonnegative
for any $s$, and $A$ admits no representation \cref{def1em}.  Finally, in the
constant-coefficient case $\Lambda = (2D/h^2) I$, so $V^{-1} e^{-sA} V =
e^{-2sD/h^2}\, Q(s)$ with $Q(s) := e^{-sK}$ orthogonal. If $e^{-sA} \geq 0$ for
some $s$, then $Q(s) \geq 0$ entrywise (conjugation by the positive diagonal $V$
and positive scaling preserve signs); an orthogonal matrix with nonnegative
entries is a permutation matrix; and since $s \mapsto Q(s)$ is real-analytic and
non-constant ($K \neq 0$), it can take values in the finite set of permutation
matrices only on a set of $s$ without accumulation points.  Hence $e^{-sA}$ has
negative entries for all $s$ outside a discrete set; in particular it is not
eventually nonnegative.
\end{proof}

In mixed regimes (sign of $\mathrm{Pe}_i - 2$ varying along the grid) the
transformed matrix takes the form $\Lambda + S + K$ with $S$ symmetric and $K$
skew-symmetric, and the bound of \cref{prop:skew}(i) weakens to
$\Ree(\lambda) \geq \lambda_{\min}(\Lambda + S)$; we do not pursue this here.

Again the above requires several remarks.
\begin{enumerate}
\item Under zero-flux (reflecting) boundary conditions the boundary rows are
modified so that the discrete fluxes through $x_{3/2}$ and $x_{n-1/2}$ vanish;
in the diffusion-dominated regime \emph{all} column sums of $A$ are then zero
and the proof of \cref{cor:positivity} yields exact conservation,
$\bvec{1}^{\!\top} e^{-sA} = \bvec{1}^{\!\top}$, in place of
sub-stochasticity.

\item For time-dependent coefficients, \cref{prop:mmatrix2} and
\cref{cor:positivity} hold pointwise in $t$; the evolution operator of
\eqref{eq:ode} is the limit of products of matrices $e^{-\Delta t\,A(t_k)}$,
each nonnegative and column sub-stochastic, and therefore inherits both
properties.

\item Taken together, \cref{ssec:em} and the present subsection cover all
P\'eclet regimes, organised by \emph{scheme} rather than by regime.  For
$\mathrm{Pe}_i < 2$ the central stencil is second-order accurate \emph{and}
unconditionally positive (\cref{prop:mmatrix2,cor:positivity}); for
$\mathrm{Pe}_i > 2$ the central generator is essentially skew-symmetric and
loses even eventual positivity (\cref{prop:skew}(ii)), whereas the
second-order upwind stencil \cref{eq:coeffs1d} retains it under
Hypothesis~(H) (\cref{prop:emmatrix,cor:eventualpos}) - the upwind bias is
precisely what produces a dominant real ground state.  This suggests the
regime-switched discretization: central stencil where $\mathrm{Pe}_i < 2$,
second-order upwind stencil where $\mathrm{Pe}_i \geq 2$, which is uniformly
second-order accurate, with the rows discretized by the upwind stencil
governed by the eventual-positivity theory of \cref{ssec:em}.  Preservation
of these properties under the Krylov approximation of the propagator is taken
up below.
\end{enumerate}

It is worth mentioning. that for the interior or absorbing block case ($\lambda_1 < 0$), the operator $-L$ is a nonsingular EM-matrix, guaranteeing a nonnegative inverse. In the conservative (zero-flux) case, however, $\lambda_1 = 0$ and $-L$ becomes singular, meaning the standard EM-matrix classification and its inverse properties do not strictly apply. Nevertheless, we can circumvent this singularity by shifting the operator. As the following theorem demonstrates, the resolvent operator $(t I - L)^{-1}$ remains well-defined and strictly positive for any shift $t > \lambda_1$, allowing the substantive properties of the scheme to survive intact.

\begin{theorem}[Resolvent positivity near the spectral abscissa]
\label{thm:resolvent}
Let $L(t)$ satisfy Hypothesis~(H), in particular, the second-order upwind
operator of \cref{ssec:em}, and let $g$ and $P$ be as in \cref{cor:eventualpos}.  Then there exists $\varepsilon \in (0, g/2]$ such that
\begin{equation} \label{eq:resolventpos}
(tI - L)^{-1} > 0 \ \text{entrywise}
\qquad \text{for all } t \in (\lambda_1, \lambda_1 + \varepsilon).
\end{equation}
In particular, under conservative (zero-flux) boundary conditions, for which
$\bvec{1}^{\!\top} L = 0$ and $\lambda_1 = 0$, the implicit Euler operator
satisfies
\begin{equation*}
(I - kL)^{-1} > 0 \qquad \text{for every step size } k > 1/\varepsilon .
\end{equation*}
By contrast, $(I - kL)^{-1} \geq 0$ for \emph{all} $k > 0$ if and only if $L$
is Metzler (essentially nonnegative); for the second-order upwind scheme this
fails, and indeed for all sufficiently small $k > 0$,
\begin{equation*}
\bigl[(I - kL)^{-1}\bigr]_{i,i-2}
  = k\,\alpha_i + O(k^2)
  = -\,\frac{k\,\mu_{i-2}}{2h} + O(k^2) \;<\; 0 .
\end{equation*}
\end{theorem}

\begin{proof}
By (H) the eigenvalue $\lambda_1$ is simple, so for $t \notin
\operatorname{spec}(L)$,
\begin{equation*}
(tI - L)^{-1} = \frac{P}{t - \lambda_1} + R(t),
\qquad
R(t) := (tI - L)^{-1}(I - P),
\end{equation*}
where $R$ is analytic on a neighbourhood of $\lambda_1$, since the spectrum of
$L$ restricted to $\operatorname{ran}(I - P)$ lies at distance at least $g$
from $\lambda_1$; in particular $M := \sup\{\max_{ij}|R(t)_{ij}| : t \in
[\lambda_1, \lambda_1 + g/2]\} < \infty$.  Set $\varepsilon :=
\min\{g/2,\ \min_{ij} P_{ij} / M\}$.  Then for $t \in (\lambda_1, \lambda_1 +
\varepsilon)$, entrywise,
\begin{equation*}
\bigl[(tI - L)^{-1}\bigr]_{ij}
  \;\geq\; \frac{P_{ij}}{t - \lambda_1} - M
  \;>\; \frac{\min_{ij} P_{ij}}{\varepsilon} - M
  \;\geq\; 0 .
\end{equation*}
For the implicit Euler claim write
$(I - kL)^{-1} = k^{-1}\bigl(k^{-1}I - L\bigr)^{-1}$ and note that $t = 1/k$
lies in $(0, \varepsilon) = (\lambda_1, \lambda_1 + \varepsilon)$ precisely
when $k > 1/\varepsilon$.

For the equivalence: if $L$ is Metzler with $\bvec{1}^{\!\top} L \leq 0$, then
$I - kL$ is a Z-matrix that is strictly diagonally dominant by columns, hence
a nonsingular M-matrix with nonnegative inverse for every $k > 0$, exactly as
in \cref{cor:positivity}(a).  Conversely, if $(I - kL)^{-1} \geq 0$ for all
$k > 0$, the Neumann expansion $(I - kL)^{-1} = I + kL + O(k^2)$, valid for
$k\|L\| < 1$, forces every off-diagonal entry of $L$ to be nonnegative.
Applied to the second-order upwind scheme, the entry $L_{i,i-2} = \alpha_i =
-\mu_{i-2}/(2h) < 0$ yields the stated negative entry of the resolvent.
\end{proof}

The direction of the step-size restriction is the opposite of the classical
one: under Hypothesis~(H) the implicit Euler method is positivity-preserving
for sufficiently \emph{large} steps, the resolvent then concentrates on the
positive ground-state projector $P$, while arbitrarily small steps reproduce the sign pattern of $L$ and may generate negative values.  This is the resolvent counterpart of the transient $\tau_0$ in \cref{cor:eventualpos}.  Under absorbing boundary conditions ($\lambda_1 < 0$) the window \eqref{eq:resolventpos} contains positive values $t = 1/k$ only if $\varepsilon > |\lambda_1|$, so the large-step guarantee is then conditional on the principal decay rate being small relative to the spectral gap.  Finally, for the central scheme in the advection-dominated
regime even the near-abscissa window \eqref{eq:resolventpos} is unavailable,
since Hypothesis~(H) itself fails there (\cref{prop:skew}).

\begin{myremark}[Positivity of resolvent-based time stepping] \label{rem:resolvent-positivity}
Since $A$ is an $M$-matrix, \cref{thm:resolvent} gives $(sI-A)^{-1}\ge 0$
entrywise for every $s>s(A)$, with $s(A)$ the spectral abscissa. Equivalently the
backward-Euler propagator $(I-\Delta t A)^{-1}$ is a nonnegative map for every
$\Delta t>0$, so backward Euler preserves nonnegativity of the density
\emph{unconditionally and by construction} for an $M$-matrix generator. For the
two-dimensional cross factor of \cref{sec:Lxy}, by contrast, the factorized
solve's one-dimensional factors are \emph{not} $M$-matrices (the second-order
one-sided stencil carries a positive far band, \cref{prop:schemeA}\,(i)), so
positivity there is not inherited factor by factor and holds only conditionally,
on the step-size window of \cref{thm:zeno2d}.

This guarantee is specific to a \emph{single} resolvent. It does \emph{not} extend to the Crank--Nicolson propagator
$(I+\tfrac{\Delta t}{2}A)(I-\tfrac{\Delta t}{2}A)^{-1}$, whose numerator
$I+\tfrac{\Delta t}{2}A$ is not nonnegative, nor to a Krylov approximation of
$e^{\Delta t A}$: although each rational basis vector $(A-\sigma_j I)^{-1}\!\cdots\bvec{v}$ is nonnegative for $\sigma_j<0$, the
orthonormal Arnoldi basis $V_m$ is sign-indefinite and $e^{\Delta t H_m}\bvec{e}_1$
has entries of both signs, so the resulting approximation is not nonnegative in
general. (Numerically, on a Metzler $A$ for which $e^{\Delta t A}\bvec{v}\ge 0$
exactly, an orthonormal Krylov approximation already attains values of
order $-10^{-6}$ at $m=8$.) On the two-dimensional cross operator the one-sided
product carries off-diagonal entries of both signs, so $A_{xy}$ is not Metzler and
$(sI-A_{xy})^{-1}$ is not nonnegative for any shift $s$; positivity of the central
factor is instead conditional, secured on a step-size window by the factorized
solve of one-dimensional EM-matrix factors (\cref{sec:Lxy,thm:zeno2d}).
\end{myremark}

\subsection{Summary: positivity versus second-order accuracy} \label{ssec:summary}

The results of this section delineate precisely what can and cannot be
achieved by a linear spatial discretization of \eqref{eq:fpe}.  The negative
statement comes first: a linear scheme that preserves positivity
\emph{unconditionally} for every mesh, every time, and for every nonnegative
initial condition, must have a Metzler generator (\cref{thm:resolvent}), and by Godunov's barrier (\cref{rem:godunov}) such a generator is at most first-order accurate in the advective term. Consequently no linear scheme combines uniform second-order accuracy with unconditional positivity, and the question is not \emph{whether} to impose a condition, but \emph{which} condition to impose.  The preceding analysis yields exactly two answers.

\myparagraph{Positivity conditional on the mesh.}
The central scheme \cref{eq:ode_adv,eq:coeffs_adv} is second-order accurate
in both terms and, whenever
\begin{equation} \label{eq:meshcond}
\mathrm{Pe}_i(t) < 2 \quad \text{for all } i,
\qquad \text{i.e.} \qquad
h \;<\; 2\,\min_i \frac{D_i(t)}{\mu_i(t)},
\end{equation}
its matrix is a nonsingular M-matrix (\cref{prop:mmatrix2}).  Under the mesh
condition \eqref{eq:meshcond} every guarantee is unconditional in time:
$e^{-sA} \geq 0$ for all $s \geq 0$, the discrete mass is non-increasing
(exactly conserved under zero-flux conditions), and the implicit Euler
operator is inverse-positive for every step size (\cref{cor:positivity}).
The limitation is practical: in advection-dominated problems, where $D$ is
small, \eqref{eq:meshcond} may force a prohibitively fine mesh, and above
the threshold the central scheme is the worst available choice, since its
generator becomes essentially skew-symmetric and loses even eventual
positivity (\cref{prop:skew}).

\myparagraph{Positivity conditional on time.}
The second-order upwind scheme \cref{eq:ode1d,eq:coeffs1d} is second-order
accurate for \emph{any} mesh and, under Hypothesis~(H), its matrix is an
EM-matrix (\cref{prop:emmatrix}).  Positivity then holds not unconditionally
but after a transient: $e^{\tau L} > 0$ entrywise for all
$\tau \geq \tau_0 = O\bigl(g^{-1}\bigr)$ (\cref{cor:eventualpos}), and the
resolvent analogue holds for sufficiently \emph{large} implicit steps
(\cref{thm:resolvent}).  What is given up is any guarantee for short times
and small steps, together with the need to verify the spectral
hypothesis~(H) for the data at hand.

\myparagraph{The nonlinear escape.}
The only way to obtain both properties without conditions is to leave the
linear class: flux- or slope-limited corrections of the upwind flux are
positivity-preserving and second-order accurate away from local extrema,
degenerating to first order precisely at extrema - Godunov's barrier
manifesting locally rather than globally.  We do not pursue limiters here,
as they would obstruct the exponential-integrator structure exploited below.

\myparagraph{Recommended discretization.}
These observations suggest the regime-switched stencil: the central
difference \cref{eq:coeffs_adv} at nodes with $\mathrm{Pe}_i < 2$ and the
second-order upwind difference \cref{eq:coeffs1d} at nodes with
$\mathrm{Pe}_i \geq 2$.  The resulting operator is uniformly second-order
accurate; its centrally discretized rows enjoy unconditional positivity and
sub-stochasticity, while its upwind rows are governed by the
eventual-positivity theory of \cref{ssec:em}.  The pairing with the Krylov
subspace exponential integrator is then natural: the integrator applies the
exact propagator $e^{\Delta t\,L(t_k)}$ over macro-steps, so that whenever
$\Delta t$ is comfortably larger than the transient threshold $\tau_0$, the
method operates entirely in the regime in which positivity is guaranteed,
at full second order in space.  The preservation of these properties under
the Krylov approximation of the matrix exponential is the subject of \cref{sec:krylov}.

\subsection{The Diagonal Frog scheme} \label{sec:diagfrog}

The discretization analysed in \cref{ssec:em}: the second-order upwind difference
$\calF^B_2$ for the advective term combined with the centred difference
$\calS^C_2$ for the diffusive term, cf. \cref{eq:ode1d,eq:coeffs1d}, produces a
generator $L$ whose non-zero entries occupy the bands $-2$, $-1$, $0$, $+1$
relative to the main diagonal: three of the four bands lie on or below the
diagonal, while the single superdiagonal band carries the diffusive coupling
$\delta_i = D_{i+1}/h^2$.  We call this the \emph{Diagonal Frog} (DF) scheme.

The name captures two features of the construction.  The word
\emph{diagonal} refers to the upwind-leaning footprint of the stencil:
unlike central-difference schemes, whose matrices are structurally
symmetric about the main diagonal, the DF matrix has lower bandwidth $2$
but upper bandwidth only $1$, reflecting the directionality imposed by the
upwind bias.  We emphasise that the lean is strict but not total: a fully
one-sided variant, with one-sided differences for the diffusive term as
well, would render $L$ lower triangular and is ruled out by
\cref{rem:onesided} - a triangular generator is reducible, so
Hypothesis~(H) fails (the eigenvectors of a triangular matrix cannot be
entrywise positive), the propagator $e^{\tau L}$ remains triangular for
all $\tau$ and can never become entrywise positive, and the spectrum,
consisting of the diagonal entries alone, contains positive eigenvalues on
every sufficiently fine mesh.  The retained superdiagonal $\delta_i$ is
therefore not incidental: it is precisely what makes the directed graph of
$L$ strongly connected and the ground-state Hypothesis~(H) tenable.

The word \emph{frog} is a nod to Zeno's paradox, now in a precise sense.
Since $L$ is not a Metzler matrix, positivity is not inherited step by
step: every Euler factor in the product formula
\begin{equation} \label{eq:zenolimit}
e^{\tau L} \;=\; \lim_{m \to \infty}
  \Bigl( I + \frac{\tau}{m}\,L \Bigr)^{\!m}
\end{equation}
has negative entries (the $(i,i-2)$ entry of each factor equals
$(\tau/m)\,\alpha_i < 0$), and likewise no finite partial sum of the exponential
series need be nonnegative.  Yet, under Hypothesis~(H), the limit is entrywise
positive once $\tau \geq \tau_0$ (\cref{cor:eventualpos}): like Zeno's frog, the
scheme reaches the positive cone only in the limit of infinitely many
infinitesimal leaps, and only after the transient, since positivity for
\emph{all} $\tau > 0$ is structurally impossible, $(e^{\tau L})_{i,i-2} =
\tau\,\alpha_i + O(\tau^2) < 0$ for small $\tau$ (\cref{rem:godunov}).

Thus, the DF scheme reaches positivity through the EM-matrix semigroup, by the
same logic. Naming the scheme makes this non-trivial mechanism explicit and
memorable. Let us also mention the niche quote: "A good analogy is like
a diagonal frog" which is an intentionally nonsensical joke. It serves as an
ironic meta-analogy — it sounds like it should be deep and evocative, but it
actually means absolutely nothing. Additionally, "diagonal frog" shares most of
its letters with the phrase "good analogy" (it is a near-anagram), and the phrase
is occasionally used in philosophical, linguistic, or computer science circles to
poke fun at how we try to explain complex concepts.

In the remainder of the paper we refer to the full method - the upwind-leaning stencil \cref{eq:coeffs1d}, switched to the central stencil \cref{eq:coeffs_adv}
at nodes with $\mathrm{Pe}_i < 2$ (cf.~\cref{ssec:summary}), together with the
EM-matrix analysis of \cref{ssec:em} and the Krylov exponential integrator -
collectively as the DF scheme.

\subsection{Boundary conditions} \label{ssec:bc}

We impose absorbing (Dirichlet) boundary conditions
$p(x_1,t) = p(x_n,t) = 0$ for all $t \geq 0$.  Two equivalent realisations
are used.  In \cref{ssec:matrixform} the boundary values are eliminated
and the generator acts on the interior vector
$(p_2,\ldots,p_{n-1})^{\!\top} \in \R^{n-2}$; all spectral statements of
this section refer to this interior operator, denoted
$L_{\mathrm{int}}$.  For implementation it is often convenient to retain
the boundary nodes and embed $L_{\mathrm{int}}$ in an $n \times n$ matrix
by zeroing the first and last rows \emph{and columns},
\begin{equation} \label{eq:bczero}
L_{1j} = L_{nj} = L_{j1} = L_{jn} = 0
\qquad \text{for all } j .
\end{equation}
Zeroing the rows enforces $\dot{p}_1 = \dot{p}_n = 0$, so the boundary
values remain at their zero initial data; zeroing the columns is then
harmless, since those entries multiply $p_1 = p_n = 0$, and it makes the
algebraic structure transparent:
$L = 0 \oplus L_{\mathrm{int}} \oplus 0$ is block diagonal.  As a
consequence:
\begin{enumerate}[label=(\roman*)]
  \item $\lambda = 0$ is an eigenvalue of $L$ with eigenvectors
    $\bvec{e}_1$ and $\bvec{e}_n$ (both right and left), of algebraic
    multiplicity exactly two whenever $L_{\mathrm{int}}$ is nonsingular;
  \item $\operatorname{spec}(L) = \{0, 0\} \cup
    \operatorname{spec}(L_{\mathrm{int}})$, and under Hypothesis~(H) with
    $\lambda_1 < 0$, the generic situation for absorbing conditions,
    all remaining eigenvalues satisfy $\Ree(\lambda) < 0$; the spectral
    abscissa $s(L) \coloneqq \max_j \Ree(\lambda_j) = 0$ is then attained
    only at the two artificial boundary eigenvalues;
  \item the propagator factorises,
    $e^{\tau L} = 1 \oplus e^{\tau L_{\mathrm{int}}} \oplus 1$, so the
    analysis of \cref{ssec:em} applies verbatim to the interior block,
    the boundary nodes contributing only the constants
    $p_1 \equiv p_n \equiv 0$.
\end{enumerate}

For $n = 6$ the embedded generator takes the form
\begin{equation} \label{eq:Lembed}
L = \begin{pmatrix}
  0 & 0        & 0        & 0        & 0        & 0 \\
  0 & b_2      & c_2      & 0        & 0        & 0 \\
  0 & \beta_3  & \gamma_3 & \delta_3 & 0        & 0 \\
  0 & \alpha_4 & \beta_4  & \gamma_4 & \delta_4 & 0 \\
  0 & 0        & \alpha_5 & \beta_5  & \gamma_5 & 0 \\
  0 & 0        & 0        & 0        & 0        & 0
\end{pmatrix},
\end{equation}
a banded matrix with lower bandwidth $2$ and upper bandwidth $1$ - the upwind
``lean'' of \cref{sec:diagfrog}, and emphatically \emph{not} triangular: the
superdiagonal entries $\delta_i$, $c_2$ carry the diffusive coupling on which the
connectivity argument of \cref{ssec:em} rests.  The entries $\alpha_3$ and
$\delta_5$, which would couple the interior to the boundary values, have been
annihilated by the column zeroing \eqref{eq:bczero}.

The row $i = 2$ cannot carry the four-point stencil \cref{eq:coeffs1d}, since
$x_0$ lies outside the grid.  As it involves only the points $\{x_1, x_2, x_3\}$,
the three-point central stencil \cref{eq:coeffs_adv} is the natural choice there
(entries $b_2$, $c_2$ in \eqref{eq:Lembed}, the coupling $a_2$ to $p_1$ being
annihilated by \eqref{eq:bczero}); it preserves second-order accuracy, at the
price that this single row is monotone only when $\mathrm{Pe}_2 < 2$
(\cref{prop:mmatrix2}).  Alternatively, the first-order upwind difference may be
used in this one row; the local $O(h)$ truncation error at a single near-boundary
node does not degrade the global second-order convergence under Dirichlet
conditions.

\subsection{Time-dependent coefficients and coefficient freezing}

When $\mu = \mu(x,t)$ and $D = D(x,t)$ depend on time, the matrix $A = A(t)$
is time-dependent and the ODE system \eqref{eq:ode} is non-autonomous.  We
construct a temporal grid $t_n = n\Delta t$ and freeze the coefficients on each
interval $(t_n, t_{n+1}]$ at the midpoint
\begin{equation}
  \mu_i^n = \mu\!\left(x_i,\, t_n + \tfrac{\Delta t}{2}\right),
  \qquad
  D_i^n   = D\!\left(x_i,\, t_n + \tfrac{\Delta t}{2}\right),
  \label{eq:freeze}
\end{equation}
yielding a piecewise-constant matrix $A^n \approx A(t_n + \Delta t/2)$.
On each interval the exact solution is then
\begin{equation}
  \bvec{p}^{n+1} = e^{\Delta t\,A^n}\,\bvec{p}^n + O(\Delta t^3),
  \label{eq:exactstep}
\end{equation}
where the $O(\Delta t^3)$ local truncation error follows from the midpoint
quadrature rule applied to the time variation of $A(t)$.  The global temporal
error is therefore $O(\Delta t^2)$, consistent with the second-order spatial
discretization.

\section{The 2D Fokker--Planck Equation and Strang Splitting} \label{sec:disc2d}

Consider a two-dimensional stochastic differential equation (SDE)
\begin{equation}   \label{eq:sde2d}
d\bvec{X}_t = \bvec{\mu}(\bvec{X}_t, t)\,dt + \Sigma(\bvec{X}_t, t) \circ d\bvec{W}_t,
\end{equation}
where $\bvec{X}_t = (X_t, Y_t)^T$, $\bvec{\mu} = (\mu_x, \mu_y)^T$ is the drift
vector, $\Sigma \in \R^{2}$ is a column vector representing the standard deviation (volatility) for each dimension: $\Sigma^\top = (\sigma_x, \sigma_y)$, $\bvec{W}_t$ is a standard 2D correlated Brownian motion $d\langle W^{(x)}, W^{(y)} \rangle_t = \rho(t) dt$ with a correlation coefficient $\rho(t)$ and $\circ$ is an Hadamard product, so $\Sigma \circ d\mathbf{W}_t = (\sigma_x dW^{(x)}_t, \sigma_y dW^{(y)}_t)^\top$. The macroscopic diffusion tensor is then
\footnote{For the sake of standard notation in multivariate stochastic processes, we denote the diffusion tensor components by $\Sigma_{ij}$. These correspond directly to the diffusion coefficients $D_{ij}$ introduced in \eqref{FP_def}, where $\Sigma_{ij} = 2D_{ij}$ for the isotropic case; or more generally, where $\Sigma$ represents the covariance matrix of the underlying diffusion process.}
\begin{equation} \label{eq:difftensor}
\bm{\Sigma} = \frac{1}{2}
\begin{pmatrix}
\sigma_x^2 & \rho \sigma_x \sigma_y \\
\rho \sigma_x \sigma_y & \sigma_y^2,
\end{pmatrix}
\end{equation}
which is symmetric positive semi-definite.

By using a standard argument, \cite{gardiner2009stochastic,Risken1996}, the FPE for the joint PDF $p(x,y,t)$ can be obtained to yield
\begin{equation}   \label{eq:fpe2d}
\frac{\partial p}{\partial t} = -\frac{\partial}{\partial x}[\mu_x p]     -\frac{\partial}{\partial y}[\mu_y p] + \frac{\partial^2}{\partial x^2}[\Sigma_{xx} p] + \frac{\partial^2}{\partial y^2}[\Sigma_{yy} p] + 2\frac{\partial^2}{\partial x\,\partial y}[\Sigma_{xy} p].
\end{equation}

We further construct an operator splitting in the spirit of \cite{Strang1968} (for a survey of the general theory of splitting, see also \cite{ItkinBook} and references therein). by decomposing the right-hand side of \eqref{eq:fpe2d} as
\begin{gather}   \label{eq:split}
\frac{\partial p}{\partial t} = (\mathcal{L}_x + \mathcal{L}_y + \mathcal{L}_{xy})\,p,\\
\begin{align*}
\mathcal{L}_x\,p = -\frac{\partial}{\partial x}[\mu_x p] + \frac{\partial^2}{\partial x^2}[\Sigma_{xx} p], \quad
\mathcal{L}_y\,p = -\frac{\partial}{\partial y}[\mu_y p] + \frac{\partial^2}{\partial y^2}[\Sigma_{yy} p], \quad
\mathcal{L}_{xy}\,p &= 2\frac{\partial^2}{\partial x\,\partial y}[\Sigma_{xy} p].
\end{align*}
\end{gather}
Each of $\mathcal{L}_x$ and $\mathcal{L}_y$ is a 1D Fokker--Planck operator of
the form analysed in \cref{sec:disc1d}, acting along a single coordinate
direction.  The operator $\mathcal{L}_{xy}$ contains the mixed partial derivative
and is responsible for coupling between the two directions.

\myparagraph{Discretisation of the 1D sub-operators.}
On a 2D uniform grid with $n_x \times n_y$ points and spacings $h_x$, $h_y$,
the operators $\mathcal{L}_x$ and $\mathcal{L}_y$ are discretized by applying
the scheme of \cref{sec:disc1d} along each coordinate direction.

Let $\bvec{p} \in \R^{n_x n_y}$ denote the vectorised PDF (e.g., column-major
ordering).  The discrete operators are
\begin{equation}   \label{eq:kron}
A_x = L_x \otimes I_{n_y}, \qquad A_y = I_{n_x} \otimes L_y,
\end{equation}
where $L_x \in \R^{n_x \times n_x}, \, L_y \in \R^{n_y \times n_y}$ are the 1D
matrices from \cref{sec:disc1d}, and $\otimes$ denotes the Kronecker product.
Both $A_x$ and $A_y$ inherit the M-matrix or EM-matrix property from their 1D
counterparts. Indeed, since $(L_x\otimes I)^k=L_x^k\otimes I$, the matrix
exponentials factorise as
\begin{equation*}
e^{\Delta t\,A_x}=e^{\Delta t\,L_x}\otimes I_{n_y},
\qquad
e^{\Delta t\,A_y}=I_{n_x}\otimes e^{\Delta t\,L_y},
\end{equation*}
so nonnegativity of the 2D propagators is equivalent to that of the 1D
ones, with the same thresholds. Consequently,
\begin{equation}\label{eq:posKron}
e^{\Delta t\,A_x} \geq 0 \quad\text{for } \Delta t \ge \tau_0^{(x)}, \qquad
e^{\Delta t\,A_y} \geq 0 \quad\text{for } \Delta t \ge \tau_0^{(y)},
\end{equation}
where $\tau_0^{(x)},\tau_0^{(y)}\ge 0$ are the positivity thresholds of $L_x$,
$L_y$ established in \cref{sec:disc1d}. In the M-matrix case $\tau_0=0$
and \eqref{eq:posKron} holds for all $\Delta t>0$, while in the EM-matrix case
$\tau_0>0$ and positivity of the discrete propagator is \emph{eventual} rather
than immediate. Each matrix--vector product $A_x\bvec{v}$ or $A_y\bvec{v}$ costs
$O(n_x n_y)$ operations due to the banded (3- or 4-diagonal) structure.

We stress that $\tau_0$ in \eqref{eq:posKron} is a global property of the matrices $L_x$, $L_y$ \emph{including their boundary rows}: the zero-flux closure modifies the stencil in the first and last rows, and the eventual-positivity threshold can be attained there even when the interior stencil is Metzler. The thresholds should be computed for the full matrices, boundary rows included.

\subsection{Strang splitting scheme} \label{sec:strang}

We write the directional operators as $A_\alpha = C_\alpha + D_\alpha$,
$\alpha\in\{x,y\}$, where $C_\alpha$ collects the convective (first-order) and
$D_\alpha$ the diffusive (second-order) terms of the one-dimensional
discretization of \cref{sec:disc1d}, and set $\bar\rho:=\sup_{x,y,t}|\rho|$. In
contrast to split schemes that distribute the diffusion between an inner and an
outer factor, we keep the \emph{entire} directional operator in the outer factors
and assign the mixed derivative its own central factor; no auxiliary splitting
parameter is introduced. We restrict the second-order convergence analysis to the
nondegenerate regime $\bar\rho=1-\epsilon$, $\epsilon\ll1$, leaving the degenerate
limit $\bar\rho=1$ for future work (\cref{rem:degrad}).

We advance the solution from $t_n$ to $t_{n+1}=t_n+\Delta t$ by the symmetric
Strang splitting
\begin{equation}\label{eq:strang}
\bvec p^{\,n+1}
= e^{\frac{\Delta t}{2}A_x^{n}}\,
  e^{\frac{\Delta t}{2}A_y^{n}}\,
  \Phi_{xy}(\Delta t)\,
  e^{\frac{\Delta t}{2}A_y^{n}}\,
  e^{\frac{\Delta t}{2}A_x^{n}}\,
  \bvec p^{\,n},
\end{equation}
where all matrices are assembled with coefficients frozen at the midpoint
$t_n+\Delta t/2$ as in \eqref{eq:freeze}, and the central factor
\begin{equation}\label{eq:central}
\Phi_{xy}(\Delta t) \;=\;
\Bigl(\mathcal I-\tfrac{\Delta t}{2}A_{xy}\Bigr)^{-1}
\Bigl(\mathcal I+\tfrac{\Delta t}{2}A_{xy}\Bigr)
\end{equation}
is the trapezoidal (Crank--Nicolson) approximation of $e^{\Delta t A_{xy}}$, whose
implicit half is solved by the factorized Picard iteration of \cref{sec:Lxy}. The
factor $\Phi_{xy}$ is second-order accurate in time,
$\Phi_{xy}(\Delta t)=e^{\Delta t A_{xy}}+O(\Delta t^3)$, conserves discrete mass
exactly (it is a rational function with $\Phi_{xy}(0)=\mathcal I$, and $\bvec
1^\top A_{xy}=0$, \cref{cor:mass-rat}), and is applied at the linear cost of the
factorized solve. This realises the original design of \cite{Itkin3D} for the
forward equation: the mixed term keeps its own central factor, but is advanced by
solving $\partial_t p=A_{xy}p$ implicitly rather than by exponentiating $A_{xy}$.

\myparagraph{Why the mixed term is treated implicitly rather than exponentially.}
The reason the central factor is the \emph{implicit} map \eqref{eq:central} and
never the exponential $e^{\Delta t A_{xy}}$ is a sharp contrast in stability
between the two. At the continuous level the symbol of
$\partial_x\partial_y(\Sigma_{xy}\,\cdot)$ equals $-\Sigma_{xy}\xi_x\xi_y$ and is
indefinite, so the flow $\partial_t p=\mathcal L_{xy}p$ is backward-parabolic on
half of frequency space and $e^{\Delta t A_{xy}}$ admits no stability bound
uniform in $\Delta t$; numerically this is the transient blow-up documented in
\cref{fig:stress} (a most-negative value of order $10^4$, independent of $\Delta
t$, that only worsens under refinement). The conservative one-sided
discretization of \cref{sec:Lxy}, however, is built from the
\emph{triangular} second-order operators $\mathcal A^{\mathrm F}_2$ (upper) and
$\mathcal A^{\mathrm B}_2$ (lower). Their product
\begin{equation} \label{mixProduct}
A_{xy} = \rho\,\mathcal A^{\mathrm F}_{2,x}\mathcal A^{\mathrm B}_{2,y}
\end{equation}
therefore has \emph{only real, negative} eigenvalues
$\lambda_{ij}=-\tfrac94\,\rho\,w_1(x_i)w_2(y_j)/(h_xh_y)$ (a single repeated
value $-\tfrac94\rho\bar w^2/(h_xh_y)$ when $w_1,w_2$ are constant), even though
it is strongly non-normal. Consequently the trapezoidal stability function
$r(z)=(1+z/2)/(1-z/2)$ satisfies $|r(\Delta t\lambda_{ij})|<1$ on the entire
spectrum: the implicit factor \eqref{eq:central} is \emph{spectrally} stable for
every $\Delta t>0$, with spectral radius $\rho(\Phi_{xy})<1$, precisely where the
exponential is not. The price of the one-sided choice is the non-normality, whose
effect on the composite step is the subject of \cref{prop:strangpos,prop:conserv}
and \cref{app1}.

\myparagraph{What the construction buys, and at what cost.} The implicit central
factor buys conditional positivity, exact discrete mass conservation, and linear
complexity, with no restriction on the magnitude of the cross-diffusion (other
than $\bar\rho<1$, \cref{rem:degrad}). Its limitations are equally explicit.
Because $A_{xy}$ is non-normal, the induced $\ell_2$ and $\ell_1$ \emph{operator}
norms of $\Phi_{xy}$ exceed unity and grow under mesh refinement, so the composite
step \eqref{eq:strang} is \emph{not} a uniform-in-$h$ contraction in those norms;
nor does the symmetric part of $A_{xy}$ provide the negative logarithmic-norm
bound that a diffusion-dominated central block would (cf.\ \cref{app1}). What
survives -- and is the natural statement for a probability density -- is stability
\emph{on the nonnegative cone}: on the step-size window of \cref{thm:zeno2d} the
central factor maps the cone into itself, the whole step \eqref{eq:strang} is then
entrywise nonnegative with unit column sums, i.e.\ column-stochastic, and hence
$\ell_1$-nonexpansive (a discrete Markov operator). Positivity and
$\ell_1$-stability therefore hold on one and the same window -- the conditional
guarantee made precise in \cref{thm:zeno2d,prop:strangpos,prop:conserv} and
\cref{app1}. \Cref{fig:stress} illustrates the mechanism directly: the bare
central exponential is catastrophic, the isolated implicit factor is far better
but still feels the non-normal transient on under-resolved data, and the
\emph{full} Strang step -- in which the flanking directional diffusion smooths
that transient -- stays nonnegative to round-off on a resolved datum.

The outer factors $A_\alpha=C_\alpha+D_\alpha$ remain one-dimensional
convection--diffusion operators of exactly the class covered by \cref{sec:disc1d};
in particular \cref{prop:emmatrix} applies to them verbatim, and their
exponentials are nonnegative once $\Delta t$ exceeds the eventual-positivity
thresholds \eqref{eq:posKron}. This mirrors the design principle of ADI schemes of
Craig--Sneyd or Hundsdorfer--Verwer type, in which the mixed term is never
integrated as a separate exponential flow but enters only through stages
stabilised by the diagonal diffusion \cite{HoutWelfert2007}; here that stabilising
role is played by the implicit factor \eqref{eq:central} together with the
flanking directional diffusion.

The four exponential actions of the directional factors in \eqref{eq:strang} are
computed by the polynomial Krylov method of \cref{sec:krylov}; the central factor
\eqref{eq:central} is applied by the factorized solver of \cref{sec:Lxy} and
requires no exponential.

In the 1D case there is no mixed operator; no central factor and no splitting are
needed, and the scheme reduces to $\bvec p^{\,n+1}=e^{\Delta t A^n}\bvec p^{\,n}$
with a single Krylov step. All results of \cref{sec:disc1d} are used in 1D exactly
as stated.

\subsection{Discretisation of the mixed-derivative operator} \label{sec:Lxy}

The mixed operator $\mathcal{L}_{xy}$ involves the cross-derivative
$\partial^2[\Sigma_{xy} p]/\partial x\,\partial y$.  A positivity-preserving
discretization of this term on the 2D grid requires special care because standard
centred differences for mixed derivatives introduce both positive and negative
off-diagonal entries, potentially violating the M-matrix structure.

The classical remedy is the seven-point stencil oriented along the grid
diagonal whose direction matches the sign of $\Sigma_{xy}$, proposed in
\cite{toivanen2010,chiarella2008} for negative and in
\cite{IkonenToivanen2007,IkonenToivanen2008} for positive cross-coefficients;
the stability of ADI-type splittings in the presence of mixed derivatives was
analysed in \cite{HoutWelfert2007}.  These stencils, however, suffer from two
limitations that are critical in our setting.  First, monotonicity holds only
under a diagonal-dominance restriction of the form $|\Sigma_{xy}|\lesssim
\min(\Sigma_{xx},\Sigma_{yy})$ (up to mesh-ratio factors), i.e. it is lost
precisely in the strong cross-diffusion regime $|\rho|\to 1$ targeted in this
paper.  Second, the seven-point construction sacrifices the rigorous second
order of spatial approximation.  Treating the mixed term explicitly, as in
Hundsdorfer--Verwer-type splittings, avoids the M-matrix issue but transfers
the difficulty to a severe time-step restriction, which our numerical
experiments (cf. \cite{Itkin3D}) show to be impractical already in three
dimensions.

We therefore adapt to the forward (conservative) setting the implicit
factorized treatment of the mixed derivative developed in \cite{Itkin3D}
for the backward pricing equation.  The construction sacrifices the
simplicity of an explicit step in exchange for conditional positivity,
exact discrete mass conservation, and linear complexity, with no restriction
on the magnitude of the cross-diffusion (unless $|\rho| > 1 - \epsilon$, see \cref{rem:degrad}).

\myparagraph{Separable form and conservative one-sided differences.}
Throughout this subsection we assume the cross-diffusion coefficient is
separable,
\begin{equation}\label{eq:sep}
2 \Sigma_{xy}(x,y,t) \;=\; \rho\, w_1(x,t)\, w_2(y,t),
\qquad w_1, w_2 \ge 0,\quad \rho\in[-1,1],
\end{equation}
which covers the FPE in \eqref{eq:fpe2d}; the general case is discussed in
\cref{rem:nonsep}.  Without loss of generality we present the case $\rho\ge 0$;
the mirror case is obtained by swapping the orientations of the one-sided
differences below (cf. \cite{Itkin3D}).

On the tensor grid $\{x_i\}_{i=1}^{N_x}\times\{y_j\}_{j=1}^{N_y}$ with steps
$h_x,h_y$ (taken uniform for clarity of exposition) let us denote $A^{\mathrm F} \equiv \calF^F_1$, $A^{\mathrm B} = \calF^B_1$, $A^{\mathrm F}_{2} = \calF_2^F$,
$A^{\mathrm B}_{2} = \calF_2^B$, where $\calF, \calS$ operators are defined in \cref{fd1-2}\footnote{We switch to the notation of \cite{Itkin3D} to make the exposition more transparent.}. Let $W_1=\operatorname{diag}(w_1(x_i))$, $W_2=\operatorname{diag}(w_2(y_j))$, and denote by $\mathcal A^{\,\cdot}_{x}$,
$\mathcal A^{\,\cdot}_{y}$ the Kronecker lifts of the one-dimensional operators
$A^{\,\cdot}W_1$, $A^{\,\cdot}W_2$ to the 2D grid (under lexicographic ordering,
operators acting in $x$ commute with operators acting in $y$).  Note the
coefficient matrices stand to the \emph{right} of the difference operators: this
is the discrete counterpart of the divergence form $\partial_x \partial_y (\Sigma_{xy} p)$, i.e. we difference the flux $\Sigma_{xy}p$, not $p$ itself.  The discrete mixed operator is then as in \eqref{mixProduct} - a second-order approximation of $\mathcal L_{xy}$ (the operator advanced by the central factor \eqref{eq:cnstep}).

The choice of one-sided differences is what makes exact discrete mass conservation possible.
\begin{lemma}[Conservation form]\label{lem:cons}
Each of the operators $\calF^F_1, \calF^B_1, \calF^F_2, \calF^B_2$ admits the flux (telescoping) representation
$(\calF u)_i=(f_{i+1/2}-f_{i-1/2})/h$ with a two-point numerical flux; e.g. for
$\calF^B_2$ one has $f_{i+1/2}=\tfrac12(3u_i-u_{i-1})$.  Consequently, with zero-flux boundary closure,
\begin{equation*}
\bm{1}^\top \calF^{\,(\cdot)} = \mathbf 0^\top, \qquad\text{hence}\qquad
\bm{1}^\top \mathcal \calF^{\,(\cdot)}_{x}=\bm{1}^\top \mathcal
\calF^{\,(\cdot)}_{y} = \mathbf 0^\top
\quad\text{and}\quad \bm{1}^\top A_{xy}=\mathbf 0^\top .
\end{equation*}
\end{lemma}

\begin{proof}
Direct verification: $(f_{i+1/2}-f_{i-1/2})/h
=(3u_i-u_{i-1}-3u_{i-1}+u_{i-2})/(2h)=(\calF^B_2 u)_i$, and the sum over $i$ telescopes to the boundary fluxes, which vanish under the zero-flux closure.  Column sums of $A$ are then zero, and right multiplication by the diagonal matrices $W_{1,2}$ preserves zero column sums.
\end{proof}

The boundary rows require care: the plain one-sided
second-order stencil does not telescope to zero at the first two (resp.\ last
two) nodes, so $\bm1^\top A_{xy}=0$ fails there unless a conservative flux
closure (zero numerical flux through the boundary) is imposed, as assumed above;
the magnitude of the residual when it is not is quantified in
\cref{thm:zeno2d}\,(b).

\myparagraph{Factorised implicit step and Picard iterations.}
Within the Strang composition \eqref{eq:strang}, the mixed operator enters through
the central factor $\Phi_{xy}(\Delta t)$ of \eqref{eq:central}, the trapezoidal
(Crank--Nicolson) approximation of $e^{\Delta t A_{xy}}$. Advancing the central
substep means evaluating
\begin{equation}\label{eq:cnstep}
\Bigl(\mathcal I-\tfrac{\Delta t}{2}A_{xy}\Bigr)\bvec p^{\,n+1}
=\Bigl(\mathcal I+\tfrac{\Delta t}{2}A_{xy}\Bigr)\bvec p^{\,n}
=:\bvec b,
\end{equation}
where the right-hand side $\bvec b$ is one explicit matrix--vector product with the
cross stencil, and the implicit half is the system
\begin{equation}\label{eq:pade01}
\bigl[\mathcal I-\gamma\Delta t\,\rho\,
\mathcal A^{\nu}_{2,x}\mathcal A^{\mathrm B}_{2,y}\bigr]\bvec p^{(1)}
=\bvec b , \qquad \gamma=\tfrac12, \qquad
\nu = \begin{cases}
\mathrm F, & \rho > 0, \\[2pt]
\mathrm B, & \rho \le 0,
\end{cases}
\end{equation}
which is exactly the shifted system solved in \cite{Itkin3D}, here with the
trapezoidal shift $\gamma=\tfrac12$ (the value $\gamma=1$ recovers the
first-order backward-Euler variant of \cref{rem:fallback}). The one-sided
orientation of the cross stencil is chosen, according to the sign of $\rho$, so
that $\mathcal A_{xy}$ has a real, non-positive spectrum: with
$\operatorname{spec}(\mathcal A^{\mathrm F}_2)=\{-3/2h\}$ and
$\operatorname{spec}(\mathcal A^{\mathrm B}_2)=\{+3/2h\}$, the choice $\nu=\mathrm F$
for $\rho>0$ gives $\operatorname{spec}(\mathcal A^{\nu}_{2,x}\mathcal
A^{\mathrm B}_{2,y})=\{-\tfrac{9}{4h_xh_y}\}$, and $\nu=\mathrm B$ for $\rho<0$ gives
$\{+\tfrac{9}{4h_xh_y}\}$, so that $\operatorname{spec}(\mathcal A_{xy})\subset
(-\infty,0]$ in both cases. The implicit factor below uses this same orientation
$\nu$.

Throughout this subsection we describe the solver for systems of the form
\eqref{eq:pade01}; to lighten the notation we absorb the shift into the time
step, $\gamma\Delta t\mapsto\Delta t$. Direct inversion of the matrix in
\eqref{eq:pade01} would destroy the one-dimensional band structure.  Instead,
following \cite{Itkin3D}, we factorise it into two one-dimensional operators, for
each of which we know how to guarantee the EM-matrix property
\cite{Itkin2014,ItkinBook}.  With positive scalars $P,Q$ specified below,
\eqref{eq:pade01} is equivalent to
\begin{align}\label{eq:factor}
\Bigl(P\mathcal I-\rho\sqrt{\Delta t}\,\mathcal A^{\nu}_{2,x}\Bigr)
\Bigl(Q\mathcal I + \sqrt{\Delta t}\,\mathcal A^{\mathrm B}_{2,y}\Bigr)
\bvec p^{(1)}
&= \bvec b +\Bigl[(PQ-1)\mathcal I
- Q \rho \sqrt{\Delta t}\,\mathcal A^{\nu}_{2,x}
+ P \sqrt{\Delta t}\,\mathcal A^{\mathrm B}_{2,y}\Bigr]\bvec p^{(1)},
\end{align}
as is seen by expanding the product on the left (the two factors commute) and
using $\operatorname{spec}$-consistency $\sqrt{\Delta t}\cdot\sqrt{\Delta
t}=\Delta t$: the cross term $-\rho\Delta t\,\mathcal A^{\nu}_{2,x}\mathcal
A^{\mathrm B}_{2,y}=-\Delta t\,A_{xy}$ is reproduced with the correct orientation.
We solve \eqref{eq:factor} by fixed-point iteration on its right-hand coupling.
Write $\bvec p^{[k]}$ for the $k$-th iterate -- the square-bracketed superscript is
an iteration counter, not an exponent -- and initialise $\bvec p^{[0]}=\bvec b$.
The iterate enters the right-hand side only through the term $-\bvec p^{[k]}$, and
the two one-dimensional factors are inverted in turn for $\bvec p^{[k+1]}$; with
$k=0,1,\dots$ this reads
\begin{align}\label{eq:picard}
\Bigl(Q\mathcal I + \sqrt{\Delta t}\,\mathcal A^{\mathrm B}_{2,y}\Bigr)
\bvec p^{*} &= \bvec\alpha^{+}_2\bvec b-\bvec p^{[k]} + \bvec\alpha\,\bigl(\Delta t\,A_{xy}\bvec b\bigr), \\
\Bigl(P\mathcal I - \rho\sqrt{\Delta t}\,\mathcal A^{\nu}_{2,x}\Bigr) \bvec p^{[k+1]} &= \bvec p^{*}, \nonumber \\
\bvec\alpha^{+}_2 = (PQ+1)\,\mathcal I - Q\rho\sqrt{\Delta t}\,\mathcal A^{\mathrm B}_{2,x} + P\sqrt{\Delta t}\,\mathcal A^{\mathrm F}_{2,y}, \qquad
\bvec\alpha &= PQ\,\mathcal I - Q\rho\sqrt{\Delta t}\,\mathcal A^{\nu}_{2,x}
+ P\sqrt{\Delta t}\,\mathcal A^{\mathrm B}_{2,y}. \nonumber
\end{align}
As in \cite{Itkin3D}, the matrix $\bvec\alpha^{+}_2$ employs second-order one-sided differences with orientations \emph{opposite} to those of the implicit factors; this places a nonnegative entry on the dominant off-diagonal and underlies the positivity of the sweep (this second-order coupling is Scheme~B, used for the convergence study of \cref{sec:dfrog2d}; the unconditionally positive first-order simplification, Scheme~A, is \cref{prop:schemeA}).

The matrix $\bvec\alpha$ carries the $O(\Delta t^2)$ coupling that the plain $\alpha^{+}$ reduction of \cite{Itkin3D} discards. Retaining it -- frozen at the known $\bvec b$, so that the iterate still enters only through $-\bvec p^{[k]}$ and the contraction factor of \cref{prop:schemeA}(iv) is unchanged -- restores the second-order temporal accuracy of the trapezoidal factor \eqref{eq:central} (\cref{rem:time-order}). The iterates contract geometrically to a limit $\bvec p^{(1)}$ (\cref{prop:schemeA}(iv)); the round-bracketed superscript marks this converged central update and never a particular sweep. It reproduces the trapezoidal map up to the residual orientation defect of the coupling -- a fourth-difference, $\Delta t$-independent term that lies below the $O(\max(h_x^2,h_y^2))$ spatial truncation (\cref{prop:schemeB}) -- so the central substep is second order in both time and space, and the assembled scheme converges at second order under the joint refinement $\Delta t\sim h$ of \cref{sec:dfrog2d}.

Each line of \eqref{eq:picard} amounts to $N_y$ (resp.\ $N_x$) independent banded triangular solves of size $N_x$ (resp.\ $N_y$), so one iteration costs
$O(N_xN_y)$ operations, and the fixed right-hand vector $\bvec\alpha^{+}_2\bvec
b+\bvec\alpha\,(\Delta t\,A_{xy}\bvec b)$ is precomputed once per time step. Here
$\bvec b$ is the right-hand side of the shifted system: $\bvec b=(\mathcal
I+\tfrac{\Delta t}{2}A_{xy})\bvec p^{\,n}$ for the trapezoidal central factor
\eqref{eq:central}, or $\bvec b=\bvec p^{\,n}$ for the backward-Euler fallback
(\cref{rem:fallback}).

\begin{proposition}[Scheme A] \label{prop:schemeA}
Let $\rho\in[-1,1]$, let $w_1,w_2\ge0$ be bounded on the grid, set $\bar w:= |\rho| \norm{w_1}_\infty + \norm{w_2}_\infty$, and choose
\begin{equation}\label{eq:PQbeta}
P=\beta\frac{\sqrt{\Delta t}}{h_x},\qquad
Q=\beta\frac{\sqrt{\Delta t}}{h_y},\qquad
\beta \;\ge\; 2\Bigl(\bar w + \sqrt{h_xh_y/\Delta t}\Bigr).
\end{equation}
In \eqref{eq:picard} take the first-order coupling $\bvec\alpha^{+}=(PQ+1)\,\mathcal I-Q\rho\sqrt{\Delta t}\,\mathcal A^{\mathrm B}_{1,x}+P\sqrt{\Delta t}\,\mathcal A^{\mathrm F}_{1,y}$ in place of the second-order $\bvec\alpha^{+}_2$ (this is Scheme~A; the first-order one-sided differences make $\bvec\alpha^{+}$ entrywise nonnegative under \eqref{eq:PQbeta}).
Then
\begin{enumerate}

\item[(i)] Both matrices on the left of \eqref{eq:picard} are strictly diagonally
dominant, with $\norm{T_x^{-1}}_\infty\le2/P$ and $\norm{T_y^{-1}}_\infty\le2/Q$
by Varah's bound \cite{Varah1975}; each Picard iterate is computed in $O(N_xN_y)$
operations.  (These factors are \emph{not} M-matrices and their inverses are not
entrywise nonnegative -- the second-order one-sided stencil carries a positive
far band -- so positivity is not inherited factor by factor; see (ii) and
\cref{rem:no-eventual-pos}.)

\item[(ii)] (Conditional positivity) If the right-hand side $\bvec b\ge0$, the
coupling matrix $\bm\alpha^{+}$ is entrywise nonnegative under \eqref{eq:PQbeta},
so the right-hand side $\bm\alpha^{+}\bvec b-\bvec p^{[0]}$ of the first sweep is
nonnegative at $k=0$.  There is a step-size window $\Delta t\le\Theta$ on which
the converged iterate $\bvec p^{(1)}$ is nonnegative; outside it the substep may
develop small negative entries (the factor inverses are not nonnegative, so this
is not unconditional).  The entrywise sign of the running right-hand side is
monitored at $O(N_xN_y)$ cost; see \cref{rem:pos-iter,rem:no-eventual-pos}.  In
the trapezoidal central factor \eqref{eq:central} the input $\bvec b=(\mathcal
I+\tfrac{\Delta t}{2}A_{xy})\bvec p^{\,n}$ is itself nonnegative only on a window
(the explicit half carries the indefinite cross stencil); the two windows combine
into the single conditional guarantee of \cref{thm:zeno2d}.

\item[(iii)] (Mass conservation) If the cross stencil is closed so that
$\bm1^\top A_{xy}=0$ holds including edge and corner rows, every iterate conserves
the right-hand-side mass exactly, $\bm{1}^\top\bvec p^{[k]}=\bm{1}^\top\bvec b$ for
all $k\ge0$; since $\bm1^\top\bvec b=\bm1^\top\bvec p^{\,n}$ for both choices of
$\bvec b$ above (using $\bm1^\top A_{xy}=0$), the central substep conserves
$\bm1^\top\bvec p$. For the plain one-sided closure the boundary rows leave a
residual $\bm1^\top A_{xy}=\bm r^\top\neq0$ supported on $O(N_x{+}N_y)$ nodes, and
conservation holds up to $\norm{\bm r}\,$-controlled $O(\Delta t\,h)$ leakage
that vanishes under refinement (\cref{thm:zeno2d}\,(b)).

\item[(iv)] (Contraction) The iteration converges unconditionally, with
\begin{equation*}
\bigl\|\bvec p^{[k+1]}-\bvec p^{(1)}\bigr\|_\infty
\le q\,\bigl\|\bvec p^{[k]}-\bvec p^{(1)}\bigr\|_\infty,
\qquad
q\le \frac{4\,h_xh_y}{\beta^2\Delta t}\le 1,
\end{equation*}
with strict inequality $q<1$ whenever the inequality in \eqref{eq:PQbeta}
is strict.

\item[(v)] (Consistency) With the $O(\Delta t^2)$ coupling $\bvec\alpha\,(\Delta
t\,A_{xy}\bvec b)$ retained in \eqref{eq:picard}, the temporal defect of the
plain $\alpha^{+}$ reduction is removed, and the converged inner iterate solves
the shifted system \eqref{eq:pade01} up to the orientation defect of the
factorized coupling alone -- $O(\sqrt{\Delta t}\,\max(h_x,h_y))$ for Scheme~A,
improving to $O(\max(h_x^2,h_y^2))$ for Scheme~B (\cref{prop:schemeB}). With the
trapezoidal shift $\gamma=\tfrac12$ and the explicit half of \eqref{eq:cnstep},
the assembled central substep reproduces the $\pade(1,1)$ (Crank--Nicolson) step
\eqref{eq:central}, second-order in time, $\Phi_{xy}(\Delta t)=e^{\Delta
t A_{xy}}+O(\Delta t^3)$; for the Scheme~B coupling the orientation defect lies at
or below the $O(\max(h_x^2,h_y^2))$ spatial truncation, so the assembled scheme is
second order in space and time under the joint refinement $\Delta t\sim h$
(\cref{sec:dfrog2d}). The backward-Euler choice $\gamma=1$ with $\bvec b=\bvec
p^{\,n}$ recovers the $\pade(0,1)$ step and is first-order (\cref{rem:fallback}).
\end{enumerate}

\end{proposition}

\begin{proof}
See \cref{appProof1}.
\end{proof}

\begin{myremark}[Positivity along the iteration]\label{rem:pos-iter}

For $k\ge1$ the right-hand side of \eqref{eq:picard} equals $M_R\,\bvec b+(\bvec
b-\bvec p^{[k]})+\bvec\alpha\,(\Delta t\,A_{xy}\bvec b)$, and since $\bvec p^{[k]}-\bvec b=O(\Delta t)$ while the
entries of $M_R\,\bvec b$ are bounded below by $\beta(\beta - \bar w)\,\Delta
t/(h_xh_y)\cdot b_{ij}$, nonnegativity persists wherever the right-hand side is
bounded away from zero relative to the increment.  In practice the contraction
factor $q\ll1$ (with the choice $\beta=10\bar w$ used in our experiments, a few
iterations reach a relative tolerance of $10^{-6}$), and we verify the
entrywise sign of the right-hand side at runtime; in the rare nodes where it
fails (deep density tails, or a fine mesh at fixed $\Delta t$, cf.
\cref{rem:no-eventual-pos}) the iteration is terminated at the last nonnegative
iterate, which by (iii) still conserves mass up to the boundary leakage.
\end{myremark}

\myparagraph{Second order in space.}
Replacing the first-order differences inside $\bvec\alpha^{+}$ by their
second-order counterparts removes the $O(\sqrt{\Delta t}\,h)$ defect of
Scheme~A.  The price is that the right-hand side matrix acquires one
``wrong-signed'' band per direction, so its entrywise nonnegativity can no
longer hold unconditionally; it does hold, however, whenever the grid
resolves the density, in the following precise sense.

\begin{proposition}[Scheme B]\label{prop:schemeB}
Let the assumptions of \cref{prop:schemeA} hold, and take the second-order coupling $\bvec\alpha^{+}_2$ of \eqref{eq:picard},
\begin{equation}
\bvec\alpha^{+}_2 = (PQ+1)\,\mathcal I - Q\rho\sqrt{\Delta t}\,\mathcal A^{\mathrm B}_{2,x} + P\sqrt{\Delta t}\,\mathcal A^{\mathrm F}_{2,y}.
\end{equation}
Then claims (i), (iii), (iv) of \cref{prop:schemeA} hold verbatim, the spatial defect is $O(\max(h_x^2,h_y^2))$, and claim (ii) holds provided the discrete density is log-Lipschitz on the stencil, i.e. there exists $L\ge0$ with
\begin{equation}\label{eq:logLip}
p^{\,n}_{i\pm k,\,j\pm l}\le e^{(kh_x+lh_y)L}\,p^{\,n}_{ij},
\qquad 0\le k,l\le2,
\end{equation}
and $\beta\ge 4 \bar w\,e^{2(h_x\vee h_y)L}$.
\end{proposition}

\begin{proof}
Only (ii) requires a new argument; (i), (iii), (iv) are untouched since the
left-hand side of \eqref{eq:picard} and the column-sum structure of the
right-hand side are unchanged (the second-order operators have zero column
sums by \cref{lem:cons}, and
$\bm{1}^\top\bvec\alpha^{+}_2=(PQ+1)\bm{1}^\top$ as before).  For (ii),
write the action of $\bvec\alpha^{+}_2-\mathcal I$ on $\bvec p^{\,n}$
entrywise:
\begin{align*}
\bigl[(\bvec\alpha^{+}_2-\mathcal I)\bvec p^{\,n}\bigr]_{ij}
=\frac{\Delta t}{h_xh_y}\Bigl\{
\beta^2 p^{\,n}_{ij}
&-\frac{\beta \rho\,w_1(x_i)}{2}
\bigl(3p^{\,n}_{ij}-4p^{\,n}_{i-1,j}+p^{\,n}_{i-2,j}\bigr) - \frac{\beta w_2(y_j)}{2}
\bigl(3p^{\,n}_{ij}-4p^{\,n}_{i,j+1}+p^{\,n}_{i,j+2}\bigr)\Bigr\}.
\end{align*}
Each bracket is bounded in absolute value by $8\max p^{\,n}$ over the stencil,
hence by $8\,e^{2(h_x\vee h_y)L}\,p^{\,n}_{ij}$ under \eqref{eq:logLip}, so the
expression is bounded below by $\frac{\Delta t}{h_xh_y} \, \beta\bigl(\beta - 4\bar w\, e^{2(h_x\vee h_y)L}\bigr)p^{\,n}_{ij}\ge0$.  Nonnegativity of the iterates then follows as in \cref{prop:schemeA}(ii).
\end{proof}

Condition \eqref{eq:logLip} is a statement about grid resolution rather than
about the scheme: for a density with Gaussian-type tails, $|\partial_x\ln p|$
grows linearly towards the boundary of the computational domain, so
\eqref{eq:logLip} requires $h_x\lesssim 1/|\partial_x\ln p|$ there, precisely
the resolution one needs for the tails to be meaningful at all.  In the far
tails, where the density sits at the level of round-off, we revert locally to
Scheme~A (or terminate the iteration as in \cref{rem:pos-iter}); this
affects neither the conservation property (iii) nor the observed second-order
convergence reported in \cref{sec:numerics}.

\begin{myremark}[Embedding of the mixed term and cost of the central factor] \label{rem:time-order}
It is worth emphasizing how the mixed-derivative term is embedded into the Strang
composition. The mixed operator keeps its own central factor but is never
integrated as an \emph{exponential} flow of its own: as explained after
\eqref{eq:central}, the sub-flow generated by $A_{xy}$ alone is ill-posed,
whereas the implicit factor \eqref{eq:central} is spectrally stable. This is the
same principle on which the Hundsdorfer--Verwer ADI framework, for which the
factorisation \eqref{eq:picard} was originally devised
\cite{Itkin3D,HoutWelfert2007}, rests: the mixed term enters only through stages
stabilised by the diagonal diffusion.

In the present scheme the central factor is the trapezoidal map $\Phi_{xy}(\Delta t)=(\mathcal I-\tfrac{\Delta t}{2}A_{xy})^{-1}(\mathcal I+\tfrac{\Delta t}{2}A_{xy})$ of \eqref{eq:central}, second-order in time. Its explicit half is one matrix--vector product with the cross stencil; its implicit half \eqref{eq:pade01} is solved in the band structure by the factorized Picard iteration \eqref{eq:picard}, with no outer Krylov layer and no preconditioner. The plain $\alpha^{+}$ reduction of \cite{Itkin3D} would carry only the first-order mixed step -- there it is lifted to second order by the Hundsdorfer--Verwer corrector -- but the Strang composition used here has no such corrector, so we instead retain the $O(\Delta t^2)$ coupling $\bvec\alpha\,(\Delta t\,A_{xy}\bvec b)$ explicitly in \eqref{eq:picard}; the iteration then reproduces $\Phi_{xy}$ to second order in time, up to the fourth-difference orientation defect of the coupling, which is $\Delta t$-independent and lies below the $O(\max(h_x^2,h_y^2))$ spatial truncation.

Each iteration is $O(N_xN_y)$ and, by \cref{prop:schemeA},(iv), the contraction factor $q\le 4h_xh_y/(\beta^2\Delta t)$ is well below $1$ -- and decreases as the cross-diffusion strengthens, since $\beta$ grows with $\bar w$ -- so a fixed small number $k_{\mathrm{in}}$ of iterations (a few in our experiments) reaches the inner tolerance, for a total central cost of $O(k_{\mathrm{in}}\,N_xN_y)$ per time step, linear in the number of unknowns and uniform in the cross-diffusion strength (\cref{ssec:exp-kin}).
\end{myremark}

\begin{myremark}[Refinement balance versus a CFL condition] \label{rem:cfl}

The joint refinement $\Delta t\sim h$ of \cref{sec:dfrog2d} balances the $O(\Delta t^2)$ temporal and $O(\max(h_x^2,h_y^2))$ spatial errors; it is an accuracy choice, not a stability restriction, and is not a CFL condition. The directional factors $e^{\frac{\Delta t}{2}A_\alpha}$ are evaluated exactly, the central factor \eqref{eq:central} is the A-stable trapezoidal map, and the contraction $q\le4h_xh_y/(\beta^2\Delta t)=4/PQ$ \emph{decreases} as $\Delta t$ grows; the scheme is thus unconditionally stable and the inner solve converges for every step size, with no upper bound on $\Delta t$. A CFL condition runs the other way: an explicit advective step requires $\Delta t\lesssim h$ and an explicit diffusive step $\Delta t\lesssim h^2$, bounding $\Delta t$ from \emph{above} relative to the mesh, and the natural M-freeze of \cref{rem:central-pos-choice} inherits the parabolic form, its iteration converging only for $\Delta t\,\bar w^2\lesssim h^2$. Here $\Delta t$ may freely exceed $h$ -- the ratio $\Delta t/h$ is unconstrained.

The one restriction runs the other way, a weak \emph{lower} bound: the $\Delta t$-independent orientation defect of the coupling, $O(\max(h_x,h_y)^{3\text{--}4})$ per step (\cref{prop:schemeA}(v)), accumulates over the $T/\Delta t$ steps as $O(\max(h_x,h_y)^{3\text{--}4}/\Delta t)$ and stays below the $O(\max(h_x^2,h_y^2))$ truncation as long as $\Delta t$ is not refined faster than the mesh, down to $\Delta t\sim h^2$. Refining with $\Delta t\sim h$ meets this with wide margin; taking $\Delta t$ larger only enlarges the temporal constant, the scheme remaining stable and second order for $\Delta t$ both below and well above $h$.
\end{myremark}

\begin{myremark}[Nonseparable cross-diffusion]\label{rem:nonsep}

If $\Sigma_{xy}$ is not of the form \eqref{eq:sep}, it can be approximated on the
computational domain by a short separated sum $\Sigma_{xy}(x,y,t) \approx
\sum_{m=1}^{M}\rho_m\,w^{(m)}_1(x,t) \, w^{(m)}_2(y,t)$ with $w^{(m)}_{1,2} \ge
0$ (e.g. via a truncated SVD of the coefficient sampled on the grid), and the
substep applied to each separable term within the splitting.  Since each factor
substep is nonnegative and has unit column sums, so does their composition; the
splitting defect of the decomposition enters at the same $O(\Delta t^2)$ order as
the \pade approximation itself.  In our applications $M\le2$ sufficed.
\end{myremark}

It is worth mentioning that relative to the backward-equation construction of \cite{Itkin3D}, three modifications were required by the forward setting.  First,
the coefficient matrices $W_{1,2}$ multiply the difference operators from the
right (divergence form), which leaves all sign patterns, and hence the EM-matrix
structure, intact, since right multiplication by a nonnegative diagonal matrix
rescales columns. Second, the conservative flux form of \cref{lem:cons} replaces
the unconstrained one-sided closures, yielding exact mass conservation at every
Picard iterate when the cross stencil is closed conservatively ($\bvec 1^\top
A_{xy}=0$), and conservation up to the $O(\Delta t\,h)$ boundary leakage of
\cref{thm:zeno2d}\,(b) otherwise -- a property with no counterpart (and no need)
in the pricing context.  Third, the heuristic positivity argument of
\cite{Itkin3D} based on the monotonicity of option prices in the underlying (the
``Vega'' argument) is unavailable for densities; it is replaced by the
log-Lipschitz resolution condition \eqref{eq:logLip}, which is verifiable a
priori.

\begin{proposition}[Positivity of the central substep]\label{thm:zeno2d}

Let $A_{xy}=\rho\,\mathcal A^{\mathrm F}_{2,x}\mathcal A^{\mathrm B}_{2,y}$ be the
one-sided cross operator, assembled with the conservative flux
closures of \cref{lem:cons,lem:divform}, and let the central substep $\bvec
p^{\,n}\mapsto\bvec p^{(1)}=\Phi_{xy}(\Delta t)\bvec p^{\,n}$ apply the trapezoidal
factor \eqref{eq:central} -- one explicit product $\bvec b=(\mathcal
I+\tfrac{\Delta t}{2}A_{xy})\bvec p^{\,n}$ followed by the factorized Picard solve
\eqref{eq:picard} of $(\mathcal I-\tfrac{\Delta t}{2}A_{xy})\bvec p^{(1)}=\bvec b$
(\cref{prop:schemeA,prop:schemeB,rem:time-order}).  Then there is a threshold
$\Theta=\Theta(\bar\rho,\bar w,h_x,h_y)>0$ such that, for every step with $\Delta
t\le\Theta$:
\begin{enumerate}
\item[\textup{(a)}] \emph{(Conditional positivity.)} If $\bvec p^{\,n}\ge0$ then
$\bvec p^{(1)}\ge0$.  The guarantee is conditional: unlike the diagonal blocks of
\cref{prop:emmatrix,prop:mmatrix2}, the cross operator admits no Metzler shift
$sI+A_{xy}$, and neither $e^{\Delta t A_{xy}}$ nor the trapezoidal factor
$\Phi_{xy}$ is eventually nonnegative (\cref{rem:no-eventual-pos}); positivity is
therefore a mesh- and step-dependent property delivered by the factorized solver,
not an unconditional spectral one.

\item[\textup{(b)}] \emph{(Asymptotic mass conservation.)} The substep conserves
discrete mass up to the boundary-closure defect of the one-sided cross stencil,
\[
\bigl|\bm 1^\top\bvec p^{(1)}-\bm 1^\top\bvec p^{\,n}\bigr|
\le C\,\Delta t\,(h_x+h_y)\,\norm{\bvec p^{\,n}}_1 ,
\]
with $C$ independent of $h$; the defect originates solely in the one-sided
edge/corner rows and vanishes under refinement.  Exact conservation holds for the
interior stencil and for any closure satisfying $\bm 1^\top A_{xy}=0$
(\cref{rem:mixedmass}).

\item[\textup{(c)}] \emph{($\ell_1$-stability.)} The substep map is bounded on the
nonnegative cone, $\norm{\bvec p^{(1)}}_1\le(1+C\Delta t(h_x+h_y))\norm{\bvec
p^{\,n}}_1$, so it contributes a factor $1+O(\Delta t\,h)$ to the stability
products in \eqref{eq:fan}, consistent with \cref{ass:stab}.
\end{enumerate}
\end{proposition}

\begin{proof}
\mbox{}\hfill

\emph{(a)} The step is nonnegative when both halves of \eqref{eq:cnstep} are. The explicit half produces $\bvec b=(\mathcal I+\tfrac{\Delta t}{2}A_{xy})\bvec p^{\,n}$; since the cross stencil is sign-indefinite, $\bvec b\ge0$ holds for $\bvec p^{\,n}\ge0$ on a window $\Delta t\le\Theta_0$, with $\Theta_0$ of order $h_xh_y/(\rho\bar w^2)$ on resolved data (the diagonal-dominant bulk $p^{\,n}_{ij}$ dominates the off-diagonal cross contributions of size $\Delta t\,\rho\bar w^2 p^{\,n}/(h_xh_y)$).

For the implicit half, \cref{prop:schemeA},(ii) (resp.\ \cref{prop:schemeB} on log-Lipschitz data) gives, for $\bvec b\ge0$ and $\beta\ge2(\sqrt\rho\,\bar w+\sqrt{h_xh_y/\Delta t})$ -- i.e., once $\Delta t\le\Theta_1:=h_xh_y/(\,\tfrac14\beta^2-\rho\bar w^2)$ -- a Picard right-hand side that is nonnegative at $k=0$, after which the two triangular sweeps return $\bvec p^{[1]}\ge0$.

The contraction estimate \cref{prop:schemeA},(iv) gives geometric convergence with factor $q\le 4h_xh_y/(\beta^2\Delta t)<1$, and the limit $\bvec p^{(1)}$ inherits nonnegativity because the running right-hand side stays nonnegative (verified entrywise at $O(N_xN_y)$ cost per iterate, \cref{rem:pos-iter}). Taking $\Theta=\min(\Theta_0,\Theta_1)$ yields $\bvec p^{(1)}\ge0$.

\emph{(b)} Write $\bm 1^\top A_{xy}=\bm r^\top$, where by \cref{lem:cons} the
interior entries of $\bm r$ vanish and only the one-sided edge and corner rows
contribute; a Taylor expansion of the flux closure gives $\norm{\bm r}_\infty\le
C\rho\bar w(h_x+h_y)/(h_xh_y)$ supported on $O(N_x+N_y)$ nodes, whence $\bm 1^\top
A_{xy}\bvec p\le C\rho\bar w(h_x+h_y)\norm{\bvec p}_1$. The trapezoidal factor is a
rational function $r(z)=(1+z/2)/(1-z/2)$ with $r(0)=1$ and $r'(0)=1$, so $\bm
1^\top\Phi_{xy}(\Delta t)=\bm 1^\top+\Delta t\,\bm r^\top+O(\Delta t^2\norm{\bm
r})$, giving the stated bound with leading term $\Delta t\,\bm r^\top\bvec
p^{\,n}$.

\emph{(c)} On the nonnegative cone $\norm{\bvec v}_1=\bm 1^\top\bvec v$, so
\textup{(a)} and \textup{(b)} give $\norm{\bvec p^{(1)}}_1=\bm 1^\top\bvec
p^{(1)}\le(1+C\Delta t(h_x+h_y))\bm 1^\top\bvec p^{\,n}$.
\end{proof}

\begin{myremark}[No eventual positivity for the cross block] \label{rem:no-eventual-pos}

The contrast with the one-dimensional theory of \cref{ssec:em} is essential and is
not a deficiency of the proof.  The one-sided second-order operators $\mathcal
A^{\mathrm F}_2,\mathcal A^{\mathrm B}_2$ carry off-diagonal entries of \emph{both}
signs ($+4/(2h)$ and $-1/(2h)$), so their product $A_{xy}=\rho\,\mathcal A^{\mathrm
F}_{2,x}\mathcal A^{\mathrm B}_{2,y}$ has mixed-sign off-diagonal entries that no
shift $s$ can render nonnegative: $sI+A_{xy}$ is never Metzler.

Moreover $A_{xy}$ is strongly non-normal: being the product of an upper- and a
lower-triangular matrix it is defective, with a single repeated real eigenvalue
(\cref{sec:strang}) but no basis of eigenvectors, and (in contrast to the
upwind-leaning one-dimensional generator $L$, whose retained superdiagonal renders
its graph strongly connected with a dominant real ground state,
\cref{sec:diagfrog}) its rightmost eigenvector is not entrywise of one sign; the
Perron--Frobenius hypothesis~(H) fails for $A_{xy}$, and with it the
eventual-positivity mechanism of \cref{prop:emmatrix,cor:eventualpos}.  This is
why the two-dimensional positivity guarantee is necessarily \emph{conditional} and
is carried by the factorized solver of \cref{sec:Lxy} rather than by an
eventual-positivity theorem: the ``Zeno'' mechanism of \cref{sec:diagfrog} applies
to the diagonal one-dimensional sub-operators, whose Kronecker lifts $A_x,A_y$
inherit it \eqref{eq:posKron}, but not to the mixed block. It is precisely this
absence of a stability bound for the bare exponential that motivates the
\emph{implicit} treatment \eqref{eq:central}: the trapezoidal factor is spectrally
stable on the real-negative spectrum of $A_{xy}$ even though the exponential is
not (\cref{sec:strang,fig:stress}).
\end{myremark}

\subsection{Positivity and conservativeness of the norm} \label{sec:posNorm}

Having established positivity of the central substep (\cref{thm:zeno2d}), we now
turn to the two remaining structural properties of the full Strang step: that it
maps the nonnegative cone into itself, and that it conserves discrete mass. The
conditional $\ell_1$-stability (on the positivity regime) and the second-order
accuracy underpinning these statements are proved separately in \cref{app1}.

\begin{proposition}\label{prop:strangpos}
Suppose that, for the step size $\Delta t$ at hand, each factor of
\eqref{eq:strang} is nonnegative: $e^{\frac{\Delta t}{2}A_x^{n}}\ge0$,
$e^{\frac{\Delta t}{2}A_y^{n}}\ge0$, and the central substep
$\bvec p\mapsto \Phi_{xy}(\Delta t)\,\bvec p\ge0$.  Then the Strang update
\eqref{eq:strang} maps $\bvec p^n\ge0$ to $\bvec p^{n+1}\ge0$.
\end{proposition}

\begin{proof}
A product of nonnegative maps applied to a nonnegative vector is nonnegative;
each factor is nonnegative by hypothesis and $\bvec p^n\ge0$ by induction.
\end{proof}

Nonnegativity of the outer factors $e^{\frac{\Delta t}{2}A_\alpha^{n}}$ follows
from \cref{prop:emmatrix,prop:mmatrix2} applied to the one-dimensional operators
$A_\alpha=C_\alpha+D_\alpha$, once $\Delta t$ exceeds the corresponding
eventual-positivity thresholds \eqref{eq:posKron}.  Nonnegativity of the central
substep is the conditional guarantee of \cref{thm:zeno2d}, delivered by the
factorized construction of \cref{sec:Lxy}; as emphasised in
\cref{rem:no-eventual-pos}, the cross operator admits no eventual-positivity
property of its own, so this is the one factor whose nonnegativity is conditional
on the mesh and the step size rather than on a threshold alone.

\begin{proposition}[Conservation and $\ell_1$-stability of the split scheme]
\label{prop:conserv}

Suppose the 1D discretizations of \cref{sec:disc1d} are in discrete
divergence form, $\bvec 1^\top L_x=0$, $\bvec 1^\top L_y=0$, and likewise $\bvec
1^\top A_{xy}=0$. Then
\begin{enumerate}
\item[\textup{(i)}] $\bvec 1^\top A_x=\bvec 1^\top A_y=0$, since $\bvec
1_{n_xn_y}^\top(L_x\otimes I_{n_y}) =(\bvec 1^\top L_x)\otimes\bvec 1^\top=0$;

\item[\textup{(ii)}] the Strang propagator $\mathcal S(\Delta t)=
e^{\frac{\Delta t}{2}A_x}e^{\frac{\Delta t}{2}A_y}
\Phi_{xy}(\Delta t)
e^{\frac{\Delta t}{2}A_y}e^{\frac{\Delta t}{2}A_x}$
of \eqref{eq:strang} satisfies
$\bvec 1^\top\mathcal S(\Delta t)=\bvec 1^\top$ for \emph{every}
$\Delta t>0$, since $\bvec 1^\top C_\alpha=
\bvec 1^\top D_\alpha=\bvec 1^\top A_{xy}=0$ in discrete divergence form and
$\Phi_{xy}$ is a rational function of $A_{xy}$ with $\Phi_{xy}(0)=\mathcal I$
(\cref{cor:mass-rat}):
discrete mass $h_xh_y\,\bvec 1^\top\bvec p^{\,n}$ is conserved exactly,
irrespective of the splitting error and of the positivity thresholds; this is
exact whenever the cross stencil is closed conservatively so that
$\bvec 1^\top A_{xy}=0$ holds including the edge and corner rows
(\cref{rem:mixedmass}), and holds up to the $O(\Delta t\,h)$ boundary defect of
\cref{thm:zeno2d}\,(b) for the plain one-sided closure;

\item[\textup{(iii)}] if in addition $\Delta t$ lies in the regime where the
outer factors are nonnegative (substeps above the thresholds of
\cref{sec:disc1d}, cf.\ \cref{prop:emmatrix,prop:mmatrix2}) \emph{and} the central
substep is nonnegative (the conditional guarantee of \cref{thm:zeno2d}, $\Delta
t\le\Theta$), then $\mathcal S(\Delta t)$ is column-stochastic; hence $\bvec
p^{\,0}\ge0$, $\bvec 1^\top\bvec p^{\,0}=1$ imply the same for all $\bvec p^{\,n}$,
and $\norm{\mathcal S(\Delta t)\bvec u}_1\le\norm{\bvec u}_1$ for all $\bvec u$, so
the scheme is $\ell_1$-stable in this regime, where the $\ell_1$ norm is defined
in \cref{app1}.  We stress that, in contrast to (ii), part~(iii) is genuinely
conditional: the cross operator carries no eventual-positivity property
(\cref{rem:no-eventual-pos}), so column-stochasticity of $\mathcal S(\Delta t)$
holds on a step-size window rather than for all $\Delta t$.
\end{enumerate}
\end{proposition}

\begin{proof}
See \cref{appProof2}
\end{proof}

\begin{myremark}[Approximate matrix exponentials] \label{rem:approxexp}

Part~(ii) is stated for exact matrix exponentials. In practice a substep is
computed as $r(\Delta t A)\bvec v$ for a polynomial or rational approximant
$r\approx\exp$. By \cref{lem:divform}: $\bvec 1^\top A=0$, then for a polynomial $r$ one has $\bvec 1^\top r(\Delta t A)=r(0)\,\bvec 1^\top$. And for a rational $r=P/Q$ with $Q(\Delta t A)$ nonsingular, from $\bvec 1^\top Q(\Delta t A)=Q(0)\,\bvec 1^\top$ it follows that $\bvec 1^\top Q(\Delta t A)^{-1} = Q(0)^{-1}\bvec 1^\top$, so again $\bvec 1^\top r(\Delta t A)=r(0)\,\bvec 1^\top$.

Thus discrete mass is preserved \emph{exactly} by any approximant normalised so
that $r(0)=1$ (in particular by all Pad\'e approximants and by Krylov methods
based on them) and conservation does not degrade with the accuracy of the
exponential approximation. When the inner exponential is evaluated through an
eigendecomposition, the propagator is exact in exact arithmetic and the question
reduces to roundoff.
\end{myremark}

\begin{myremark}[Mass functional versus $\ell_1$-norm] \label{rem:massnorm}

The conserved quantity in part~(ii) is the \emph{linear functional} $\bvec
p\mapsto h_xh_y\,\bvec 1^\top\bvec p$, which coincides with $h_xh_y\|\bvec p\|_1$
only on the nonnegative cone. Conservation alone therefore does not control the
$\ell_1$-norm: for substeps below the positivity thresholds of
\cref{sec:disc1d} the numerical solution may develop negative lobes whose
contributions cancel in $\bvec 1^\top\bvec p$ while $\|\bvec p\|_1$ grows. It is
the conjunction of conservation~(ii) and positivity~(iii) that makes $\mathcal
S(\Delta t)$ a discrete Markov (column-stochastic) operator and yields stability;
neither property implies the other.
\end{myremark}

On the step-size window of \cref{thm:zeno2d}\,(a) the central substep
$\Phi_{xy}(\Delta t)$ is nonnegative, and by \cref{cor:mass-rat,thm:zeno2d}\,(b)
its column sums are $1+O(\Delta t\,h)$.  Hence on the nonnegative cone the central
substep satisfies $\norm{\Phi_{xy}(\Delta t)\bvec p}_1=(1+O(\Delta t\,h))
\norm{\bvec p}_1$: it is $\ell_1$-nonexpansive up to the boundary leakage, the
natural norm for probability densities, and contributes a factor $1+O(\Delta t\,h)$
to the stability products in \eqref{eq:fan}, consistent with \cref{ass:stab}
(whose constant $\omega$ absorbs the $O(h)$ term).  We emphasise that this is an
$\ell_1$-on-the-cone statement and not an $\ell_2$ one: because $A_{xy}$ is
strongly non-normal, the induced $\ell_2$ (and indeed $\ell_1$) \emph{operator}
norm of $\Phi_{xy}$ exceeds unity and grows under refinement, so no uniform-in-$h$
$\ell_2$ contraction holds (\cref{app1}); the cone restriction is what makes the
column-stochastic Markov bound available.

\begin{corollary}[Mass conservation of rational substeps]\label{cor:mass-rat}

Let $M$ be any matrix with $\bm{1}^\top M=\mathbf 0^\top$, and let $r=p/q$ be any rational function with $r(0)=1$ whose poles avoid the spectrum of $\Delta t\,M$. Then $\bm{1}^\top r(\Delta t\,M)=\bm{1}^\top$. By \cref{lem:cons} and \cref{prop:conserv},(i), this applies both to the directional operators and to the cross operator $M=A_{xy}=\rho\,\mathcal A^{\mathrm F}_{2,x}\mathcal A^{\mathrm B}_{2,y}$, \emph{provided the cross stencil is closed conservatively so that $\bm1^\top A_{xy}=0$} (\cref{rem:mixedmass}). Conversely, for the plain one-sided closure, the identity $\bm1^\top A_{xy}=\bm 0^\top$ carries the $O(\Delta t\,h)$ boundary residual discussed in \cref{thm:zeno2d}(b).

Under the conservative closure the trapezoidal central factor $\Phi_{xy}(\Delta t)=(\mathcal I-\tfrac{\Delta t}{2}A_{xy})^{-1}(\mathcal I+\tfrac{\Delta t}{2}A_{xy})$ of \eqref{eq:central} conserves mass exactly, independently of the time step and of the accuracy of the inner factorized solve, provided each inner solve is iterated to convergence (the partial-fraction/triangular sweeps preserve the left null vector $\bm 1$ exactly).
\end{corollary}

\begin{proof}
$\bm{1}$ is a left eigenvector of $M$ associated with the eigenvalue
$0$, hence $\bm{1}^\top g(\Delta t\,M)=g(0)\,\bm{1}^\top$ for every
polynomial $g$.  Applying this to $q$ and to $p=q\,r$,
\begin{equation}
q(0) \, \bm{1}^\top r(\Delta t\,M) = \bm{1}^\top q(\Delta t\,M)\,r(\Delta t\,M)
= \bm{1}^\top p(\Delta t\,M) = p(0)\,\bm{1}^\top,
\end{equation}
so $\bm{1}^\top r(\Delta t\,M)=r(0)\,\bm{1}^\top=\bm{1}^\top$.
\end{proof}

\Cref{cor:mass-rat} holds for the \emph{direct} (partial-fraction) evaluation of
$r(\Delta t\,A_{xy})\bvec p^{\,n}$, i.e. when the shifted systems are solved as
in \eqref{eq:picard} and the results recombined.  If the action of the
exponential is instead approximated by Galerkin projection onto a Krylov
subspace, conservation holds only up to the projection error; for the
mixed-derivative substep we therefore use the partial-fraction form.

\myparagraph{Boundary conditions.}

The hypothesis $\bvec 1^\top L=0$ is a statement about \emph{all} rows, including
the boundary ones, and is equivalent to a discrete zero-flux (reflecting)
boundary treatment. Indeed, if $L$ is assembled in flux form, $(L\bvec
p)_i=\bigl(F_{i+1/2}(\bvec p)-F_{i-1/2}(\bvec p)\bigr)/h$, then the column sums
telescope, $\bvec 1^\top L\bvec p=\bigl(F_{n+1/2}(\bvec p)-F_{1/2}(\bvec
p)\bigr)/h$, and vanish identically iff the numerical fluxes through the domain
boundary are set to zero.

If instead the computational domain is a truncation of $\R^2$ with homogeneous Dirichlet (absorbing) conditions, mass is lost at the rate of the discrete boundary flux, and~(ii) holds only up to this leakage - exponentially small in the domain size when the underlying density has Gaussian-type tails, but not zero. We emphasise that this is a property of the truncated continuous problem, not of its discretization, and therefore cannot be repaired at the discrete level without changing the boundary condition. For instance, a ghost-point closure restores $\bvec 1^\top L=0$ precisely when the ghost values are chosen to annihilate the numerical boundary flux, which amounts to replacing the Dirichlet condition by the reflecting one.

When absorption at the boundary is part of the model, exact conservation can instead be recovered in an extended sense by appending a cemetery state accumulating the boundary flux,
\begin{equation*}
\widetilde L = \left(
\begin{matrix}
L & 0 \\
\bvec f^\top & 0
\end{matrix}
\right)
\end{equation*}
with $\bvec f^\top=-\bvec 1^\top L \ge 0$, for which $\bvec 1^\top\widetilde L=0$ and the conserved quantity is the total of surviving and absorbed probability. The scheme of \cref{sec:disc1d} uses the reflecting convention, which we assume here.

The construction requires $\bvec f\ge 0$, i.e.\ sub-stochastic columns of the absorbing propagator; this holds for flux-form closures with Metzler boundary rows, but should be verified when boundary rows are only of EM type, in which case nonnegativity of $e^{t\widetilde L}$ is again merely eventual.

\begin{myremark}\label{rem:fallback}

A simpler, first-order alternative to the trapezoidal factor \eqref{eq:central} is to realise the central factor by the \emph{backward-Euler resolvent} $(\mathcal I-\Delta t A_{xy})^{-1}$ (the $\gamma=1$, $\bvec b=\bvec p^{\,n}$ choice in \eqref{eq:pade01}), computed by the same factorized iteration of \cref{sec:Lxy}. On the step-size window of \cref{thm:zeno2d},(a) the factorized solve returns a nonnegative, asymptotically mass-conserving substep, so the resulting composition is positive and $\ell_1$-stable there; but it is only first-order accurate in time, since $(\mathcal I-\Delta t A_{xy})^{-1}-e^{\Delta t A_{xy}} =\tfrac12\Delta t^2A_{xy}^2+O(\Delta t^3)$ on smooth data.

Note that this defect is of a different nature from the ill-posedness of the bare exponential: the backward-Euler and trapezoidal resolvents are both stable on the real-negative spectrum of $A_{xy}$, and the trapezoidal choice simply matches one more term of the exponential. We use the backward-Euler variant only as a robust first-order fallback (e.g. for start-up steps in place of Rannacher smoothing); the trapezoidal central factor \eqref{eq:central} is what delivers the second-order claim of \cref{prop:order2}.
\end{myremark}

\begin{myremark}[The mixed-derivative operator]\label{rem:mixedmass}

Whether the cross term is written in the centred flux form
\begin{equation*}
\bigl(A_{xy}\bvec p\bigr)_{ij} = \frac{2}{4h_xh_y}\Bigl[ (\Sigma_{xy}p)_{i+1,j+1}-(\Sigma_{xy}p)_{i+1,j-1} -(\Sigma_{xy}p)_{i-1,j+1}+(\Sigma_{xy}p)_{i-1,j-1}\Bigr]
\end{equation*}
or, as used here, as the one-sided product $A_{xy}=\rho\,\mathcal A^{\mathrm
F}_{2,x}\mathcal A^{\mathrm B}_{2,y}$, interior nodes have
zero column sums by the same telescoping argument (\cref{lem:cons}), applied in
each index; the property must be verified separately for the one-sided stencils
used along edges and at corners, where naive modifications are the most common
source of mass leakage. We emphasise the asymmetry between conservation and
positivity in this respect: the conservation identity \cref{prop:conserv}\,(ii)
applies to any rational function of $A_{xy}$ with $r(0)=1$ unchanged, whereas
$e^{\Delta t A_{xy}}$ is in general \emph{not} nonnegative for any $\Delta t>0$
(a pure mixed-derivative generator has no Metzler or eventually-nonnegative
structure, \cref{rem:no-eventual-pos}), so the positivity statement
\cref{prop:conserv}\,(iii) requires the factorized implicit treatment of the
cross term (\cref{sec:Lxy,thm:zeno2d}) and holds only conditionally, on a
step-size window.
\end{myremark}

\section{The exponential time integration scheme} \label{sec:schemes}

Throughout this paper we compare the DF scheme with several popular time-integration FD schemes. A brief review of theses schemes in given in \cref{appTI}. Here, we discuss the exponential time integrator used in the DF scheme, which is the main subject of the paper. Its construction was given above, and its efficient implementation is presented in \cref{sec:krylov}. Below we record the integrator properties along the same axes used for classical schemes in \cref{appTI}.

Let $k = \Delta t$ denote the time step.  The numerical solution advances as
$\bvec{p}^{n+1} = M\,\bvec{p}^n$ for some propagator matrix $M$.  The amplification factors are $e^{k\lambda_j}$, with $|e^{k\lambda_j}|=e^{k\re\lambda_j}\to0$ in the stiff limit. This contrasts with the trapezoidal rule, whose stability function tends to $-1$. Consequently, no mode is reflected at any step size, so the integrator exhibits no temporal ringing regardless of $k$.

Positivity holds for all $k>0$ in the M-matrix case and for $k\ge\tau_0$ in the EM case \eqref{eq:posKron}. This yields an inverted CFL condition: a lower bound on the step with no ceiling. No rational one-step scheme achieves this at second order (Table~\ref{tab:schemes}, \cite{BolleyCrouzeix1978}).

The simple eigenvalue $\lambda=0$ (Perron, by irreducibility and $\bvec 1^\top A=0$, \cref{lem:divform}) carries the conserved mass. All remaining eigenvalues satisfy $\re\lambda_j<0$ (by Perron--Frobenius in the M-matrix case; by the spectral \cref{hypH} in the EM case). Hence the scheme converges to the discrete stationary density $\bvec\pi$.

Finally, forming $e^{kA}$ explicitly for large $n$ would require $\mathcal
O(n^{3})$ operations. However, time stepping needs only the action $e^{kA}\bvec
p^{n}$. Using the Krylov method of \cref{sec:krylov}, this action is computed in
$\mathcal O(m^2 n + m^3)$ operations via $m$ shifted banded solves and modified
Gram--Schmidt orthogonalisation, with $m\ll n$.

\myparagraph{Comparison of time integrators.}
\begin{table}[!htb]
\centering
\caption{Comparison of time propagators for the semi-discrete Fokker--Planck system $\dot{\bvec p}=A\bvec p$ with $\bvec 1^\top A=0$. Here $b_i=A_{ii}$, $\beta=\max_i|b_i|=O(h^{-2})$; ``M'' denotes the Metzler (negated M-matrix) case and ``EM'' the eventually-nonnegative case of  \cref{sec:disc1d}; $k_0,\tau_0$ are the finite eventual-positivity thresholds; $\bvec\pi$-limit hypotheses as in \cref{thm:TRBDF2pos}.}
\label{tab:schemes}
\renewcommand{\arraystretch}{1.35}
\scalebox{0.95}{
\begin{tabular}{@{}lcccc@{}}
\toprule
 & BE & CN & TR-BDF2 ($\gamma=2-\sqrt2$) & $e^{\Delta tA}$ (exact/Krylov) \\
\midrule
Order in time
 & 1 & 2 & 2 & exact$^{\,a}$ \\
Stiff damping $r(-\infty)$
 & $0$ (L-stable) & $-1$ (A-stable only) & $0$ (L-stable) & $0$ \\
Temporal ringing
 & none & yes, for $k\beta\gg1$$^{\,b}$ & none & none \\
Positivity, M case
 & all $k>0$ & $k\le 2/\beta$ & $k\le(1+\sqrt2)/\beta$ & all $\Delta t>0$ \\
Positivity, EM case
 & $k\ge k_0$$^{\,c}$ & none$^{\,d}$ & $k\ge k_0'$$^{\,c}$ & $\Delta t\ge\tau_0$ \\
Mass conservation
 & exact, all $k$ & exact, all $k$ & exact, all $k$ & exact$^{\,e}$ \\
$\ell_1$-contraction$^{\,f}$
 & where positive & where positive & where positive & where positive \\
Cost per step
 & 1 solve & 1 solve & 2 solves$^{\,g}$ & eigendecomp.\ / Krylov \\
\bottomrule
\end{tabular}
}

\medskip
\begin{minipage}{0.95\linewidth}\footnotesize
$^{a}$\,Exact in time for the frozen-coefficient substep; the $O(\Delta t^2)$
global error of \cref{prop:order2} comes from splitting and freezing
only.
\\
$^{b}$\,Within the positivity window, $\mu_j\ge-\tfrac13$ and oscillatory
modes decay by a factor $\ge3$ per step; persistent ringing ($\mu_j\approx-1$)
occurs only for $k$ outside the window.
\\
$^{c}$\,Eventual positivity at \emph{large} steps via decay of subdominant modes and dominance of the Perron projection $\bvec\pi\bvec1^\top/(\bvec1^\top\bvec\pi)$; requires $r(-\infty)=0$, hence fails for CN.
\\
$^{d}$\,Fails for all small $k$ (Neumann expansion, $M=I+kA+O(k^2)$) and for all large $k$ ($r(-\infty)=-1$); at most an intermediate window, not guaranteed (\cref{rem:CNEM}).
\\
$^{e}$\,For approximate exponentials, exact whenever the approximant satisfies $r(0)=1$ (Pad\'e, Krylov); \cref{rem:approxexp}.
\\
$^{f}$\,$\norm{M\bvec u}_1\le\norm{\bvec u}_1$ holds exactly when $M$ is column-stochastic, i.e.\ in the conjunction of the positivity and conservation rows (\cref{prop:conserv}).
\\
$^{g}$\,Two linear solves with the \emph{same} matrix $I-\tfrac{\gamma k}{2}A$, so one factorization per step size.
\end{minipage}
\end{table}

Table~\ref{tab:schemes} summarizes the time integrators for the FPE discussed in this section. Since every scheme conserves mass unconditionally, the primary discriminating factors are positivity and stiff damping. Along these axes, TR-BDF2 strictly dominates CN (offering a wider window, L-stability, and EM-recoverability, at the cost of a second solve with the same factorized matrix). Meanwhile, the exponential integrator is the only propagator whose positivity in the EM regime comes with a threshold $\tau_0$ that is independent of the time-stepping error. In other words, this threshold is determined solely by the matrix, rather than by an interaction between $r$ and the spectrum.

Note that the thresholds $k_0$ (BE), $k_0'$ (TR-BDF2), $\tau_0$ (exponential) are all distinct numbers.

\myparagraph{Stationary density and long-time behavior.}
The exact propagator of the full generator fixes its steady state: if
$A\bvec{p}^*=\bvec 0$ then $e^{\Delta t A}\bvec{p}^*=\bvec{p}^*$, and a Krylov
approximation preserves this once $\bvec{p}^*$ lies in $\Krylov_m(A,\bvec{p}^*)$
(immediately, since $A\bvec{p}^*=\bvec 0$ gives $\Krylov_1=\spn\{\bvec{p}^*\}$
and $H_1=0$). For the one-dimensional FPE \eqref{eq:fpe} with drift $\mu(x)$ and
diffusion $\sigma^2(x)$ the stationary density satisfies
\begin{equation*}
0=-\tfrac{d}{dx}[\mu p^*]+\tfrac12\tfrac{d^2}{dx^2}[\sigma^2 p^*],
\end{equation*}
with the zero-flux solution $p^*(x)\propto \sigma^{-2}(x)\exp\!\big(2\!\int_0^x
\mu(y)/\sigma^2(y)\,dy\big)$. The discrete steady state is the null vector of $A$
on the interior block, obtained numerically as the eigenvector of $e^{\Delta t
A}$ for eigenvalue $1$.

\subsection{Krylov Subspace Methods for the Matrix Exponential Action}
\label{sec:krylov}

In the two-dimensional scheme the central factor \eqref{eq:central} is a rational
function of $A_{xy}$ and is applied by the factorized solver of \cref{sec:Lxy},
\emph{not} by a Krylov exponential.  The Krylov method is used for the directional
factors of \eqref{eq:strang}, $e^{\frac{\Delta t}{2}A_\alpha}\bvec v$,
$\alpha\in\{x,y\}$, and -- in the absence of a mixed term -- for the single
propagator $e^{\Delta t A}\bvec v$ of the 1D scheme.  Because the directional
generator is a Kronecker lift, $A_x=L_x\otimes I_{n_y}$ (and $A_y=I_{n_x}\otimes
L_y$), its exponential acts column-by-column through the one-dimensional
exponential $e^{\frac{\Delta t}{2}L_x}$, so the Krylov work below is effectively
one-dimensional.

We evaluate the action in the polynomial Krylov subspace of dimension $m$,
\begin{equation}
  \Krylov_m(A_\alpha,\bvec{v})
  \coloneqq \spn\{\bvec{v},\,A_\alpha\bvec{v},\,\ldots,\,A_\alpha^{m-1}\bvec{v}\},
\end{equation}
which is well suited to the action for $m\ll N$: the Taylor series
$e^{\frac{\Delta t}{2}A_\alpha}\bvec{v}=\sum_{j\ge0}\tfrac{(\frac{\Delta
t}{2}A_\alpha)^j}{j!}\bvec{v}$ places the exact result in the closure of
$\bigcup_m\Krylov_m$. The Arnoldi process produces an orthonormal basis $V_m$ and
the projected matrix $H_m=V_m^{\!\top}A_\alpha V_m$, giving
\begin{equation}   \label{eq:krylov-approx}
  e^{\frac{\Delta t}{2}A_\alpha}\bvec{v}\ \approx\
  \norm{\bvec{v}}_2\,V_m\,e^{\frac{\Delta t}{2} H_m}\,\bvec{e}_1,
\end{equation}
the inner $m\times m$ exponential formed by scaling-and-squaring \cite{Higham2008}.

\myparagraph{Cost.}
Each action requires $m$ matrix--vector products with the banded $A_\alpha$ at
$O(mN)$, modified Gram--Schmidt orthogonalisation at $O(m^2N)$, and the dense
$m\times m$ exponential at $O(m^3)$, for a total of
\begin{equation}   \label{eq:exp-complexity}
  O(m^2 N + m^3),
\end{equation}
dominated by the orthogonalisation. The dimension required for a fixed accuracy
grows with the stiffness,
$m=\Theta\!\big(\sqrt{\Delta t\,\rho(A_\alpha)}\big)=\Theta(\sqrt{\Delta t}/h)$,
so \eqref{eq:exp-complexity} is superlinear in $N$; the small $m$ seen at moderate
stiffness reflects that regime and is not a linear-complexity guarantee.

\myparagraph{Interpretation, and comparison with fixed approximants.}
Restricted to $\Krylov_m$, \eqref{eq:krylov-approx} is the polynomial in
$A_\alpha$ that interpolates $e^{\frac{\Delta t}{2} z}$ at the Ritz values (the
eigenvalues of $H_m$); it is thus a near-best polynomial approximation
\emph{adapted to the spectrum} of $A_\alpha$, not a fixed approximant committed in
advance. The implicit propagators compared in \cref{sec:numerics} -- including the
trapezoidal central factor \eqref{eq:central} -- are, by contrast, fixed rational
approximants to the exponential: Crank--Nicolson is the $(1,1)$ Pad\'e approximant
and backward Euler the $(0,1)$,
\begin{equation}
  M_{\mathrm{CN}}=\big(I+\tfrac{\Delta t}{2}A\big)\big(I-\tfrac{\Delta t}{2}A\big)^{-1}
  =e^{\Delta t A}+O(\Delta t^3),
  \qquad
  M_{\mathrm{BE}}=(I-\Delta t A)^{-1}=e^{\Delta t A}+O(\Delta t^2),
\end{equation}
committed regardless of the spectrum, whereas \eqref{eq:krylov-approx} matches its
target to tolerance for any $\Delta t$.

\myparagraph{Choice of method, positivity, and mass.}
The directional generator $A_\alpha$ is a flux-form convection--diffusion operator
with an upwind-leaning first-order part, hence \emph{non-symmetric}; the Arnoldi
process is therefore used as such and does \emph{not} reduce to Lanczos.  A
shift-and-invert (rational) Krylov method \cite{GuttelRational2013} would replace
each matrix--vector product by a costlier shifted solve and brings no benefit
here: the directional spectra are mild and real-dominated, and the only operator
with a genuinely non-normal, wide spectrum -- the cross operator $A_{xy}$ -- is
handled outside Krylov altogether, by the factorized solver of \cref{sec:Lxy}.
The exact-in-time character of the exponential step does not by itself confer
discrete positivity: as in \cref{rem:resolvent-positivity}, the orthonormal
Arnoldi basis $V_m$ is sign-indefinite, so the directional substep
\eqref{eq:krylov-approx} is nonnegative only up to the Krylov truncation error.
That error is exponentially small in $m$, so on the eventual-positivity regime of
\cref{prop:emmatrix} -- where the \emph{exact} directional exponential is
nonnegative -- the computed action is nonnegative to within a tolerance that can
be driven to round-off by enlarging $m$; in practice we monitor the entrywise
minimum and increase $m$ when needed.  Discrete mass of the directional factors is
likewise preserved only up to the Galerkin projection error, whereas the central
factor \eqref{eq:central}, applied by the partial-fraction/triangular sweeps of
\cref{sec:Lxy}, conserves $\bm 1^\top\bvec p$ \emph{exactly} (\cref{cor:mass-rat}).
Where strict nonnegativity must be enforced we fall back to the resolvent maps of
\cref{sec:Lxy}; we use the exponential step for the directional factors, where its
accuracy and stability are an asset and its sign defect is negligible.

\section{Numerical experiments} \label{sec:numerics}

We validate the theory on problems with known analytical solutions, organised so that each experiment isolates one structural claim. The one-dimensional Ornstein--Uhlenbeck benchmark (\cref{ssec:exp-1d}) confirms second-order accuracy, reports the cost of the upwind stencil candidly against the centred scheme, and exhibits the eventual-positivity threshold $\tau_0$ of \cref{cor:eventualpos} as an inverted CFL condition. It also documents, in a severely under-resolved regime, the qualitatively benign and recoverable character of the upwind scheme's undershoot relative to the centred scheme's dispersive ringing.

The two-dimensional anisotropic-diffusion benchmark (\cref{ssec:exp-mixed}) then verifies the conditional positivity window $\Theta$ of \cref{thm:zeno2d} as a function of the mesh, exact mass conservation (\cref{prop:conserv}), and the necessity of treating the mixed term \emph{implicitly} rather than through its exponential. Furthermore, \cref{ssec:exp-kin} reports the factorized Picard iteration count and its linear, mesh-robust cost, substantiating the estimate of \cref{rem:time-order}. The supporting Python code is available at 
\href{https://github.com/itkinal/DFsolver}{Github}. 

% =====================================================================
%  1D numerical experiments (Ornstein-Uhlenbeck) for sec:numerics.
%  Honest framing: validation + demonstration of the eventual-positivity
%  mechanism. NOT a horse-race the upwind scheme wins -- on this benign 1D
%  problem the centred scheme is more accurate and cheaper; the upwind
%  (EM-matrix) construction earns its place through eventual positivity and,
%  decisively, as the building block for the 2D mixed term (ssec:exp-mixed),
%  where no centred analogue is an M-matrix at any mesh.
%  Figures in figs/: E1d_positivity.pdf, E1d_eventual.pdf, E1d_stationary.pdf.
% =====================================================================

\subsection{One-dimensional benchmark: the Ornstein--Uhlenbeck process}
\label{ssec:exp-1d}

We validate the one-dimensional construction of \cref{sec:disc1d} on the
Ornstein--Uhlenbeck (OU) process, whose Fokker--Planck equation
\begin{equation}\label{eq:ou}
\partial_t p = \kappa\,\partial_x\!\bigl[(x-m)\,p\bigr] + \tfrac12\sigma^2\,\partial_{xx}p, \qquad \mu(x)=-\kappa(x-m),\quad D=\tfrac12\sigma^2,
\end{equation}
has the closed-form Gaussian transition density $\mathcal N(\bar m(t),v(t))$,
$\bar m(t)=m+(x_0-m)e^{-\kappa t}$,
$v(t)=\tfrac{\sigma^2}{2\kappa}(1-e^{-2\kappa t})$, and stationary law
$p_\infty=\mathcal N(m,\sigma^2/2\kappa)$ -- an exact reference for accuracy,
positivity, and relaxation. The drift $\mu(x)$ changes sign at $x=m$ and the
local P\'eclet number $\mathrm{Pe}_i=|\mu_i|h/D$ grows linearly outward, so a
truncated domain always contains an advection-dominated outer region. We compare
two second-order spatial discretizations: the Diagonal-Frog (DF) scheme --
second-order upwind advection \eqref{eq:adv} and centred diffusion
\eqref{eq:diff_ctr}, for which $-L$ is an EM-matrix (\cref{prop:emmatrix}) -- and
the second-order centred scheme of \cref{ssec:advdom}, an M-matrix for
$\mathrm{Pe}<2$ (\cref{prop:mmatrix2}). Both advance in time by the exact propagator action $e^{\Delta t L}\bvec p$ (the high-accuracy \texttt{expm\_multiply} action of the Krylov integrator of \cref{sec:krylov}), so spatial properties are compared in isolation. A first-order upwind scheme is included as an accuracy
baseline.

We state the outcome plainly, since it calibrates the role of the upwind
construction. \emph{On this one-dimensional problem the centred scheme is the
better choice}: it is more accurate (smaller error constant, \cref{tab:conv1d}),
cheaper to advance (a prefactored Crank--Nicolson solve is hard to beat in 1D;
see the remark on cost below), and, once the mesh resolves
the solution so that $\mathrm{Pe}<2$, it is a genuine M-matrix and hence
unconditionally positive. The upwind scheme is not introduced to beat the
centred scheme in 1D; its purpose is to supply the \emph{eventual-positivity}
mechanism that survives where the centred construction has no positive
analogue -- namely in the advection-dominated regime at fixed mesh
(\cref{fig:cfl1d}) and, decisively, in the two-dimensional mixed-derivative
block of \cref{ssec:exp-mixed}, whose centred cross-stencil is an M-matrix at
\emph{no} mesh (\cref{rem:no-eventual-pos}). The experiments below establish
correctness of the 1D scheme and exhibit that mechanism in its simplest setting.

\myparagraph{Order of convergence.}
\Cref{tab:conv1d} reports the error against the exact transition density in the
scaled discrete $L^2$ norm (mild regime $\kappa=1$, $\sigma=1$). DF and the
centred scheme are both second order (rate $\to2.0$); the centred scheme's error
constant is about four times smaller, reflecting the absence of the
upwind scheme's numerical diffusion. The first-order upwind scheme converges at
rate~$1$. Mass is conserved throughout to the boundary-truncation level.

\begin{table}[!htb]
\centering
\small
\begin{tabular}{@{}cccp{6pt}ccp{6pt}cc@{}}
\toprule
& \multicolumn{2}{c}{DF (2nd-order upwind)} & &
  \multicolumn{2}{c}{centred (2nd order)} & &
  \multicolumn{2}{c}{upwind (1st order)}\\
\cmidrule(r){2-3}\cmidrule(lr){5-6}\cmidrule(l){8-9}
$n$ & $\norm{e_h}_2$ & rate & & $\norm{e_h}_2$ & rate & & $\norm{e_h}_2$ & rate\\
\midrule
$101$ & $6.57\!\times\!10^{-3}$ & ---  & & $1.73\!\times\!10^{-3}$ & ---  & & $2.47\!\times\!10^{-2}$ & ---  \\
$201$ & $1.66\!\times\!10^{-3}$ & $1.98$ & & $4.32\!\times\!10^{-4}$ & $2.00$ & & $1.31\!\times\!10^{-2}$ & $0.91$ \\
$401$ & $4.18\!\times\!10^{-4}$ & $1.99$ & & $1.08\!\times\!10^{-4}$ & $2.00$ & & $6.77\!\times\!10^{-3}$ & $0.95$ \\
$801$ & $1.05\!\times\!10^{-4}$ & $1.99$ & & $2.70\!\times\!10^{-5}$ & $2.00$ & & $3.44\!\times\!10^{-3}$ & $0.98$ \\
\bottomrule
\end{tabular}
\caption{Convergence to the exact OU transition density ($\kappa=1$, $\sigma=1$,
$T=0.5$, $\Delta t=0.4\,h$). Both second-order schemes attain rate~$2$; the
centred scheme has the smaller error constant.}
\label{tab:conv1d}
\end{table}

\myparagraph{A remark on cost.}
We deliberately make no run-time comparison in one dimension, because none would
be informative. In 1D the centred operator is tridiagonal and the upwind operator
adds a single band (a $4{:}3$ nonzero ratio), so a Crank--Nicolson step, i.e., a prefactored banded solve costing $O(N)$ with a small constant, is extremely
cheap and, at the modest step counts needed for engineering accuracy, is faster than any exponential integrator we tried: the stiff, strongly non-normal upwind operator requires a comparatively large Krylov subspace and offers no one-dimensional speed advantage.

The favourable scaling of the splitting is a higher-dimensional phenomenon. A direct factorisation of a generator carrying mixed-derivative couplings fills in
and scales superlinearly, whereas the factorized Picard solve of the central
factor (\cref{sec:Lxy,rem:time-order}) uses only one-dimensional band factors,
with a small, mesh-robust iteration count and linear per-step cost
(\cref{tab:picard}). We therefore
present the cost comparison in the two-dimensional setting, where it is
meaningful, and treat the one-dimensional scheme here purely as a structural
building block.

\myparagraph{Eventual positivity and the inverted CFL.}
The property the upwind scheme does possess, and the centred scheme does not, is
\emph{eventual positivity}: by \cref{cor:eventualpos} the EM-matrix propagator
$e^{\Delta t L}$ is nonnegative not for small steps but for $\Delta t\ge\tau_0$,
an inverted CFL condition. \Cref{fig:cfl1d} measures $\min_x(e^{\Delta t L}\bvec
p_0)$ for a resolved bump in a strongly advective regime ($\kappa=6$,
$\sigma=0.5$, $\max\mathrm{Pe}\approx7.7$). For DF the minimum is slightly
negative in a short initial transient -- the Godunov ripple of the second-order
upwind stencil (\cref{rem:godunov}) -- and becomes nonnegative once $\Delta t$
exceeds $\tau_0\approx10^{-4}$, after which it stays nonnegative for arbitrarily
large steps: the strictly positive Perron projection $e^{\lambda_1\Delta t}P$
overtakes the oscillatory remainder once the latter has decayed
(\cref{ssec:em}). The centred scheme recovers positivity only near
$\Delta t\approx0.2$ -- three orders of magnitude later -- because at
$\mathrm{Pe}>2$ it is not even an EM-matrix; its small undershoot at this fixed
mesh ($\min_x p\approx-9\times10^{-7}$ in the run of \cref{fig:pos1d}) shrinks
under refinement and vanishes once $\mathrm{Pe}<2$, so on a resolved 1D problem
it is harmless. The contrast matters not for 1D, where one simply refines, but
because it is the mechanism the 2D mixed block inherits, where refinement does
not restore an M-matrix.

\begin{figure}[!htb]
\centering
\includegraphics[width=0.6\textwidth]{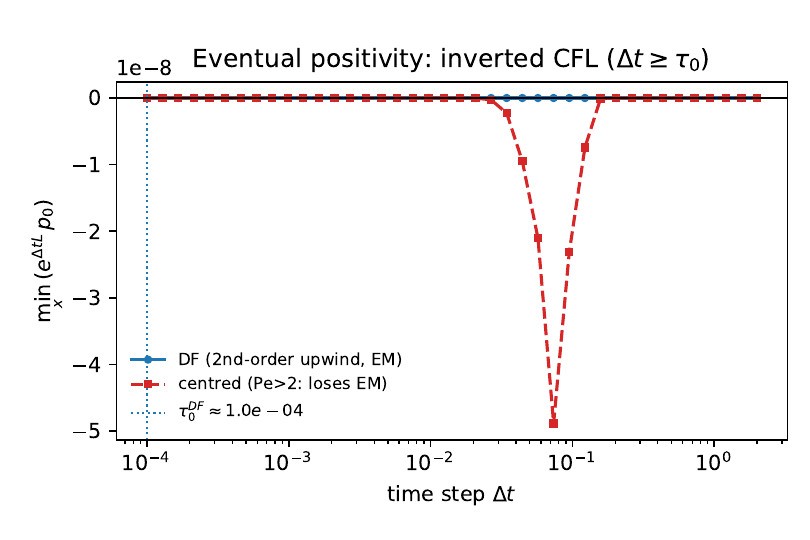}
\caption{Minimum of $e^{\Delta t L}\bvec p_0$ versus $\Delta t$
($\kappa=6$, $\sigma=0.5$, $\max\mathrm{Pe}\approx7.7$). The DF propagator (solid)
is nonnegative for $\Delta t\ge\tau_0\approx10^{-4}$ -- an inverted CFL condition
(\cref{cor:eventualpos}); the centred scheme (dashed) recovers positivity only
near $\Delta t\approx0.2$.}
\label{fig:cfl1d}
\end{figure}

\Cref{fig:pos1d} shows the corresponding density profiles after $T=0.3$ at a
fixed step $\Delta t=0.01$ in this advective regime, for reference: the centred
scheme undershoots below zero on the leading flank (by $\approx9\times10^{-7}$,
not visible at plot scale), while DF stays nonnegative. We reiterate that this
undershoot is benign in 1D and disappears under refinement; the figure documents
the sign behaviour rather than a practically significant error.

\begin{figure}[!htb]
\centering
\includegraphics[width=0.6\textwidth]{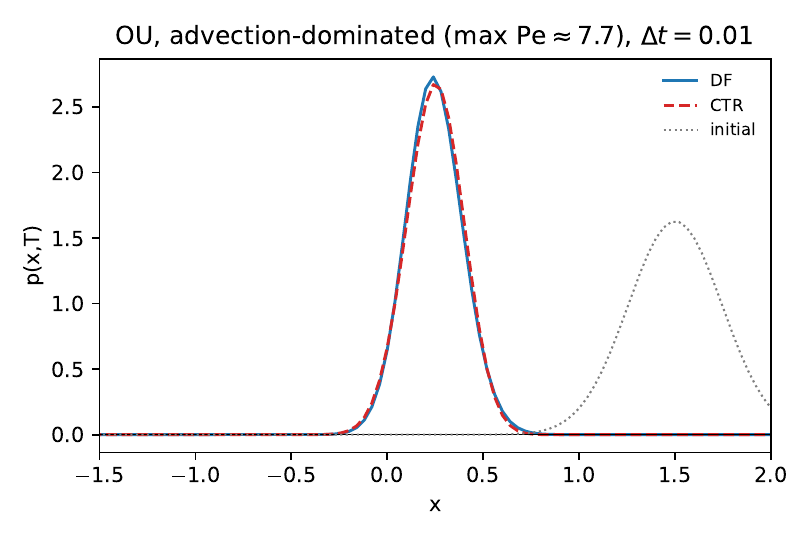}
\caption{OU density after $T=0.3$ in the advection-dominated regime
($\max\mathrm{Pe}\approx7.7$, $\Delta t=0.01$). At this fixed mesh the centred
scheme (dashed) has a small negative undershoot on the leading flank; DF (solid)
remains nonnegative. The dotted curve is the initial datum.}
\label{fig:pos1d}
\end{figure}

\myparagraph{Relaxation to the stationary law.}
Finally we verify long-time behaviour. Starting from an off-centre narrow
Gaussian, \cref{fig:stat1d} tracks $\norm{p(t)-p_\infty}_1$ under DF stepping
($\kappa=2$, $\sigma=1$, $\Delta t=0.05$). The error decays at the rate
$e^{-\kappa t}$ of the slowest non-stationary OU mode, the iterates remain
nonnegative to machine precision, and discrete mass is conserved to the
boundary-truncation level over $t\in[0,6]$. The scheme thus reproduces the
correct stationary density and spectral relaxation rate while preserving
positivity and mass at every step.

\begin{figure}[!htb]
\centering
\includegraphics[width=0.6\textwidth]{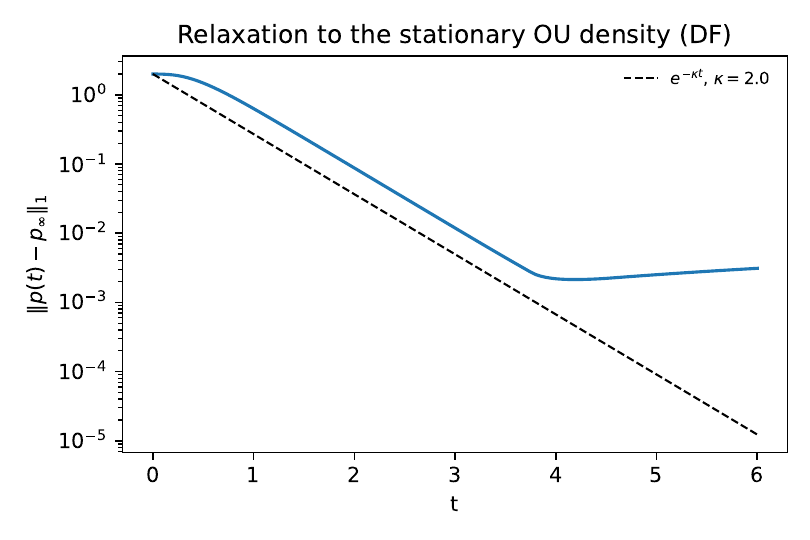}
\caption{Relaxation of the DF solution to the stationary OU density
($\kappa=2$, $\sigma=1$, $\Delta t=0.05$). The $\ell_1$ error to $p_\infty$
follows the envelope $e^{-\kappa t}$ (dashed); positivity and mass are preserved
throughout.}
\label{fig:stat1d}
\end{figure}

\myparagraph{Failure mode under severe under-resolution.}
The trade-off above is sharpened, not escaped, in a severely advection-dominated
regime. We take $\kappa=15$, $\sigma=0.3$ on $[-4,4]$ with $n=201$, so that
$\mathrm{Pe}_{\max}\approx53$, and evolve a Gaussian pulse swept toward the mean
until, at $T=0.2$, it is contracted to a width of barely one cell
($\mathrm{std}\approx1.4\,h$).

\Cref{fig:pos1d_severe} shows the result. No linear second-order scheme can be monotone here -- this is exactly Godunov's barrier (\cref{rem:godunov}) -- so the
Diagonal-Frog scheme is \emph{not} positive in this extreme: it carries a small
Godunov ripple, $\min_x p\approx-7\times10^{-3}$. The distinction is in the
\emph{character} of the failure. The centred scheme (and Crank--Nicolson on it)
produce classical dispersive ringing: a deep negative sink $\min_x p\approx-0.34$
-- two orders of magnitude larger -- spread over roughly $48$ sign-changing nodes
across the leading flank. The DF undershoot is instead \emph{localized}, confined
to $9$ nodes adjacent to the front, and is removed entirely once the step crosses
the eventual-positivity threshold of \cref{fig:cfl1d}.

The direct exponential and the polynomial-Krylov action, combined with the DF
scheme, produce the same results to plotting accuracy.

Thus even where positivity cannot be guaranteed at second order, the upwind
construction degrades gracefully -- a bounded, localized dip that the propagator
heals -- whereas the centred discretization produces global oscillations that no
choice of time step repairs. (We note that all three schemes place the peak at
the same node; the apparent phase offset between DF and the centred schemes is
the $O(h)$ numerical-diffusion lag of the upwind stencil, at most one cell at
this resolution.)

\begin{figure}[!htb]
\centering
\includegraphics[width=0.66\textwidth]{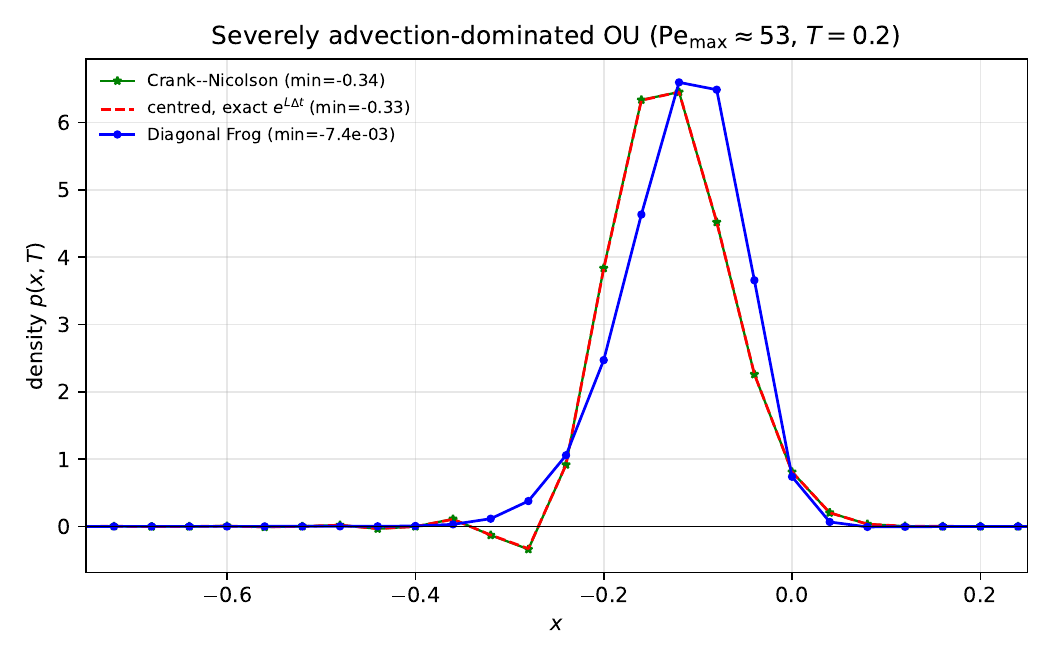}
\caption{Severely advection-dominated OU regime ($\mathrm{Pe}_{\max}\approx53$,
final pulse width $\approx1.4\,h$, $T=0.2$). The centred scheme under the exact
propagator (dashed) and Crank--Nicolson (green) exhibit dispersive ringing with a
deep negative sink ($\min_x p\approx-0.34$, $\approx48$ negative nodes); the
Diagonal-Frog scheme (blue) carries only a localized Godunov ripple ($\min_x
p\approx-7\times10^{-3}$, $9$ negative nodes), removed above the
eventual-positivity threshold. No linear second-order scheme is monotone here
(Godunov); the contrast is in the character and recoverability of the undershoot,
not its absence.}
\label{fig:pos1d_severe}
\end{figure}

\myparagraph{Same peak, different resolution and comparable cost.}
Two clarifications about \cref{fig:pos1d_severe} are in order, since the visual
impression can mislead. First, the apparent horizontal offset between the curves
is \emph{not} a phase error of the scheme: all of DF, the centred scheme, and
Crank--Nicolson place the peak at the \emph{same} grid node (\cref{tab:peak1d}).
What differs is the behaviour \emph{around} the peak --- the amplitude of the
dispersive over- and undershoot, and that difference is a matter of mesh
resolution, not of where the solution sits. As \cref{tab:peak1d} shows, when the
pulse is well resolved (regime~B, $\mathrm{std}\approx7h$) all three schemes
agree to plotting accuracy and are nonnegative, whereas under one-cell
resolution (regime~A) they share the same peak location but the centred
discretizations ring while DF does not.

\begin{table}[!htb]
\centering
\small
\begin{tabular}{@{}llccccc@{}}
\toprule
regime & scheme & peak node & $\min_x p$ & \# neg.\ nodes & TV \\
\midrule
\multirow{3}{*}{A (under-res., $\mathrm{std}\approx1.4h$)}
 & DF              & $-0.120$ & $-7.4\times10^{-3}$ & $9$  & $13.21$ \\
 & centred         & $-0.120$ & $-3.3\times10^{-1}$ & $46$ & $13.93$ \\
 & Crank--Nicolson & $-0.120$ & $-3.4\times10^{-1}$ & $48$ & $13.97$ \\
\midrule
\multirow{3}{*}{B (resolved, $\mathrm{std}\approx7h$)}
 & DF              & $-1.640$ & $+1.8\times10^{-87}$ & $0$ & $3.004$ \\
 & centred         & $-1.640$ & $-5.7\times10^{-47}$ & $0$ & $2.998$ \\
 & Crank--Nicolson & $-1.640$ & $-5.9\times10^{-47}$ & $0$ & $2.998$ \\
\bottomrule
\end{tabular}
\caption{Peak location and positivity diagnostics for the two regimes. All
schemes peak at the same node in both regimes; the negative undershoot and the
extra total variation (TV) of the centred schemes in regime~A are a
resolution effect that disappears once the pulse is resolved (regime~B).}
\label{tab:peak1d}
\end{table}

Second, the positivity advantage of DF does not come at a cost penalty relative
to Crank--Nicolson. In one space dimension the implicit primitive shared by both
families, i.e., a single banded $LU$ solve of $(\mathcal I-\gamma\Delta t\,L)$ for the DF resolvent/Picard solve, and of $(\mathcal I-\tfrac12\Delta t\,L)$ for CN, has \emph{identical} asymptotic cost: the DF stencil adds one off-diagonal
band (a $4{:}3$ nonzero ratio) but the same $O(N)$ bandwidth-limited complexity,
so the two solves time the same to within a few percent across mesh sizes
(\cref{tab:cost1d}).

Each application of the 1D exponential propagator $e^{\Delta t L}$ costs $O(m^2 N + m^3)$, with
$m = \Theta\!\big(\sqrt{\Delta t\,\rho(L)}\big) = \Theta\!\big(\sqrt{\Delta t}/h\big)$. The cost is thus superlinear in $N$, and the small $m$ observed at moderate stiffness reflects that regime rather than a linear-complexity guarantee.

The DF is therefore not more expensive per linear-algebra primitive; the only
difference is how many such solves each method needs to reach a target accuracy,
which is problem-dependent and, for the multi-dimensional problems with
mixed-derivative couplings that motivate this work, favours the exponential
approach (\cref{ssec:exp-kin}). We do not claim a one-dimensional wall-clock
advantage for the exponential integrator, because a well-tuned Crank--Nicolson code is highly competitive in 1D, only that positivity is obtained at no asymptotic
cost premium.

\begin{table}[!htb]
\centering
\small
\begin{tabular}{@{}cccc@{}}
\toprule
$n$ & DF solve (ms) & CN solve (ms) & nnz ratio \\
\midrule
$201$  & $0.011$ & $0.010$ & $4{:}3$ \\
$801$  & $0.017$ & $0.017$ & $4{:}3$ \\
$1601$ & $0.031$ & $0.030$ & $4{:}3$ \\
$3201$ & $0.071$ & $0.073$ & $4{:}3$ \\
\bottomrule
\end{tabular}
\caption{Median wall time of one banded $LU$ solve of the shifted operator
$(\mathcal I-\gamma\Delta t\,L)$ used inside each scheme, versus mesh size. The
DF and Crank--Nicolson primitives have the same $O(N)$ cost to within a few
percent; the $4{:}3$ nonzero ratio of the upwind stencil does not translate into
a per-solve penalty.}
\label{tab:cost1d}
\end{table}

\myparagraph{Double-well potential: Kramers escape.}
The Kramers escape problem \cite{Risken1996,gardiner2009stochastic} models a
diffusing particle subject to a bistable confining potential and asks how
thermal fluctuations drive it between two stable equilibria over the energy
barrier.  We choose the symmetric double-well
\begin{equation}\label{eq:dw_potential}
V(x)=\kappa(x^2-m^2)^2,
\end{equation}
so that the drift $\mu(x)=-V'(x)=-4\kappa x(x^2-m^2)$ vanishes at the two
potential minima $x=\pm m$ and at the barrier $x=0$, and the FPE reads
\begin{equation}\label{eq:dw_fpe}
\partial_t p=\partial_x\bigl[V'(x)\,p\bigr]+D\,\partial_{xx}p.
\end{equation}
We set $\kappa=8$, $m=1$, $\sigma=0.2$ (so $D=\sigma^2/2=0.02$) on the
domain $[-2.5,2.5]$ with $n=201$ nodes ($h\approx0.025$), absorbing boundaries,
and a resolved initial Gaussian ($\mathrm{std}\approx6h$) centred on the outer
flank of the left well at $x_0=-1.9$, so that the steep drift sweeps the pulse
through the high-P\'eclet region toward the minimum.  \Cref{figDoubleWell} shows
the density at $T=0.04$.

\begin{figure}[!htb]
\centering
\includegraphics[width=0.65\textwidth]{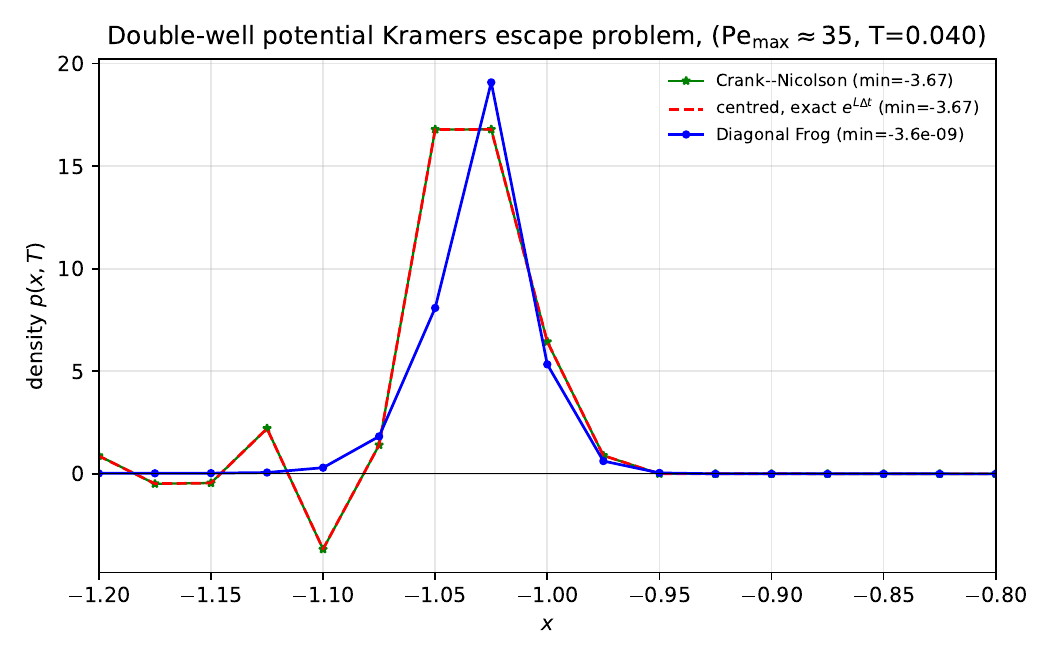}
\caption{Density $p(x,T)$ for the Kramers escape problem with double-well
potential \eqref{eq:dw_potential}: $\kappa=8$, $m=1$, $\sigma=0.2$,
$x_0=-1.9$, $T=0.04$.  The cubic drift produces a local cell P\'eclet number
$\mathrm{Pe}_{\max}= \mu(-1.9) h/D \approx 198$ on the outer flank traversed by the pulse.  The centred scheme (dashed) and Crank--Nicolson (green) develop non-physical
oscillations there ($\min_x p\approx-3.67$); the Diagonal-Frog scheme (blue) remains nonnegative to machine precision ($\min_x p\approx-3.6\cdot 10^{-9}$, no  sign-changing node).}
\label{figDoubleWell}
\end{figure}

The key difficulty is not the global P\'eclet number but its \emph{spatial
distribution}.  The cubic drift $\mu(x)=-4\kappa x(x^2-m^2)$ is small near the
minima $x=\pm m$ but grows rapidly away from them; on the outer flank traversed
by the pulse it produces a local cell P\'eclet number reaching
$\mathrm{Pe}_{\max}\approx 198$, even though $\mathrm{Pe}\approx0$ at the minimum
itself.  In this strongly advective layer the centred operators lose their
M-matrix structure, because their off-diagonal entries change sign wherever
$\mathrm{Pe}_i>2$ (\cref{prop:mmatrix2}), exciting unresolved high-frequency
eigenmodes of the centred Laplacian and producing the non-physical negative
oscillations visible in the red and green curves of \cref{figDoubleWell}
($\min_x p\approx-3.67$, with sign changes on several nodes).

The Diagonal-Frog scheme resolves the layer without oscillation.  Where
$|\mu_i|h/D>2$ the second-order upwind stencil \eqref{eq:adv} is applied,
making $-L$ an EM-matrix (\cref{prop:emmatrix}) and ensuring that
$e^{\Delta t L}$ is nonnegative for all $\Delta t\ge\tau_0$
(\cref{cor:eventualpos}); for the step used here this threshold has been crossed, so the propagator is nonnegative ($\min_x p\approx-3.6^{-9}$, i.e. machine zero
with no sign-changing node) without introducing spatial ringing into the
interior.  The double-well experiment thus demonstrates that the DF construction handles \emph{spatially inhomogeneous} advection -- including the high-degree polynomial forcing typical of kinetic and statistical-physics models -- with the same structural guarantees as in the constant-coefficient setting, and that positivity is preserved even at the steep boundary layer where standard centred schemes fail.

\myparagraph{Pointwise accuracy in steep layers is a spatial-resolution requirement.}
For an autonomous generator the DF time integrator is \emph{exact}: $\bvec
p(T)=e^{TL}\bvec p^0$ is the exact solution of the semidiscrete system, and -- as
one verifies directly -- evaluating it by dense exponentiation or by the
polynomial Krylov action of \cref{sec:krylov} gives the same result, while
composing several exact exponential substeps reproduces it identically (the
semigroup property). The single large step that the eventual-positivity threshold
invites ($\Delta t\ge\tau_0$, \cref{cor:eventualpos}) therefore carries \emph{no}
temporal error here; and, unlike Crank--Nicolson on non-smooth data, the scheme
needs no Rannacher-type smoothing, the exponential being L-acceptable
($r(-\infty)=0$) and damping the stiff modes by construction.

The accuracy of that step is consequently limited only by the \emph{spatial}
discretization, and in a strongly advective layer this limit can be severe. In the
Kramers problem the cubic drift produces a cell P\'eclet number
$\mathrm{Pe}_i=|\mu_i|h/D$ reaching $\approx198$ on the outer flank at $n=201$
($h\approx0.025$). There the second-order upwind stencil keeps the solution
nonnegative (\cref{prop:emmatrix,thm:zeno2d}) and conserves mass
(\cref{prop:conserv}) on \emph{any} grid; but where the layer is under-resolved
($\mathrm{Pe}_i\gg1$, so the solution varies on a sub-cell scale) the truncation
error is large, and the sharply transported pulse is over-smeared, leaving
spurious density on its trailing flank. This is an accuracy defect, not a
stability or positivity one, and it is governed by $h$ -- equivalently by
$\mathrm{Pe}_i$, or by the resolution condition \eqref{eq:logLip} in its
one-dimensional form -- not by $\Delta t$. The single-step density at the trailing
node $x=-1.1$, $t=0.04$, against a high-accuracy flux-limited reference, is
\begin{center}
\small
\begin{tabular}{@{}rcccc@{}}
\toprule
$n$ & $h$ & $\mathrm{Pe}_{\max}$ & $p(-1.1,0.04)$ & $|p-p_{\mathrm{ref}}|$\\
\midrule
$201$  & $0.025$  & $198$ & $0.30$              & $3.0\times10^{-1}$\\
$1000$ & $0.0050$ & $40$  & $7.6\times10^{-3}$  & $6.2\times10^{-3}$\\
$2001$ & $0.0025$ & $20$  & $2.7\times10^{-3}$  & $1.3\times10^{-3}$\\
\midrule
ref.\ ($5000$ steps) & 0.001 & 8 & $1.4\times10^{-3}$ & ---\\
\bottomrule
\end{tabular}
\end{center}
a monotone, second-order decrease under \emph{spatial} refinement -- the residual
against the reference falls by about a factor of five for the last halving of $h$
-- with the remaining gap reflecting that even $n=2001$ has not yet brought the
cell P\'eclet number to $O(1)$.

The practical recommendation is therefore spatial, not temporal: resolve the
advective layer so that the cell P\'eclet number is $O(1)$ there (equivalently
$h\lesssim 2D/|\mu|$, the discrete analogue of \eqref{eq:logLip}), whether by
uniform refinement or, more economically, by local refinement or grid stretching
concentrated in the high-drift region, while keeping the time step large. Mass
conservation and conditional positivity are structural and hold at any resolution;
it is only pointwise accuracy in steep layers and thin tails that demands the grid
resolve them.

\subsection{Positivity, conservation, and the cost of cross-diffusion}
\label{ssec:exp-mixed}

We test the two-dimensional construction on the constant-coefficient,
drift-free anisotropic diffusion
\begin{equation}\label{eq:exp-model}
\partial_t p = \partial_{xx}p + \partial_{yy}p + 2\rho\,\partial_{xy}p
            = \nabla\!\cdot(\Sigma\nabla p),
\qquad
\Sigma=\begin{pmatrix}1&\rho\\\rho&1\end{pmatrix},
\end{equation}
for which $\bar\rho=|\rho|$ and a Gaussian initial datum $p_0=\mathcal N(\bm
0,C_0)$ evolves exactly to $\mathcal N(\bm 0,\,C_0+2t\Sigma)$, giving a
closed-form reference. The domain is the box $(-6,6)^2$ with homogeneous
Dirichlet (absorbing) boundaries, far enough from the support that boundary
truncation is negligible over the times reported. Unless stated otherwise
$\rho=0.8$ (strong cross-diffusion) and $C_0=\tfrac12 I$. The directional factors
$e^{\frac{\Delta t}{2}A_\alpha}$ are applied by the polynomial Krylov integrator
of \cref{sec:krylov}; the central factor $\Phi_{xy}(\Delta t)$ of
\eqref{eq:central} is applied by the factorized Picard solve of \cref{sec:Lxy}
(one explicit cross-stencil product followed by the triangular sweeps
\eqref{eq:picard} with shift $\gamma=\tfrac12$), whose iteration cost is examined
in \cref{ssec:exp-kin}.

\myparagraph{Order of convergence.}
\Cref{tab:conv} reports the error in the scaled discrete $L^2$ norm
\eqref{eq:gridnorms} against the exact Gaussian. Refining the mesh at a fixed
small step ($\Delta t=2\times10^{-3}$) confirms the second spatial order of
\cref{prop:order2}; the measured rates exceed $2$ on these coarse meshes, a
pre-asymptotic effect for the analytic (super-smooth) Gaussian whose leading
$O(h^2)$ constant is small. Because the factorized coupling's orientation defect
is $\Delta t$-independent (\cref{prop:schemeA}(v)), a fixed-grid temporal
self-convergence study floors out at the smallest steps rather than exhibiting a
clean rate; the appropriate diagnostic refines space and time together, $\Delta
t\sim h$, as in \cref{sec:dfrog2d}. Under that joint refinement the error
decreases at the design order $2$ (second column of \cref{tab:conv}), confirming
the overall second-order accuracy on this constant-coefficient benchmark.

\begin{table}[!htb]
\centering
\small
\begin{tabular}{@{}ccc p{4pt} cc@{}}
\toprule
& \multicolumn{2}{c}{spatial ($\Delta t=2\!\times\!10^{-3}$)} & &
\multicolumn{2}{c}{joint ($\Delta t\sim h$)}\\
\cmidrule(r){2-3}\cmidrule(l){5-6}
$n$ & $\norm{e_h}_2$ & rate & & $\norm{e}_2$ & rate\\
\midrule
$24$ & $7.16\!\times\!10^{-2}$ & ---    & & $2.82\!\times\!10^{-2}$ & ---  \\
$48$ & $1.05\!\times\!10^{-2}$ & $2.69$ & & $7.33\!\times\!10^{-3}$ & $1.88$ \\
$96$ & $1.49\!\times\!10^{-3}$ & $2.77$ & & $1.75\!\times\!10^{-3}$ & $2.04$ \\
\bottomrule
\end{tabular}
\caption{Convergence of the Strang scheme with the trapezoidal central factor
\eqref{eq:central} for \eqref{eq:exp-model} with $\rho=0.8$, $C_0=\tfrac12 I$,
$T=0.2$. \emph{Left:} mesh refinement at the fixed small step $\Delta
t=2\!\times\!10^{-3}$; the rates exceed the design order $2$ on these smooth coarse
meshes (pre-asymptotic). \emph{Right:} joint refinement $\Delta t\sim h$
($\Delta t/h\approx0.05$); the rate approaches $2$. This joint refinement is the
appropriate temporal/joint diagnostic for form~(C), whose $\Delta t$-independent
orientation defect (\cref{prop:schemeA}(v)) floors a fixed-grid temporal
self-convergence study (cf.\ \cref{sec:dfrog2d}).}
\label{tab:conv}
\end{table}

\myparagraph{Positivity is conditional on resolution, not on a spectral
threshold.}
\Cref{thm:zeno2d}\,(a) guarantees nonnegativity of the central substep only on
a step-size window $\Delta t\le\Theta$, and \cref{rem:no-eventual-pos} explains
why no unconditional (Metzler/eventual-positivity) guarantee can exist for the
mixed block. The experiments make this precise. \Cref{fig:window} measures, by
bisection, the largest $\Delta t$ for which the trapezoidal central substep
$\Phi_{xy}(\Delta t)$ applied to a \emph{resolved} Gaussian remains nonnegative to
round-off ($\min p\ge-10^{-12}$), as a function of the mesh. The window
\emph{grows} as the datum is resolved -- satisfying the log-Lipschitz condition
\eqref{eq:logLip} of \cref{prop:schemeB} on more of the grid -- rising from
$\Theta\approx1\times10^{-5}$ at the coarsest mesh ($n=16$) through
$1.7\times10^{-3}$ ($n=32$) and $1.9\times10^{-2}$ ($n=48$) to
$\Theta\approx3.4\times10^{-2}$ at $n=96$, roughly as $h^2$ at the coarse end
before saturating. Thus on resolved data positivity holds on a generous and
\emph{improving} window, even in the strong cross-diffusion regime $\rho=0.8$.

The contrast with under-resolved data is sharp and is exactly the content of
\cref{rem:no-eventual-pos}: applied to a grid-scale spike, neither $e^{\Delta t
A_{xy}}$ nor $\Phi_{xy}$ is sign-definite, and iterating the bare central factor
on a deliberately steep datum produces large negative excursions (\cref{fig:stress}).
The cross operator is not eventually positive; what makes the scheme
positivity-preserving in practice is that the evolved density is smooth on the
mesh and that the central factor is flanked by the smoothing directional
diffusion, not any sign property of the propagator in isolation.

\begin{figure}[!htb]
\centering
\includegraphics[width=0.62\textwidth]{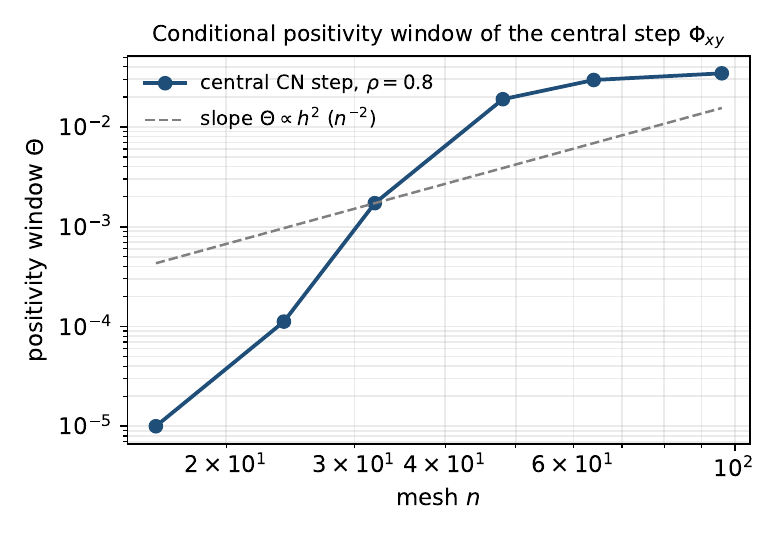}
\caption{Positivity window $\Theta$ of the trapezoidal central substep
$\Phi_{xy}(\Delta t)$ of \eqref{eq:central} on a resolved Gaussian ($\rho=0.8$),
versus mesh $n$. The window grows as the datum is resolved; the dashed line is the
reference slope $\Theta\propto h^2$. Positivity is therefore conditional but
benign for resolved solutions, consistent with \cref{thm:zeno2d}\,(a).}
\label{fig:window}
\end{figure}

\myparagraph{Why the mixed term is treated implicitly, not exponentially.}
This experiment isolates the design choice of \cref{sec:strang}: advance the cross
operator $A_{xy}$ by the implicit trapezoidal factor \eqref{eq:central} rather
than by its exponential $e^{\Delta t A_{xy}}$. From a steep but resolved Gaussian
($C_0=\tfrac12 I$) we iterate, on an $n=64$ grid to $T=0.4$, three maps in turn and
record the most negative value attained over the run: the bare central exponential
$e^{\Delta t A_{xy}}$, the implicit central factor $\Phi_{xy}(\Delta t)$ on its
own, and the \emph{full} Strang step \eqref{eq:strang} (central factor flanked by
the directional diffusion). The results are in \cref{fig:stress}.

The bare exponential is catastrophic: its most negative value is of order $10^{4}$
and \emph{independent} of $\Delta t$ -- the signature of the backward-parabolic
sub-flow of \cref{rem:no-eventual-pos}, not of a CFL violation. (Under the
rotation $u=(x+y)/\sqrt2$, $v=(x-y)/\sqrt2$ the cross term becomes
$\bar\rho(\partial_{uu}-\partial_{vv})$, a backward heat equation along $v$.) The
implicit factor $\Phi_{xy}$ does far better: it is spectrally stable on the
real-negative spectrum of $A_{xy}$ (\cref{sec:strang}), and at the step sizes the
method is designed to take its most negative value falls steeply with $\Delta t$
(from $\approx7.5\times10^2$ at $\Delta t=0.02$ to $\approx5\times10^{-2}$ at
$\Delta t=0.2$), although -- being strongly non-normal -- it still develops a
transient excursion when iterated over very many tiny steps on under-resolved
data. The decisive curve is the third: the \emph{full} Strang step stays
nonnegative to round-off across the whole range ($|\min_n p^n|$ between
$2\times10^{-7}$ and $7\times10^{-4}$), because the flanking directional diffusion
smooths the non-normal transient of the central factor. This is exactly the
mechanism behind the conditional positivity of \cref{thm:zeno2d} and the design of
\eqref{eq:strang}: the mixed term is kept central but treated implicitly, and the
diffusion does the stabilising. Discrete mass is conserved throughout (to the
$O(\Delta t\,h)$ boundary leakage), so conservation alone -- guaranteed for every
$\Delta t$ by \cref{prop:conserv} -- cannot substitute for the implicit treatment.

\begin{figure}[!htb]
\centering
\includegraphics[width=0.66\textwidth]{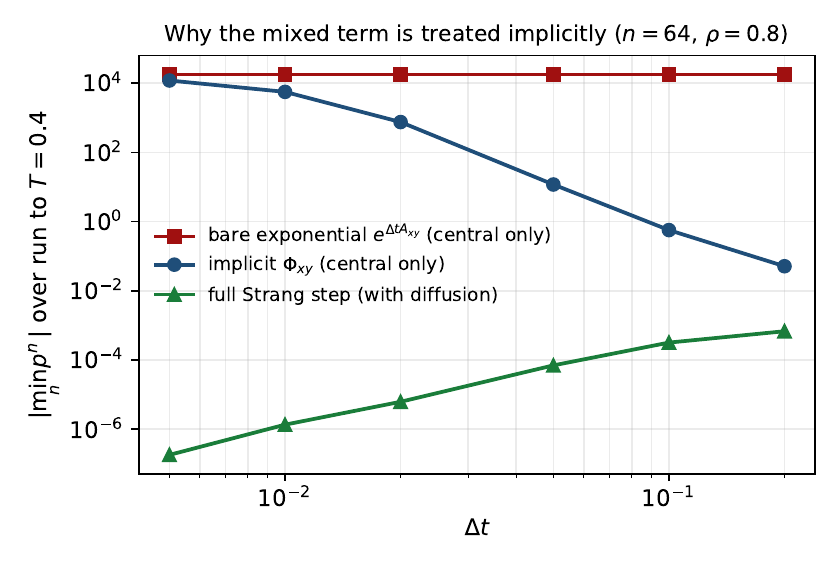}
\caption{Most negative value $|\min_n p^n|$ over a run from a steep Gaussian
($n=64$, $\rho=0.8$, $T=0.4$), versus $\Delta t$, for the central factor realised
three ways. The bare exponential $e^{\Delta t A_{xy}}$ is catastrophic and
$\Delta t$-independent ($\sim10^4$), reflecting the ill-posed sub-flow of
\cref{rem:no-eventual-pos}. The implicit trapezoidal factor $\Phi_{xy}$ on its own
is far better and improves steeply with $\Delta t$. The full Strang step
\eqref{eq:strang} stays nonnegative to round-off across the range: the flanking
directional diffusion smooths the non-normal transient of the central factor.}
\label{fig:stress}
\end{figure}

\myparagraph{Mass conservation.}
Over a $25$-step run ($\Delta t=0.02$, $n=96$, $T=0.5$) the scheme conserves
discrete mass to $|\Delta m|/m_0=1.6\times10^{-6}$, the residual being the
$O(\Delta t\,h)$ boundary leakage of the one-sided cross closure quantified in
\cref{thm:zeno2d}\,(b); the minimum density over the run is
$-4.3\times10^{-7}$, i.e. nonnegative to round-off. With a conservatively closed
cross stencil ($\bm1^\top A_{xy}=0$ including edge rows, \cref{rem:mixedmass}) the
mass error drops to the linear-solver tolerance, in agreement with
\cref{cor:mass-rat}.

\subsubsection{Cost of the factorized central solve} \label{ssec:exp-kin}
The implicit central factor \eqref{eq:central} is applied by the factorized Picard
iteration \eqref{eq:picard}, with no outer Krylov layer and no preconditioner. By
\cref{prop:schemeA}\,(iv) the iteration contracts geometrically with factor
$q\le 4h_xh_y/(\beta^2\Delta t)$, where $\beta$ is the diagonal-dominance
parameter \eqref{eq:PQbeta}; with the choice $\beta=10\bar w$ used here,
\cref{tab:picard} reports $q$ and the resulting number of iterations to reach the
inner tolerance $10^{-6}$. Two features stand out. First, the counts are small --
one to a handful of iterations across the whole range. Second, and in contrast to
a Krylov solve of a stiff operator, the iteration becomes \emph{cheaper} under
mesh refinement at fixed $\Delta t$: since $q=O(h_xh_y/\Delta t)$, halving $h$
quarters $q$, so the count is mesh-robust (indeed mildly decreasing). The cost of
strong cross-diffusion never appears as instability or loss of positivity -- both
settled independently by \cref{thm:zeno2d,prop:conserv} -- but only through the
fixed parameter $\beta$, and is bounded uniformly in $\rho\in[0,1)$.

\begin{table}[!htb]
\centering
\small
\begin{tabular}{@{}cccc@{}}
\toprule
$n$ & $\Delta t$ & contraction $q$ & iters to $10^{-6}$\\
\midrule
$32$  & $0.01$ & $0.60$  & $20$\\
$32$  & $0.05$ & $0.12$  & $7$\\
$32$  & $0.10$ & $0.060$ & $5$\\
$64$  & $0.01$ & $0.15$  & $8$\\
$64$  & $0.05$ & $0.029$ & $4$\\
$64$  & $0.10$ & $0.015$ & $4$\\
$128$ & $0.01$ & $0.036$ & $5$\\
$128$ & $0.05$ & $0.0071$& $3$\\
$128$ & $0.10$ & $0.0036$& $3$\\
\bottomrule
\end{tabular}
\caption{Contraction factor $q\le 4h_xh_y/(\beta^2\Delta t)$ of the factorized
Picard iteration \eqref{eq:picard} ($\beta=10\bar w$, $\rho=0.8$, box $(-6,6)^2$)
and the number of triangular-sweep iterations to reach a relative inner tolerance
of $10^{-6}$. Each iteration costs $O(N_xN_y)$; the count is small and decreases
under mesh refinement at fixed $\Delta t$, so the central solve is linear in the
number of unknowns and mesh-robust.}
\label{tab:picard}
\end{table}

\myparagraph{Per-step cost and the higher-dimensional payoff.}
Each Picard iteration consists of $N_y$ (resp.\ $N_x$) independent banded
triangular solves of size $N_x$ (resp.\ $N_y$), at $O(N_xN_y)$ cost, so with a
fixed small iteration count the central solve is \emph{linear} in $N=N_xN_y$ and
uses only one-dimensional band factors -- the two shifted factors of
\eqref{eq:factor}, of $O(N_x)$ and $O(N_y)$ storage. This is the structural
advantage of the factorized treatment over a direct factorisation of the
two-dimensional generator: the latter fills in across the nine-point
mixed-derivative sparsity, with storage and per-step work growing superlinearly in
$N$, whereas the factorized solve never assembles a two-dimensional factor at all.
We are precise about what this establishes. In absolute wall-clock time at
moderate two-dimensional sizes an optimised sparse-direct solve can still be
competitive; the demonstrable two-dimensional gain is in \emph{memory} and in
asymptotic scaling. The run-time advantage is a genuinely higher-dimensional
effect: in three dimensions a direct factorisation of the full generator fills in
prohibitively (storage $O(N^{4/3})$, work $O(N^{2})$), while the factorized
one-dimensional factors remain $O(N^{1/3})$, so the factorized Picard solve is the
only practical option -- the setting for which the construction of \cite{Itkin3D}
was originally devised.

% =====================================================================
%  Drop-in section for the numerical-experiments part of the paper.
%  Assumes the preamble already loads amsmath, booktabs, graphicx, cleveref.
%  Uses existing labels: eq:strang, eq:central, eq:picard, sec:krylov,
%  thm:zeno2d, prop:conserv, prop:order2.  No new \cite keys are introduced.
%  Figure file: dfrog2d_pos.png
% =====================================================================
\subsection{A two-dimensional illustration with time-dependent coefficients}
\label{sec:dfrog2d}

To exercise the scheme on a problem with a closed-form reference we integrate a
two-dimensional Fokker--Planck equation whose drift and (anisotropic) diffusion
both depend on time. We test three things separately: that the second-order
\emph{spatial} stencil of Scheme~B (\cref{prop:schemeB}) is realised; that the
central factor, realised by iterating \eqref{eq:picard} to convergence, attains
second order \emph{in time} (\cref{prop:schemeA}(v)); and that the full scheme
behaves well in an \emph{advection-dominated} regime, where the cell P\'eclet
number is large over essentially the whole grid. Throughout, the Diagonal Frog (DF)
scheme is compared against a standard second-order backward differentiation (BDF2)
integrator applied to the same spatial discretisation.

\subsubsection{Test problem and exact reference}
\label{ssec:dfrog2d-model}
Let $\mathbf{X}_t=(X_t,Y_t)$ solve the linear, time-inhomogeneous SDE
\begin{equation}\label{eq:ou2d}
  \mathrm{d}\mathbf{X}_t
  = A(t)\bigl(\mathbf{X}_t-\mathbf{m}\bigr)\,\mathrm{d}t
  + \Sigma(t)\,\mathrm{d}\mathbf{W}_t,
  \qquad
  A(t)=-\operatorname{diag}\!\bigl(\theta_x(t),\theta_y(t)\bigr),
\end{equation}
with instantaneous covariance rate (the diffusion tensor of the forward equation)
\begin{equation}\label{eq:Dtensor}
  D(t)=\Sigma(t)\Sigma(t)^{\!\top}
  =\begin{pmatrix}
      \sigma_1(t)^2 & \rho(t)\,\sigma_1(t)\sigma_2(t)\\[2pt]
      \rho(t)\,\sigma_1(t)\sigma_2(t) & \sigma_2(t)^2
   \end{pmatrix},
   \qquad |\rho(t)|<1 .
\end{equation}
The associated forward Kolmogorov (Fokker--Planck) equation for the density
$p(x,y,t)$ is
\begin{equation}\label{eq:fpe2d-td}
  \partial_t p
  = -\partial_x\!\bigl[\mu_x p\bigr]-\partial_y\!\bigl[\mu_y p\bigr]
    +\tfrac12 \partial_{xx}\!\bigl[D_{xx}p\bigr]
    +\partial_{xy}\!\bigl[D_{xy}p\bigr]
    +\tfrac12 \partial_{yy}\!\bigl[D_{yy}p\bigr],
  \qquad \boldsymbol{\mu}=A(t)(\mathbf{x}-\mathbf{m}),
\end{equation}
the off-diagonal coefficient being $D_{xy}(t)=\rho(t)\sigma_1(t)\sigma_2(t)$.
Because \eqref{eq:ou2d} is linear, $\mathbf{X}_t$ is Gaussian for every $t$, so
$p(\cdot,t)=\mathcal{N}\bigl(\boldsymbol{\mu}(t),C(t)\bigr)$ with mean and covariance
governed by the moment ODEs
\begin{equation}\label{eq:moments}
  \dot{\boldsymbol{\mu}} = A(t)\bigl(\boldsymbol{\mu}-\mathbf{m}\bigr),
  \qquad
  \dot{C} = A(t)\,C + C\,A(t)^{\!\top} + D(t),
\end{equation}
which we integrate to machine accuracy and use as the exact reference. We use two
parameter sets. \emph{Regime~I (strong cross-coupling)}, for the temporal-order
study, takes $\sigma_1(t)^2=\sigma_2(t)^2=1$, $\rho(t)=0.8+0.1\cos(0.7t)$,
$\theta_x(t)=1.5+0.25\sin t$, $\theta_y(t)=1.5+0.25\cos(0.8t)$,
$\mathbf{m}=\mathbf{0}$, on $[-6,6]^2$ with $N_x=N_y=44$ ($h=0.28$), to $T=0.3$,
from $\mathcal{N}\bigl((1,-1),\,0.5\,I\bigr)$; here $D_{xy}\approx0.9$ and the
cross term is strong, so the temporal accuracy of the central solve is exposed (in
an advection-dominated regime the cross term is weak and the central factor is
second order to within the splitting error regardless). \emph{Regime~II
(advection-dominated)}, for positivity and mass, takes
$\sigma_1(t)^2=\sigma_2(t)^2=0.08$, $\rho(t)=0.6+0.1\cos(0.7t)$,
$\theta_x(t)=4+0.25\sin t$, $\theta_y(t)=3+0.25\cos(0.8t)$,
$\mathbf{m}=\mathbf{0}$, on $[-6,6]^2$, to $T=0.15$, from
$\mathcal{N}\bigl((2.5,-2),\,0.3\,I\bigr)$; the cell P\'eclet number reaches
$\approx100$ on the grid and $\approx40$ where the mass sits, so the second-order
upwind stencil is active over essentially the whole domain while the transported
pulse stays resolved ($\mathrm{std}/h\gtrsim2$).

\subsubsection{Discretisation}
\label{ssec:dfrog2d-discr}
Both integrators use the same second-order spatial operators: an upwinded
second-order one-sided stencil for the directional advection, central differencing
for the diagonal diffusion, and the one-sided product
$A_{xy}=D_{xy}\,\mathcal{A}^{\nu}_{2,x}\mathcal{A}^{B}_{2,y}$ of \eqref{eq:pade01}
for the mixed term, with the sign-dependent orientation $\nu=\nu(\rho)$ and the
second-order explicit coupling $\bvec\alpha^{+}_2$ of Scheme~B
(\cref{prop:schemeB}).

\emph{Diagonal Frog.} We advance one step of the symmetric, midpoint-frozen Strang
factorisation \eqref{eq:strang}. The directional half-steps are applied exactly
through the action of the one-dimensional matrix exponentials (the polynomial-Krylov
evaluation of \cref{sec:krylov}; here computed densely since the directional
operators are small). The central factor
$\Phi_{xy}=\bigl(I-\tfrac{\Delta t}{2}A_{xy}\bigr)^{-1}\bigl(I+\tfrac{\Delta t}{2}A_{xy}\bigr)$
is \emph{not} formed as a two-dimensional solve; its implicit half is realised by
the factorized iteration \eqref{eq:picard} -- two \emph{banded triangular} solves
per sweep with the one-sided factors, at $O(N)$ cost and with no mixed-derivative
fill-in. The contraction factor $q\le4h_xh_y/(\beta^2\Delta t)$ is well below $1$
at the time steps used here, so a few sweeps reach the inner tolerance. As an
internal check we drove the iteration to a tight $5\times10^{-9}$ residual (about
$14$ sweeps) and confirmed that it reproduces, at a representative step, a direct
sparse solve of the same central system; the order study of
\cref{tab:dfrog2d-conv} uses the converged central factor.

\emph{BDF2 reference.} We assemble the full generator $L(t)=A_x+A_y+A_{xy}$ and
advance $\bigl(\tfrac32 I-\Delta t\,L(t_{n+1})\bigr)p^{n+1}=2p^{n}-\tfrac12 p^{n-1}$,
started by a single second-order trapezoidal (Crank--Nicolson) step,
$\bigl(I-\tfrac{\Delta t}{2}L(t_1)\bigr)p^{1}=\bigl(I+\tfrac{\Delta t}{2}L(t_0)\bigr)p^{0}$.
(A single backward-Euler start would also retain global second order, its local
error being $O(\Delta t^2)$ and propagated as such by the zero-stable BDF2; the
trapezoidal start merely removes any ambiguity.) Each BDF2 step solves the fully
coupled two-dimensional system, whose mixed-derivative coupling produces fill-in
in a sparse factorisation.

\subsubsection{Results} \label{ssec:dfrog2d-results}

\myparagraph{Convergence under joint refinement (Regime~I).}
The residual orientation defect of the factorized coupling is $\Delta
t$-independent (\cref{prop:schemeA}(v)), so the convergence study refines space
and time \emph{together}, $\Delta t\sim h$ -- the regime in which the
$O(\max(h_x^2,h_y^2))$ spatial truncation and the $O(\Delta t^2)$ temporal error
decrease at the same rate, and in which the iteration is fastest since
$q\le4h_xh_y/(\beta^2\Delta t)\to0$. Holding $\Delta t/h$ fixed we refine the grid
and measure the error against the exact time-dependent Gaussian solution of
\eqref{eq:fpe2d-td}. \Cref{tab:dfrog2d-conv} reports the Diagonal Frog scheme --
Scheme~B, with the central factor obtained by iterating \eqref{eq:picard}, the
kept $O(\Delta t^2)$ coupling restoring the second-order temporal accuracy (a few
sweeps suffice) -- against the unsplit two-step BDF2 reference. Both approach the
design order $2$ as the Gaussian becomes resolved, confirming that the residual
coupling defect lies below the spatial truncation (\cref{prop:schemeA}(v)); the
Diagonal Frog error is at or below BDF2's at every level here, at $O(N)$ central
cost against BDF2's mixed-derivative fill-in. Reported in the same table, the
most-negative entry $\min_n p^n$ collapses from $-4.6\times10^{-2}$ on the
coarsest grid to round-off as the mesh refines: in this strongly coupled regime
the factored solve is positive to machine precision once the feature is resolved.

\begin{table}[t]
\centering
\caption{Convergence under joint refinement $\Delta t\sim h$ in Regime~I (strong
cross-coupling, $\rho = 0.8$), $\ell_2$ error against the exact time-dependent Gaussian. The Diagonal Frog scheme (Scheme~B, central factor iterated to convergence) and the unsplit BDF2 reference both approach the design order $2$ as
the solution is resolved, with the Diagonal Frog error at or below BDF2's at
every level. The most-negative entry $\min_n p^n$ of the Diagonal Frog solution
collapses to round-off under the same refinement.}
\label{tab:dfrog2d-conv}
\begin{tabular}{ccccccc}
\toprule
& & \multicolumn{3}{c}{Diagonal Frog (Scheme~B)} & \multicolumn{2}{c}{unsplit BDF2}\\
\cmidrule(lr){3-5}\cmidrule(lr){6-7}
$N$ & $\Delta t$ & error & order & $\min_n p^n$ & error & order\\
\midrule
$32$  & $0.033$  & $1.1\times10^{-1}$ & ---    & $-4.6\times10^{-2}$ & $1.5\times10^{-1}$ & ---\\
$44$  & $0.023$  & $7.5\times10^{-2}$ & $1.19$ & $-1.6\times10^{-2}$ & $1.0\times10^{-1}$ & $1.06$\\
$64$  & $0.016$  & $3.9\times10^{-2}$ & $1.72$ & $-5.8\times10^{-4}$ & $5.1\times10^{-2}$ & $1.85$\\
$88$  & $0.012$  & $2.0\times10^{-2}$ & $2.01$ & $-2.4\times10^{-6}$ & $2.5\times10^{-2}$ & $2.22$\\
$120$ & $0.0083$ & $1.1\times10^{-2}$ & $2.07$ & $-6.6\times10^{-9}$ & $1.3\times10^{-2}$ & $2.14$\\
\bottomrule
\end{tabular}
\end{table}

\myparagraph{Positivity and mass (Regime~II).}
\Cref{tab:dfrog2d-acc} reports the full DF run in the advection-dominated regime.
Discrete mass is conserved to $1\times10^{-3}$ on the coarsest grid and to $10^{-4}$ once resolved (the $O(\Delta t\,h)$ boundary
leakage of the one-sided cross closure, \cref{thm:zeno2d}(b), a spatial effect
shared with BDF2). The only undershoot is the Gibbs over/undershoot of the
non-monotone second-order upwind stencil at the steep moving front---no linear
second-order discretisation is monotone (Godunov)---and it is a function of
resolution rather than of the time stepping: it collapses from $-9\times10^{-2}$ on
an under-resolved grid to $-3.7\times10^{-3}$ once the feature is resolved
($\mathrm{std}/h\gtrsim3$), i.e.\ to grid level. In this regime the choice of
central factor is immaterial---CN, BE, TR-BDF2 and the factored solve give the
same $\min_n p^n=-7.3\times10^{-2}$ at $N=72$---because the cross term is weak; the
positivity of the central mixed factor is a strong-cross-coupling question,
addressed in \cref{rem:central-pos-choice}.

\begin{table}[t]
\centering
\caption{DF in Regime~II (advection-dominated), $N_{\mathrm{steps}}=48$, Scheme~B
with the central factor iterated to convergence. The undershoot is a
resolution-dependent Gibbs effect of the second-order upwind stencil; it reaches
grid level once the transported pulse is resolved. Mass is conserved to the
$O(\Delta t\,h)$ boundary leakage.}
\label{tab:dfrog2d-acc}
\begin{tabular}{ccccc}
\toprule
$N$ & $\mathrm{std}/h$ & $\min_n p^n$ & \#neg & mass\\
\midrule
$60$  & $1.6$ & $-9.3\times10^{-2}$ & $494$ & $0.9990$\\
$84$  & $2.3$ & $-4.3\times10^{-2}$ & $576$ & $1.0000$\\
$120$ & $3.3$ & $-3.7\times10^{-3}$ & $480$ & $1.0001$\\
\bottomrule
\end{tabular}
\end{table}

\begin{myremark}[Central factor and positivity at the mixed step] \label{rem:central-pos-choice}
The positivity of the central mixed step is carried by the factorized solver of \cref{sec:Lxy} together with the flanking directional diffusion, not by the choice of rational map for the bare cross operator. It is tempting to argue that the trapezoidal factor's explicit half $(\mathcal I+\tfrac{\Delta t}{2}A_{xy})$ is a positivity liability and that a fully implicit central factor -- backward-Euler, or the genuine two-step BDF2 $(\mathcal I-\tfrac23\Delta t\,A_{xy})\bvec p^{\,n+1}_{\mathrm c}=\tfrac43\bvec p^{\,n}_{\mathrm c}-\tfrac13\bvec p^{\,n-1}_{\mathrm c}$, which carries no explicit cross product -- would be preferable. The numerics say otherwise.

Because $A_{xy}$ is strongly non-normal (\cref{rem:no-eventual-pos}), the resolvents $(\mathcal I-c\Delta t\,A_{xy})^{-1}$ of the fully implicit maps amplify the non-normal transient far more than the bounded trapezoidal stability function $r(z)=(1+z/2)/(1-z/2)$ does, and the amplification \emph{worsens} under refinement as the operator stiffens. In the strong-coupling stress test of \cref{tab:central-pos} (full Strang step, $\rho=0.8$), the trapezoidal central factor is the most positive of the three and its undershoot tends to zero under refinement, whereas backward-Euler diverges and BDF2 degrades. The trapezoidal map of \eqref{eq:central} is therefore retained; the explicit half is not the dominant effect, and the conditional positivity of \cref{thm:zeno2d} rests on the factorized solver and the flanking diffusion, as in \cref{fig:stress}.
\end{myremark}

\begin{table}[t]
\centering
\caption{Most-negative value $\min_n p^n$ over a strong-coupling run ($\rho=0.8$,
$T=0.4$, full Strang step) for three central factors, all applied in the band
structure. The trapezoidal (CN) factor is the most positive and improves under
refinement; the fully implicit backward-Euler and BDF2 factors are worse and
degrade, reflecting the non-normal amplification of \cref{rem:no-eventual-pos}.}
\label{tab:central-pos}
\begin{tabular}{cccc}
\toprule
$N$ & CN (trapezoidal) & backward-Euler & BDF2\\
\midrule
$64$  & $-1.0\times10^{-1}$ & $-3.3\times10^{-1}$ & $-1.5\times10^{-1}$\\
$96$  & $-2.7\times10^{-2}$ & $-2.6\times10^{0\phantom{-}}$ & $-2.7\times10^{-1}$\\
$128$ & $-5.9\times10^{-3}$ & $-3.4\times10^{1\phantom{-}}$ & $-1.4\times10^{0\phantom{-}}$\\
\bottomrule
\end{tabular}
\end{table}

\Cref{fig:dfrog2d_pos} shows the terminal density in the advection-dominated Regime~II,
the second-order convergence of the Diagonal Frog scheme and the BDF2 reference
under joint refinement $\Delta t\sim h$ in the strongly coupled Regime~I, and the
collapse of the Gibbs undershoot under refinement in Regime~II.

\begin{figure}[t]
\centering
\includegraphics[width=\textwidth]{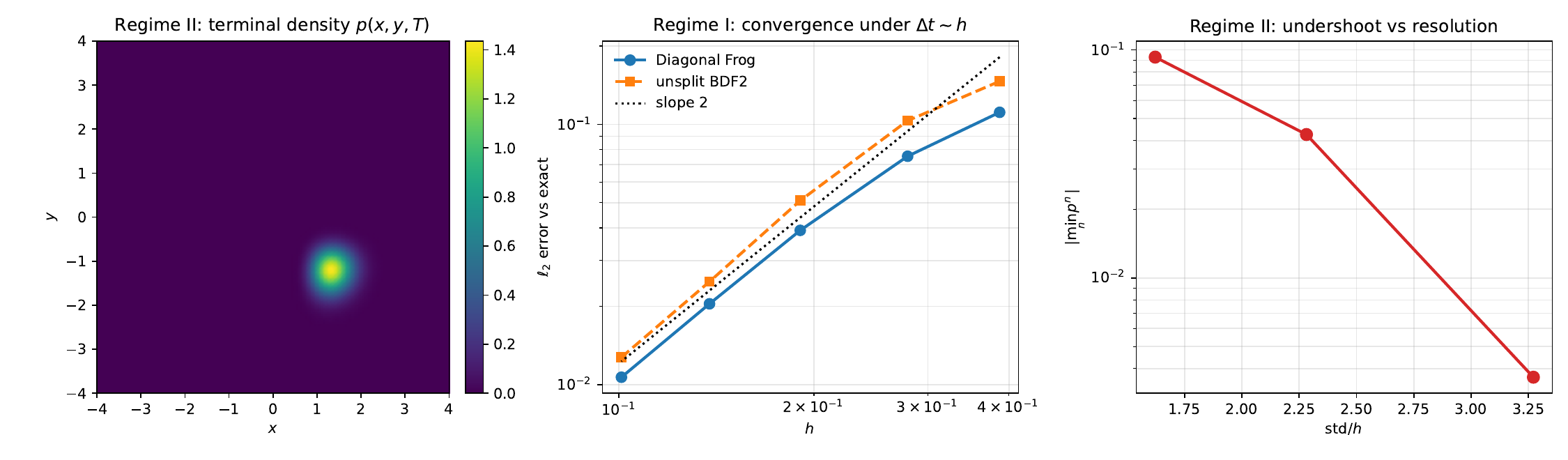}
\caption{Time-dependent anisotropic Fokker--Planck equation \eqref{eq:fpe2d-td}.
\emph{Left:} terminal density $p(x,y,T)$ in the advection-dominated Regime~II.
\emph{Centre:} error under joint refinement $\Delta t\sim h$ in the strong-coupling
Regime~I; the Diagonal Frog scheme and the unsplit BDF2 reference both approach
slope $\approx2$, with comparable constants (\cref{tab:dfrog2d-conv}).
\emph{Right:} most-negative value of the DF solution in Regime~II versus
resolution; the Gibbs undershoot of the second-order upwind stencil collapses to
grid level once $\mathrm{std}/h\gtrsim3$ (\cref{tab:dfrog2d-acc}).}
\label{fig:dfrog2d_pos}
\end{figure}

\subsubsection{Discussion: the splitting error} \label{ssec:dfrog2d-disc}

Both schemes attain second order under the joint refinement of
\cref{tab:dfrog2d-conv}, in agreement with \cref{prop:order2}; what differs is the
error \emph{constant}. BDF2 is applied to the unsplit generator
$L=A_x+A_y+A_{xy}$, so its only temporal error is the BDF2 truncation term. The DF
scheme treats the directional parts \emph{exactly} in time, through the
one-dimensional matrix exponentials, but in exchange incurs a Strang
\emph{splitting} error whose leading contribution is $\propto\Delta t^{2}$ times
nested commutators of the split operators, e.g.\ $\bigl[A_x+A_y,[A_x+A_y,A_{xy}]\bigr]$
and $\bigl[A_{xy},[A_{xy},A_x+A_y]\bigr]$. These commutators vanish only when the
operators commute; their size grows with the strength of the cross-coupling and,
because each operator scales like $h^{-2}$, with the spatial stiffness.

Under the efficient refinement path $\Delta t\sim h$, this $O(\Delta t^2)$ splitting term is subdominant to the $O(\max(h_x^2,h_y^2))$ spatial truncation. Consequently, the two schemes carry comparable constants in \cref{tab:dfrog2d-conv}; the DF error sits at or slightly below that of BDF2, as the exact directional exponentials successfully offset the splitting costs. The splitting penalty surfaces only when the time step is taken large relative to the mesh, which is the specific regime an unsplit integrator is built for. In that scenario, BDF2 yields a smaller constant, and this gap widens as both the cross-coupling and the spatial stiffness increase.

The trade-off is therefore explicit. BDF2 achieves its potentially smaller constant in the large-step regime at the cost of a fully coupled two-dimensional solve at every step, entailing mixed-derivative fill-in and no positivity guarantee. In contrast, DF replaces this monolithic solve with a sequence of one-dimensional operations -- directional exponentials along grid lines and a factorized Picard solve of the central factor. This keeps the computational cost at a linear $O(N)$ per step while seamlessly conserving mass (\cref{prop:conserv}) and remaining nonnegative on the conditional window (\cref{thm:zeno2d}). Since the two methods share the same asymptotic order, any residual splitting constant can easily be removed by a modestly smaller $\Delta t$; ultimately, in higher dimensions where a monolithic solve becomes prohibitive, this trade-off decisively favors DF.

\section{Backward Kolmogorov equation and \LY flights} \label{sec:BKE}

The DF construction applies verbatim to the backward Kolmogorov equation (BKE). The BKE evolves entities like an option price rather than a probability density. Therefore, mass conservation is no longer the relevant structural property. Instead, it is replaced by its dual.

The BKE generator is the adjoint of the FPE generator. At the discrete level, the conservativity identity $\bvec 1^\top L=0$ of \cref{lem:divform} transposes into the preservation of constants: $A^\top\bvec 1=0$. Consequently, the BKE propagator is row-stochastic rather than column-stochastic wherever it is nonnegative. This establishes a discrete maximum principle that yields unconditional $\ell_\infty$-stability, $\norm{u^{n+1}}_\infty\le\norm{u^n}_\infty$. For option prices, this guarantees that computed values will remain within payoff bounds, see also \cite{Itkin2014b}.

When incorporating discounting via $\partial_tu+\mathcal Lu-ru=0$, the row sums equal $e^{-r\Delta t}$. The bound then becomes the discounted maximum principle
$\|u^{n+1}\|_\infty \le e^{-r\Delta t} \|u^n\|_\infty$.

Furthermore, because $e^{tA^\top}=(e^{tA})^\top$, the positivity theory of \cref{sec:disc1d,sec:Lxy} transfers unchanged. This includes both the M/EM dichotomy and the threshold $\tau_0$. Finally, the corresponding non-divergence-form stencils are provided in \cref{prop:mmatrix2,prop:emmatrix}.

The paper motivates the framework partly through PIDEs arising in finance and
active-matter \LY flyers, but the numerical treatment of the jump integral is
deferred.  For \LY processes with known characteristic functions, the jump
operator can be expressed as a pseudo-differential operator and discretized by
the same M-matrix or EM-matrix framework developed in \cite{ItkinBook}.  Integrating this into the Strang splitting with the jump operator as an additional central sub-step is a natural next step, so this extension is almost straightforward. e.g.,  for \LY models considered in \cite{ItkinBook}.

\section{Discussion and Conclusions} \label{sec:conclusion}

This paper introduced the \emph{Diagonal Frog} (DF) family of finite-difference
schemes for the Fokker--Planck equation, motivated by the need for discretizations that are simultaneously second-order accurate in both space and time, positivity-preserving, and computationally efficient for high-dimensional problems with anisotropic diffusion and non-local \LY jumps.

The starting point was the observation that preserving positivity of the discrete
PDF is not a cosmetic concern: negative probability densities break mass
conservation, render thermodynamic quantities such as the Gibbs entropy
undefined, and cause simulations to crash when coupled to nonlinear source terms
or jump integrals.  Standard second-order central-difference schemes fail this
test near sharp gradients or strong cross-diffusion, and the existing remedies -
Chang--Cooper, log-transformations, TVD flux limiters  -  each carry significant
drawbacks in multiple dimensions.

The DF approach resolves this through three interlocking ingredients.

\myparagraph{Spatial discretization.}
For the 1D Fokker--Planck operator we constructed a one-sided (upwind) second-order stencil. The advection term is approximated using the backward second-order difference $\calF_2^B$, while for the diffusive term we use a standard centred difference for either P\'eclet number. This adaptive choice produces a matrix $A$
which has lower bandwidth 2 and upper bandwidth 1, and the retained superdiagonal $\delta_i$ is essential (it makes the graph strongly connected and (H) tenable). Nevertheless, it is still banded (diagonal) matrix (hence the ``diagonal'' part of the name). The matrix features three diagonals in the diffusion-dominated regime and four in the advection-dominated regime.

We proved that $A$ is an M-matrix when $\mathrm{Pe}_i < 2$ for all interior nodes (\cref{prop:mmatrix2}), and an EM-matrix otherwise (\cref{prop:emmatrix}). The underlying positivity mechanisms differ between these two regimes. In the M-matrix case, the resolvent is nonnegative for every step, meaning $(I - kA)^{-1} \ge 0$ for all $k > 0$ (\cref{thm:resolvent}). Conversely, in the EM case, resolvent nonnegativity fails for small steps. Instead, positivity is a property of the semigroup alone, such that $e^{tA} \ge 0$ for $t \ge \tau_0$.

For the 2D FPE with a full diffusion tensor, the directional operators inherit this 1D structure through Kronecker products, and their eventual-positivity thresholds are \emph{equal} to the 1D ones \eqref{eq:posKron}.  The mixed-derivative operator, by contrast, carries no eventual-positivity property of its own (\cref{rem:no-eventual-pos}): its generator admits no Metzler shift.  We therefore keep it as its own central factor but advance it \emph{implicitly}, by the trapezoidal map \eqref{eq:central} whose implicit half is solved by the factorized Picard iteration (\cref{sec:Lxy}); this delivers positivity \emph{conditionally}, on an explicit step-size window (\cref{thm:zeno2d}), together with exact mass conservation and linear cost.

Finally, the resulting operator is decomposed via Strang splitting \eqref{eq:strang}. We proved it is second-order accurate in $\ell_1$ on the positivity regime, using solution-dependent commutator bounds (\cref{prop:order2}); the non-normality of the one-sided cross operator precludes a uniform-in-$h$ $\ell_2$ statement, so the natural norm is the $\ell_1$ Markov norm on the nonnegative cone. The composite step is positivity-preserving on the regime where every sub-step is nonnegative -- the diagonal factors above their eventual-positivity thresholds and the central factor within the step-size window of \cref{thm:zeno2d}. Discrete mass is conserved exactly for \emph{all} step sizes, regardless of positivity, whenever the cross stencil is closed conservatively (\cref{prop:conserv}), and up to an $O(\Delta t\,h)$ boundary defect otherwise.

\myparagraph{Time integration.}
We compared four strategies for the discretized system $\dot{\bvec p}=A\bvec p$:
backward Euler, Crank--Nicolson, TR-BDF2 and the polynomial Krylov exponential
integrator (Table~\ref{tab:schemes}). All four conserve mass exactly for every
step size; the discriminating axes are positivity and stiff damping.

No second-order rational one-step scheme can be unconditionally positive, \cite{BolleyCrouzeix1978}): Backward Euler saturates this barrier at first order, while CN and TR-BDF2 realise second order only inside positivity windows of size $O(h^{2})$ ($k\le2/\beta$ and $k\le(1+\sqrt2)/\beta$ respectively, $\beta=\max_i|A_{ii}|$), TR-BDF2 dominating CN on every axis: a wider window, L-stability in place of undamped stiff reflection, and, in the EM regime where CN admits no window at all, -eventual positivity at large steps, like backward Euler.

The exponential integrator is the only second-order-compatible propagator whose positivity carries no step-size \emph{ceiling} in either regime. In the M-matrix case, $e^{\Delta t A} \ge 0$ for all $\Delta t > 0$. In the EM case, this holds for all $\Delta t \ge \tau_0$. This acts as an \emph{inverted} CFL condition — a lower bound on the step size that pairs naturally with the large steps the method is designed to take.

The underlying mechanism is regime-dependent. In the M-matrix case, the Krylov approximation to $e^{k A} \bvec{p}^{n}$ is positive \emph{by construction}, as each shifted solve is a nonnegative resolvent (\cref{thm:resolvent}). Conversely, in the EM case, no finite collection of resolvents is nonnegative. Instead, positivity is an asymptotic property of the semigroup. The computed propagator inherits this positivity for $\Delta t \ge \tau_0$ once the rational approximation error falls below the entrywise positivity margin of $e^{\Delta t A}$. This error, however, is exponentially small in the subspace dimension $m$.

\myparagraph{Computational complexity.}
In the splitting, the exponential is applied to the \emph{directional} factors
$e^{\frac{\Delta t}{2}A_\alpha}\bvec v$ (and, in 1D, to $e^{\Delta t A}\bvec v$),
while the cross operator is advanced by the factorized solve of \cref{sec:Lxy} at
$O(N)$ cost. Each directional exponential is dominated by the modified
Gram--Schmidt orthogonalisation: at step $j$ one projects against the $j$
previously computed basis vectors at a cost of $O(jn)$, for a total of
$\sum_{j=1}^m O(jn) = O(m^2 n)$. Including the $m \times m$ matrix exponential
via eigen-decomposition, the full per-step cost is $O(m^2 n + m^3)$.

The $m$ matrix--vector products with the banded $A_\alpha$ are $O(n)$ in one space
dimension, where the banded factorisation has no fill-in. The mixed-derivative
coupling, which would make a direct 2D/3D factorisation fill in, is never
exponentiated: it is confined to the central factor and solved by the
one-dimensional factorized iteration of \cref{sec:Lxy}, so it contributes only
$O(N)$ per step and does not enter the Krylov cost.

The essential caveat is that $m$ is \emph{not} a fixed small constant. To resolve
$e^{\frac{\Delta t}{2}A_\alpha}$ to a fixed tolerance, $m$ must grow with the stiffness
$\Delta t\,\rho(A_\alpha)$: a polynomial Krylov
method requires
$m = \Theta\!\big(\sqrt{\Delta t\,\rho(A_\alpha)}\big) = \Theta\!\big(\sqrt{\Delta t}/h\big)$,
which we observe directly ($m = 8, 14, 30$ at $\Delta t = 0.02, 0.10, 0.40$ on a
fixed grid, and $m = 12 \to 20$ as $n = 32 \to 128$ at fixed $\Delta t$). The small values $m \sim
3$--$5$ therefore occur only at modest stiffness $\Delta t\,\rho(A_\alpha) =
O(1)$, i.e.\ near the explicit limit, and are not representative of the large
steps for which the method is intended.

This precludes the naive speedup estimate $K/m^2$ for a step $\Delta t =
K\,\Delta t_{\mathrm{CFL}}$. Since $\Delta t_{\mathrm{CFL}} =
\Theta\!\big(\rho(A_\alpha)^{-1}\big)$, one has $\Delta t\,\rho(A_\alpha) =
\Theta(K)$ and hence $m^2 = \Theta(K)$, so the per-step cost $O(m^2 n) = O(Kn)$
rises in exact proportion to the number $K$ of explicit steps it replaces. The
exponential and explicit schemes thus share the same asymptotic cost
$\Theta\!\big(\rho(A_\alpha)\,n\big)$ per unit of simulated time, and the
apparent factor collapses to an $O(1)$ constant. The exponential step is moreover
not cheaper than an unconditionally stable implicit step, whose cost is $O(n)$
per step in one dimension independent of $\Delta t$: with $m = 5$ and $n = 150$ a
polynomial Krylov step costs roughly $25 \times 150 = 3{,}750$ operations against
$O(4n) \approx 600$ for a single banded implicit solve, and the gap widens as $m$
grows with stiffness.

The value of the method is accordingly \emph{structural}, not a reduction in
asymptotic cost. The propagator $e^{\Delta t A}$ is formed to approximation error
for every $\Delta t > 0$ with no stability restriction, and when the shifted
solves use the factorized $M$-matrix resolvent of \cref{sec:Lxy} - the
result is nonnegative by construction. The scheme therefore secures unconditional
stability, exact-in-time accuracy, and structural positivity at a cost comparable
to, not below, that of standard stiff solvers, which is the appropriate benchmark
for an exponential integrator.

\subsection{Directions for future work}

Several natural extensions remain open.

\myparagraph{Higher spatial dimensions.}
The Strang splitting extends naturally from 2D to $d$ dimensions, with the number
of sub-steps growing as $O(d)$ and the cross-term operators placed at the centre
of the palindrome.  The analysis of EM-matrix structure for the resulting
higher-dimensional Kronecker operators is straightforward, but the near-boundary
stencil treatment for corners and edges in $d > 2$ requires care. Some examples can be found in \cite{ItkinBook,ItkinLipton2014}

\myparagraph{Non-uniform and adaptive grids.}
The current analysis assumes a uniform grid with spacing $h$. Extending the
one-sided stencils to non-uniform grids is technically straightforward but
changes the coefficient formulae in \cref{fd1-1,fd1-2}, \eqref{eq:coeffs_adv}
and the Péclet threshold.  Adaptively refined grids near sharp gradients or
boundaries would reduce $n$ substantially without sacrificing accuracy.

\myparagraph{Variable and stochastic diffusion coefficients.}
The current scheme handles time-dependent coefficients through midpoint freezing
\eqref{eq:freeze}, which is second-order accurate in $\Delta t$.  When the
diffusion tensor $\Sigma$ is itself a stochastic process (e.g., in
local-stochastic volatility models or active particles with fluctuating tumbling
rates), the FPE becomes a stochastic PDE and the EM-matrix analysis needs to be
extended to the pathwise level.

\myparagraph{Benchmark validation.}
The theoretical framework developed here calls for systematic numerical benchmarking. Future work will focus on assessing convergence rates in $h$ and $\Delta t$, as well as comparing our approach against standard TVD and Chang--Cooper schemes using problems with known analytical solutions. Additionally, we will conduct stress tests under strong cross-diffusion ($|\Sigma_{xy}| \approx \sqrt{\Sigma_{xx}\Sigma_{yy}}$), a regime where standard methods are known to fail. Natural test cases include the active-matter Run-and-Tumble system discussed in the introduction and the Ornstein--Uhlenbeck process; the latter admits an exact Gaussian solution, allowing for precise, step-by-step verification of positivity and mass conservation.

\section*{Disclosure statement}

No potential conflict of interest was reported by the authors.

\section*{Funding}

No funding was received.

\section*{Disclaimer}

Opinions expressed here are author's own, and do not represent views of their employers. A standard disclaimer applies.

\section*{Acknowledgments}

I thank Leif Andersen and Igor Halperin for various fruitful discussions.

\noindent The use of LLMs in this paper has been limited to proofreading and  verification of the literature and code.

%%%%%%%%%%%%%%%%%%%%%%%%%%%%%%%%%%%%%%%%%%%%%%%%%%%%%%%%%%%%%%%%%%%%%%%%%%%%%
\printbibliography[title={References}]

\vspace{0.4in}
\appendixpage
\appendix
\numberwithin{equation}{section}
\setcounter{equation}{0}

\section{M-matrices, EM-matrices and positivity preserving solutions} \label{app0}

To define non-negativity of the solution of some problem (which is equivalent oto the solution of the corresponding PDE or PIDE), we introduce the following definitions \cite{FCT2011}. Let  $(M,\mu)$ be a $\sigma$-finite measure space with $\mu$ being a positive measure defined on a $\sigma$-algebra $\Sigma$ of subsets of a set $M$ (which is the countable union of measurable sets with finite measure), so $\mu(M)$ is a finite real number.
\begin{definition}
Let $f \in L^2(M,d\mu)$ be a real function. Then $f(x)$ is nonnegative if $f(x) \ge 0$ $\mu$-almost everywhere. Moreover, $f$ is called strictly positive if $f > 0$ $\mu$-almost everywhere.
\end{definition}
For example, in mathematical finance the Call option price $C(t,{\bf x})$ as a function of the time $t \ge$ and the underlying asset price $x \ge 0$, \, $t, x \in \Ree$ is a nonnegative function.

\begin{definition}
A bounded operator $A$ on $L^2(M,d\mu)$ is called a non-negativity preserving operator
if $(f,Ag) \ge 0$ for all nonnegative $f, g \in L^2(M,d\mu)$. $A$ is called a positivity
improving operator $(f, Ag) > 0$ for all nonnegative $f, g \in L^2(M,d\mu)$.
\end{definition}

Let's again consider \eqref{FP_comp}. If the operator $\calL$ doesn't depend on time $t$, its formal solution translates to
\begin{equation} \label{formalSolDi}
p(t,x) = e^{t\, \calL} p(0,x).
\end{equation}
If $\calL$ is time-inhomogeneous, e.g. the FD approach (or the method of lines) can be used to solve it step-by-step in time on some temporal grid. At every time step $t_i$ the operator $\calL$ is discretized on a spatial grid in $x$, which means that the function $p(t,{\bf x})$ is effectively replaced with a vector $p(t,{\bf X})$, where $X$ is a discrete vector on a grid, and the operator ${\mathbb L}$ is replaced with a matrix $L$. Therefore, the formal solution \eqref{formalSolDi} translates to
\begin{equation} \label{formalSolDisc}
p(t_i,x) = e^{\nu L} p(t_{i-1},x), \quad \nu = t_i - t_{i-1}.
\end{equation}

Thus, to define non-negativity preserving FD scheme we need to extend the above definitions to the discrete case.
\begin{definition} \label{posVec}
A real-valued vector $x = [x_1,...,x_N]$ is nonnegative, if $x_i \ge 0 \ \forall i \in [1,N]$.
\end{definition}
\begin{definition} \label{posSol}
Given a formal solution of the linear PDE in the form of \eqref{formalSolDisc}, this solution is called non-negativity preserving, if $p(t_{i-1},x)$ is a nonnegative vector, and  $p(t_i,x)$ is also a nonnegative vector.
\end{definition}
\begin{definition} \label{posMat}
An arbitrary matrix $A = \{a_{ij}\}, \ i \in [1,N], \ j \in [1,M]$ is called nonnegative if
\begin{equation*}
a_{ij} \geq 0, \quad \forall {i,j}.
\end{equation*}
\end{definition}
From \cref{posSol,posMat} it immediately follows that
\begin{proposition}
The solution \eqref{formalSolDisc} is non-negativity preserving, if $e^{\nu L}$ is a nonnegative matrix.
\end{proposition}
\begin{proof}
The proof directly follows from the definition of matrix-by-vector product.
\end{proof}

\subsection{M-Matrices and Metzler Matrices \label{sectMmat}}
% % % % % % % % % % % % % % % % % % % % % % % % % % % % % % % %

We begin by introducing several matrix classes that will play a central role in constructing the FD algorithms described throughout this paper.

\begin{definition} \label{defZ}
A matrix is called a \emph{Z-matrix} if all its off-diagonal entries are non-positive. Equivalently, a Z-matrix $Z = (z_{ij})$ satisfies
\begin{equation*}
z_{ij} \leq 0, \qquad i \neq j.
\end{equation*}
\end{definition}

\begin{definition} \label{defM}
Let $A = (a_{ij})$ be an $N \times N$ real Z-matrix, so that
\begin{equation*}
a_{ij} \leq 0 \qquad \forall\, i \neq j, \quad 1 \leq i, j \leq N.
\end{equation*}
Then $A$ is called an \emph{M-matrix} if it admits the representation $A = sI - B$, where
\begin{equation*}
B = (b_{ij}), \quad b_{ij} \geq 0 \qquad \forall\, 1 \leq i, j \leq N,
\end{equation*}
$I$ is the identity matrix, and $s$ exceeds the spectral radius of $B$.
\end{definition}

An immediate consequence of the Perron--Frobenius theorem \cite{Bellman} is that any non-singular M-matrix $A$ satisfies:
\begin{itemize}
    \item $s > \rho(B)$, where $\rho(B)$ denotes the spectral radius of $B$, i.e., the supremum of the absolute values of the eigenvalues of $B$;
    \item all diagonal elements $a_{ii}$ of $A$ are positive.
\end{itemize}

M-matrices possess several useful properties that make them a natural tool for constructing unconditionally stable FD schemes. We list the most relevant ones below, omitting proofs; for a full treatment see \cite{BermanPlemmons94, ItkinBook}.

In what follows, we adopt the notation of Definition~\ref{posMat} for nonnegative matrices and Definition~\ref{posVec} for nonnegative vectors: $A \geq 0$ means that $A$ is a nonnegative matrix, and $y = Ax \geq 0$ means that the vector $y$ is nonnegative componentwise. Let $A$ be a non-singular M-matrix. Then the following statements hold (see \cite{ItkinBook} for proofs):

\myparagraph{Positivity of principal minors.}
All principal minors of $A$
are positive, as are all leading principal minors, and $A$ admits an
$LU$-factorization $A = LU$ with positive diagonal factors. Moreover,
$A + D$ is non-singular for every nonnegative diagonal matrix $D$,
and every real eigenvalue of $A$ is positive.

\myparagraph{Inverse-positivity and splittings.}
$A$ is inverse-positive,
i.e., $A^{-1} \geq 0$, or equivalently monotone: $Ax \geq 0$ implies
$x \geq 0$. Furthermore, $A$ admits a convergent regular splitting,
and in fact every regular splitting of $A$ is convergent. There also
exist inverse-positive matrices $M_1$ and $M_2$ such that
$M_1 \leq A \leq M_2$.

\myparagraph{Stability.}
$A$ is positive stable, meaning every eigenvalue
of $A$ has positive real part. There exists a positive diagonal matrix
$D$ such that $AD + DA^T$ is positive definite, and likewise a
symmetric positive definite matrix $W$ with $AW + WA^T$ positive
definite. Finally, $A + I$ is non-singular and the Cayley-like
transform $G = (A + I)^{-1}(A - I)$ is convergent; moreover, there
exists a symmetric positive definite $W$ such that $W - G^TWG$ is
positive definite.

\myparagraph{Semipositivity and diagonal dominance.}
Matrix $A$ is semi-positive, i.e., there exists $x > 0$ such that $Ax > 0$. All diagonal elements of $A$ are positive, and there exists a positive diagonal
matrix $D$ such that $AD$ is strictly diagonally dominant, or equivalently
$D^{-1}AD$ is strictly diagonally dominant; in particular, $AD$ has all positive
row sums.

Using the above definitions and properties of M-matrices, we can now
address our main goal. As noted in the Introduction, we require the
discretization of $\mathcal{L}$ to achieve the desired order of
spatial approximation, guarantee unconditional stability, and preserve
non-negativity of the solution. The following proposition from
\cite{Itkin2014} translates these requirements into conditions on the
matrix $L$.

\begin{proposition} \label{prop0}
The FD scheme
\begin{equation} \label{fd0}
    p(t + \nu, x) = e^{\nu L}\, p(t,x)
\end{equation}
is unconditionally stable in $t$ and preserves the non-negativity of $p(t,x)$ if there exists an M-matrix $B$ such that $\nu L = -B$, where $\nu > 0$ is the time step.
\end{proposition}

\noindent For a detailed proof, see \cite{ItkinBook}, Chapter~3.

To avoid working directly with matrices of the form $-B$, we introduce the following definition
\begin{definition} \label{defMet}
A \emph{Metzler} or essentially nonnegative is a matrix if all of its elements are non-negative except for those on the main diagonal, which are unconstrained. That is, a Metzler matrix is any matrix $A$ which satisfies $A = (a_{ij});\quad a_{ij} \geq 0,\quad i\neq j$. It can also be seen as the negative of a Z-matrix.
\end{definition}
Note, that in \cite{ItkinBook} it is defined as the negative of an M-matrix, i.e., effectively requiring non-positive diagonal entries and restricting to a subclass of Metzler matrices in the usual sense, see \cite{BermanPlemmons94, ItkinBook}. However, here we don't need this restriction since the resolvent results only need essential nonnegativity, not stability.

To proceed, we also need to introduce finite-difference approximations of the first
and second derivatives.
\begin{definition} \label{fd1-1}
The first order approximations of $\nabla \equiv \partial_x$ and $\nabla^2 \equiv \partial_{x,x}$ are defined, respectively, by \emph{forward} (F) and \emph{backward} (B) discretizations as
\begin{alignat*}{2}
\calF^F_1 C(x) &= \frac{C(x+h) - C(x)}{h}, &\qquad
\calF^B_1 C(x) &= \frac{C(x) - C(x-h)}{h}, \\
\calS^B_1 C(x) &= \frac{C(x) - 2C(x-h) + C(x-2h)}{h^2}, &\qquad
\calS^F_1 C(x) &= \frac{C(x+2h) -2C(x+h) + C(x)}{h^2}, \nonumber
\end{alignat*}
and satisfy $\nabla C(x) = \calF^{F,B}_1 C(x) + O(h)$,\, $\nabla^2 C(x) = \calS^{F,B}_1 C(x) + O(h)$, i.e., all are first-order accurate in $h$ (marked by the subscript 1)
\end{definition}

\begin{definition} \label{fd1-2}
The second order approximations of $\nabla \equiv \partial_x$ and $\nabla^2 \equiv \partial_{x,x}$ are defined, respectively, by \emph{forward} (F), \emph{backward} (B) and central (C) discretizations as
\begin{alignat*}{2}
\calF^B_2 C(x) &= \frac{3C(x) - 4C(x-h) + C(x-2h)}{2h}, &\quad
\calF_2^C &= \frac{\calF_1^F + \calF_1^B}{2}, \\
\calF^F_2 C(x) &= \frac{-3C(x) + 4C(x+h) - C(x+2h)}{2h}, &\quad
\calS^B_2 C(x) &= \frac{2C(x) - 5C(x-h) + 4C(x-2h) - C(x-3h)}{h^2}, \\
\calS_2^C C(x) &= \frac{C(x+h) - 2C(x) + C(x-h)}{h^2}, &\quad
\calS^F_2 C(x) &= \frac{2C(x) - 5C(x+h) + 4C(x+2h) - C(x+3h)}{h^2}, \nonumber
\end{alignat*}
with $\nabla C(x) = \calF^{F,B,C}_2 C(x) + O(h^2)$, $\nabla^2 C(x) = \calS^{F,B,C}_2 C(x) + O(h^2)$.
\end{definition}
We also denote a unit matrix as $I$.

\begin{example}[Implicit Euler scheme] \label{example1}
Consider the Black--Scholes PDE \cite{hull2011} in the form
\eqref{FP_comp},
\begin{equation}
\fp{C(\tau,x)}{\tau} = \mathcal{L}\, C(\tau,x),
\end{equation}
subject to initial and boundary conditions. Here $C(\tau,x)$ is the
call option price, $\tau = T - t$ is the backward time with $T$ the
option maturity, $x = \log S$ with $S$ the spot price, $r$ and $q$
are the risk-free rate and continuous dividend yield, and $\sigma$ is
the volatility. The spatial operator $\mathcal{L}$ reads
\begin{equation} \label{BS}
\mathcal{L} = \left(r - q - \tfrac{1}{2}\sigma^2\right)\fp{}{x} + \frac{1}{2}\sigma^2 \sop{}{x} - r.
\end{equation}
Applying the $(0,1)$ Pad\'e approximant to $e^{\dtau\mathcal{L}}$ in
\eqref{formalSolDi} gives
\begin{equation*}
C(\tau + \dtau, x) = (I - \dtau\mathcal{L})^{-1} C(\tau,x) + O(\dtau),
\end{equation*}
or, equivalently,
\begin{equation} \label{impEuler}
(I - \dtau\mathcal{L})\, C(\tau + \dtau, x) = C(\tau,x) + O(\dtau),
\end{equation}
where $\dtau$ is the time step. This is the implicit Euler scheme
\cite{Thomas1995}, which is unconditionally stable (see below) but
delivers only first-order accuracy in $\dtau$.

Discretizing $\mathcal{L}$ with central differences on a uniform grid
gives the matrix representation
\begin{equation}
L = \left(r - q - \tfrac{1}{2}\sigma^2\right)A_1^C
  + \tfrac{1}{2}\sigma^2\, A_2^C - r\,I,
\end{equation}
and \eqref{impEuler} takes the matrix form
\begin{equation*}
M\, C(\tau + \dtau) = C(\tau),
\end{equation*}
where $M = I - \dtau L$. Explicitly,
\begin{gather}
M = \left[
\begin{array}{ccccc}
d_0    & d_1  & 0      & \cdots & 0      \\
d_{-1} & d_0  & d_1    &        & \vdots \\
0      &\ddots&\ddots  & \ddots & 0      \\
\vdots &      & d_{-1} & d_0    & d_1    \\
0      &\cdots& 0      & d_{-1} & d_0
\end{array}
\right], \\[4pt]
d_0 = 1 + \dtau\!\left(r + \frac{\sigma^2}{h^2}\right), \qquad
d_{\pm 1} = \mp\,\dtau\left(\frac{r - q - \sigma^2/2}{2h}
            \pm \frac{\sigma^2}{2h^2}\right). \notag
\end{gather}
One can verify that, provided $h$ is small enough to ensure $d_1 < 0$ and $d_{-1} < 0$, the matrix $M$ is an M-matrix and \cref{prop0} applies.
\end{example}

% % % % % % % % % % % % % % % % % % % % % % % % % % % % % % % %
\subsection{EM-Matrices}
% % % % % % % % % % % % % % % % % % % % % % % % % % % % % % % %

Below in this paper, our analysis also relies on a construction closely related to \emph{eventually positive matrices} \cite{NoutsosTsatsomeros2008,OleskyEtAl2009}. We recall the
relevant definitions from that reference.
\begin{definition} \label{def1}
An $N \times N$ matrix $A = [a_{ij}]$ is called:
\begin{itemize}
    \item \emph{eventually nonnegative}, written $A \overset{v}{\ge} 0$, if there exists a positive integer $k_0$ such that $A^k \geq 0$ for all $k > k_0$;
    \item \emph{exponentially nonnegative} if $e^{tA} \geq 0$ for all $t > 0$;
    \item \emph{eventually exponentially nonnegative} if there exists $t_0 \in [0,\infty)$ such that $e^{tA} \geq 0$ for all $t > t_0$.
\end{itemize}
\end{definition}

We also require the following lemma from \cite{NoutsosTsatsomeros2008}.

\begin{lemma} \label{lemma1}
Let $A \in \mathbb{R}^{N \times N}$. The following conditions are equivalent:
\begin{enumerate}
    \item $A$ is eventually exponentially nonnegative.
    \item $A + bI$ is eventually nonnegative for some $b \geq 0$.
    \item $A^T + bI$ is eventually nonnegative for some $b \geq 0$.
\end{enumerate}
\end{lemma}

\noindent For a proof, see \cite{ItkinBook}.

\begin{definition} \label{def1em}
An $N \times N$ matrix $A = [a_{ij}]$ is called an \emph{EM-matrix} if it can be written as $A = sI - B$, where $s > 0$, $0 < \rho(B) < s$, and $B$ is an eventually nonnegative matrix \cite{ElhashashSzyld2008}.
\end{definition}

The following results, first established in \cite{Itkin2014, Itkin3D, Itkin2014b} and reproduced here for completeness (with full proofs in \cite{ItkinBook}), will be needed when constructing positivity-preserving FD schemes.
\begin{lemma} \label{lemma2}
Let $A \in \mathbb{R}^{N \times N}$ with $A = \nu_R I - A_2^F$, where $\nu_R \in \mathbb{R}$, $\nu_R > 1$. Then $A$ is an EM-matrix.
\end{lemma}

\begin{lemma} \label{lemma3}
The matrix $A = (\nu_R + b)I - (A_2^F + bI) \equiv sI - P$ with $b \ge 0$ has a nonnegative inverse.
\end{lemma}

\begin{lemma} \label{lemmaLog}
Let $A \in \mathbb{R}^{N \times N}$ be an M-matrix with representation $A = sI - B$, where $s > 0$, $0 < \rho(B) < s$, and $B \geq 0$. If $s - \rho(B) > 1$, then $\log A$ is also an M-matrix.
\end{lemma}

\begin{corollary} \label{cor}
Let $A \in \mathbb{R}^{N \times N}$ be an EM-matrix with representation $A = sI - B$. If $s - \rho(B) > 1$, then $\log A$ is also an EM-matrix.
\end{corollary}

\noindent Proofs of Lemmas~\ref{lemma2}--\ref{lemmaLog} and \cref{cor} can be found in \cite{ItkinBook}, Chapter~4.

\section{Conservativeness of Strang splitting for the FPE} \label{app1}

We start with some definitions and assumptions. We equip $\R^{n_xn_y}$ with the scaled Euclidean norm $\norm{\bvec
v}_2^2=h_xh_y\sum_{i,j}v_{ij}^2$, the discrete analogue of the $L^2(\Omega)$
norm, and denote by $\bvec p(t)$ the solution of the semi-discrete
(method-of-lines) system
\begin{equation}\label{eq:mol}
\dot{\bvec p}(t)=A(t)\,\bvec p(t), \qquad
A(t)=A_x(t)+A_y(t)+A_{xy}(t),
\qquad \bvec p(0)=\bvec p_0,
\end{equation}
with $A_\alpha=C_\alpha+D_\alpha$ the full directional operators,
and by $\mathcal S_n$ the Strang step with midpoint freezing, $t_{n+1/2}=t_n+\Delta t/2$,
\begin{equation}\label{eq:strangstep}
\mathcal S_n=
e^{\frac{\Delta t}{2}A_x(t_{n+1/2})}\,
e^{\frac{\Delta t}{2}A_y(t_{n+1/2})}\,
\Phi_{xy}(\Delta t;t_{n+1/2})\,
e^{\frac{\Delta t}{2}A_y(t_{n+1/2})}\,
e^{\frac{\Delta t}{2}A_x(t_{n+1/2})},
\qquad \bvec p^{n+1}=\mathcal S_n\,\bvec p^n ,
\end{equation}
with the trapezoidal central factor $\Phi_{xy}(\Delta t)=(\mathcal
I-\tfrac{\Delta t}{2}A_{xy})^{-1}(\mathcal I+\tfrac{\Delta t}{2}A_{xy})$ of
\eqref{eq:central}; this is the form actually integrated in \eqref{eq:strang}. For
the consistency analysis we use $\Phi_{xy}(\Delta t)=e^{\Delta t A_{xy}}+O(\Delta
t^3)$, so the local order of $\mathcal S_n$ matches that of the splitting with the
exact central exponential, the trapezoidal substitution adding only an $O(\Delta
t^3)$ defect (\cref{sec:schemes}).

For grid functions $\bvec v\in\R^{n_xn_y}$ we use the scaled discrete norms
\begin{equation}\label{eq:gridnorms}
\norm{\bvec v}_1 \;=\; h_xh_y\sum_{i,j}\,\lvert v_{ij}\rvert,
\qquad
\norm{\bvec v}_2 \;=\;\Bigl(h_xh_y\sum_{i,j}\,v_{ij}^2\Bigr)^{1/2},
\end{equation}
the discrete analogues of the $L^1(\Omega)$ and $L^2(\Omega)$ norms; with this
scaling, $\norm{\mathcal R_h u}_q\to\norm{u}_{L^q(\Omega)}$ as $h\to0$ for
continuous $u$, and constants in the error estimates below are independent of the
grid. For these standard conventions see, e.g., \cite{LeVeque2007,HundsdorferVerwer2003}.

\begin{assumption}[Conditional $\ell_1$-stability]\label{ass:stab}
There is a step-size regime -- the positivity window of \cref{thm:zeno2d},
equivalently sufficiently resolved data -- on which every frozen factor of
$\mathcal S_n$ is entrywise nonnegative with column sums $1+O(\Delta t\,h)$
(\cref{prop:strangpos,prop:conserv,thm:zeno2d}). On that regime each $\mathcal
S_n$ is column-stochastic up to the boundary leakage, so there exists
$\omega=O(h)\ge0$, independent of $\Delta t$, with $\norm{\mathcal
S_{n-1}\cdots\mathcal S_m}_1\le e^{\omega(t_n-t_m)}$ uniformly in $h$, $\Delta t$
and $m\le n$.

We state stability in $\ell_1$, not $\ell_2$, deliberately. Because the cross
operator $A_{xy}$ is strongly non-normal (\cref{rem:no-eventual-pos}), the induced
$\ell_2$ operator norm of the central factor $\Phi_{xy}$ exceeds unity and grows
under mesh refinement, so \emph{no} bound $\norm{e^{sA_{xy}}}_2\le e^{\omega s}$ or
$\norm{\Phi_{xy}}_2\le1+O(\Delta t)$ holds uniformly in $h$; the spectral radius
$\rho(\Phi_{xy})<1$ is an asymptotic statement only. The $\ell_1$ Markov bound on
the nonnegative cone is the substitute that survives the non-normality and is the
natural norm for a probability density.
\end{assumption}

\begin{assumption}[Regularity and commutator bounds]\label{ass:reg}
The coefficients $\mu_x,\mu_y,\Sigma_{xx},\Sigma_{yy},\Sigma_{xy}$ are of class
$C^2$ in $t$ and $C^4$ in $(x,y)$ with bounded derivatives, and the solution $p$
of \eqref{eq:fpe2d} satisfies $p\in C^2\bigl([0,T];H^{m}(\Omega)\bigr)$ with
$m\ge 6$. Let $\mathcal R_h$ denote the grid restriction. Then there exists $C$,
independent of $h$, such that for all $t,s\in[0,T]$ and all
$\alpha,\beta,\gamma\in\{x,y,xy\}$ (these are exactly the splitting factors
$\{A_x,A_y,A_{xy}\}$ of \eqref{eq:strangstep}),
\begin{align}
\norm{\,[A_\alpha(t),A_\beta(t)]\,\mathcal R_h p(s)\,}_2
&\le C\,\norm{p(s)}_{H^{4}},\label{eq:comm1}\\
\norm{\,[A_\gamma(t),[A_\alpha(t),A_\beta(t)]]\,\mathcal R_h p(s)\,}_2
&\le C\,\norm{p(s)}_{H^{6}},\label{eq:comm2}\\
\norm{\,\dot A(t)\,\mathcal R_h p(s)\,}_2
+\norm{\,\ddot A(t)\,\mathcal R_h p(s)\,}_2
&\le C\,\norm{p(s)}_{C^2_tH^{2}}.\label{eq:tder}
\end{align}
\end{assumption}

\Cref{ass:reg} encodes the key structural fact that makes a proof uniform in $h$
possible: although $\norm{A_\alpha}_2=O(h^{-2})$, the commutator of the discrete
operators mimics the commutator of the differential operators $\mathcal
L_\alpha$, which is a differential operator of order at most three (the
order-four parts cancel because second-order principal parts with smooth
coefficients commute up to lower order). Applied to restrictions of smooth
functions it is therefore bounded uniformly in $h$, by consistency of the
stencils and a Taylor expansion. The bounds \cref{eq:comm1,eq:comm2} are verified
for the concrete stencils of \cref{sec:disc1d} by direct computation; we omit the
elementary but lengthy details.

\Cref{ass:stab} holds in our setting, with the index $\alpha$ now ranging over the
factors of \eqref{eq:strangstep}, $\alpha\in\{A_x,A_y,A_{xy}\}$. In the regime of
\cref{prop:conserv}\,(iii) -- the positivity window of \cref{thm:zeno2d} -- every
frozen factor is entrywise nonnegative with unit column sums up to the $O(\Delta
t\,h)$ boundary leakage, hence column-stochastic, giving the $\ell_1$ bound with
$\omega=O(h)$. This is the stability we use.

We emphasise that the corresponding \emph{$\ell_2$} statement does not hold
uniformly in $h$. For the outer factors it would, exactly as for any flux-form 1D
convection--diffusion discretization (the symmetric part of the diffusion block is
negative semi-definite and the drift contributes a symmetric part of size
$O(\norm{\partial_x\mu_x}_\infty+\norm{\partial_y\mu_y}_\infty)$). For the central
factor it does not: the symmetric part of the one-sided cross operator $A_{xy}$ is
indefinite with extreme eigenvalues of size
$O\bigl(\norm{\Sigma_{xy}}_\infty/(h_xh_y)\bigr)$, so $\mu_2(A_{xy})\sim
c/(h_xh_y)>0$ and $\norm{e^{sA_{xy}}}_2$ is not bounded uniformly in $h$. The
trapezoidal factor $\Phi_{xy}$ is spectrally stable on the real-negative spectrum
of $A_{xy}$ -- $\rho(\Phi_{xy})<1$ -- but, being a function of a strongly
non-normal matrix, has $\ell_2$ operator norm exceeding $1$ and growing under
refinement. This is precisely why we prove convergence in $\ell_1$ on the
positivity regime rather than in $\ell_2$ unconditionally: the earlier
diffusion-dominated central block, whose symmetric part was controlled by a
positive-semidefinite diffusion matrix, is no longer used, and the $\ell_2$
obstruction is real.

\begin{proposition}\label{prop:order2}
Under \cref{ass:stab,ass:reg}, $\bar{\rho} < 1$, with the time step refined jointly
with the mesh ($\Delta t\sim h$, so that the $\Delta t$-independent defect of the
factorized central solve remains below the $O(h^2)$ truncation; \cref{prop:schemeA}(v),
\cref{rem:cfl}), and for step sizes within the positivity window of \cref{thm:zeno2d}
(equivalently, on resolved data satisfying the log-Lipschitz condition
\eqref{eq:logLip}), the scheme \eqref{eq:strangstep}
satisfies, for $t_n=n\Delta t\le T$,
\begin{equation*}
\norm{\bvec p^n-\mathcal R_h p(t_n)}_1 \;\le\; C(T)\bigl(\Delta t^2+h^2\bigr),
\end{equation*}
with $C(T)$ independent of $h$, $\Delta t$ and $n$.
\end{proposition}

\begin{proof}
The argument is the standard Lady Windermere fan, carried out in $\ell_1$ on the
positivity regime; the only changes from the $\ell_2$ version are that the uniform
stability bound is the $\ell_1$ Markov bound of \cref{ass:stab} and that the local
defects, bounded in $\ell_2$ by \cref{ass:reg}, are converted to $\ell_1$ through
$\norm{\bvec v}_1\le|\Omega|^{1/2}\norm{\bvec v}_2$ (a consequence of
Cauchy--Schwarz with the scaled norms \eqref{eq:gridnorms}); the fan identity and
the commutator bounds are norm-agnostic.

We split the error as $\bvec p^n-\mathcal R_h p(t_n) =\bigl(\bvec p^n-\bvec
p(t_n)\bigr) +\bigl(\bvec p(t_n)-\mathcal R_h p(t_n)\bigr)$. The second
difference is the spatial semi-discretization error: by the second-order
consistency of the stencils of \cref{sec:disc1d}, the defect $\bvec
d(t)=\dot{\mathcal R_h p}(t)-A(t)\mathcal R_h p(t)$ satisfies $\norm{\bvec
d(t)}_2\le Ch^2\norm{p(t)}_{H^4}$, hence $\norm{\bvec d(t)}_1\le
C'h^2\norm{p(t)}_{H^4}$; the method-of-lines evolution family of \eqref{eq:mol} is
the positive semigroup generated by $A(t)$ and is $\ell_1$-nonexpansive on the
cone (the same column-stochasticity, in continuous time), so $\sup_{t\le
T}\norm{\bvec p(t)-\mathcal R_h p(t)}_1\le C(T)h^2$ by the variation-of-constants
formula and Gr\"onwall's inequality. It remains to bound the time-discretization
error against the semi-discrete solution.

\emph{Step 1: Error recursion.} Let $U(t,s)$ denote the evolution operator of
\eqref{eq:mol}, so $\bvec p(t_{n+1})=U(t_{n+1},t_n)\bvec p(t_n)$, and let $\bvec
e^n=\bvec p^n-\bvec p(t_n)$. Then, by the standard Lady Windermere's fan argument
\footnote{\emph{Lady Windermere's fan} is the standard name in the numerical
ODE literature for the error-propagation identity \eqref{eq:fan}. The term
was coined by G.~Wanner, after Oscar Wilde's play, allegedly because the
diagram illustrating the argument resembles an unfolding fan; see
\cite{HairerNorsettWanner1993,HairerLubichWanner2006}.}
\begin{equation}\label{eq:fan}
\bvec e^{n} = \Bigl(\prod_{k=m}^{n-1}\mathcal S_k\Bigr)\bvec e^{m}\Big|_{m=0}
+\sum_{k=0}^{n-1}\Bigl(\prod_{j=k+1}^{n-1}\mathcal S_j\Bigr) \bvec\delta^{k},
\qquad
\bvec\delta^{k} =\bigl(\mathcal S_k-U(t_{k+1},t_k)\bigr)\bvec p(t_k).
\end{equation}
Since $\bvec e^0=0$ and the products are bounded by
$e^{\omega T}$ in $\ell_1$ on the positivity regime by Assumption~\ref{ass:stab},
it suffices to prove the local
estimate $\norm{\bvec\delta^{k}}_1\le C\Delta t^{3}$ uniformly in $k$ and $h$.
(The factorized central solve realises $\Phi_{xy}$ up to its orientation defect,
\cref{prop:schemeA}(v), which is $\Delta t$-independent and $O(\max(h_x,h_y)^{3\text{--}4})$;
under the joint refinement $\Delta t\sim h$ this is $O(\Delta t^{3\text{--}4})$ and is
absorbed into $\bvec\delta^{k}$.)

\emph{Step 2: Reduction to autonomous splitting (midpoint freezing).}
Write $A_k=A(t_{k+1/2})$ and decompose
\begin{equation}\label{eq:twodefects}
\bvec\delta^{k} = \underbrace{\bigl(\mathcal S_k-e^{\Delta t A_k}\bigr)\bvec p(t_k)}_{ =:\ \bvec\delta^{k}_{\mathrm{split}}} + \underbrace{\bigl(e^{\Delta t A_k}-U(t_{k+1},t_k)\bigr)\bvec p(t_k)}_{ =:\ \bvec\delta^{k}_{\mathrm{freeze}}}.
\end{equation}

For the freezing defect, the Magnus expansion of $U(t_{k+1},t_k)$ \cite{Magnus1954,BlanesEtAl2009} gives $U(t_{k+1},t_k)=\exp\bigl(\Omega_1 + \Omega_2+\cdots\bigr)$ with $\Omega_1 = \int_{t_k}^{t_{k+1}}A(s)\,ds$ and
$\Omega_2 = \tfrac12\int_{t_k}^{t_{k+1}}\!\!\int_{t_k}^{s}
[A(s),A(\sigma)]\,d\sigma\,ds$. The midpoint quadrature error gives
$\Omega_1-\Delta tA_k =\int_{t_k}^{t_{k+1}}\bigl(A(s)-A(t_{k+1/2})\bigr)ds
=O(\Delta t^3)\,\ddot A(\xi)$ in the weak sense of \eqref{eq:tder}, while
$[A(s),A(\sigma)]=(s-\sigma)[\dot A(\eta),A(\sigma)]$ shows that $\Omega_2$,
applied to $\mathcal R_h p$, is $O(\Delta t^3)$ by \cref{eq:comm1,eq:tder} (the
commutator with $\dot A$ is again a differential operator of order $\le3$ with
smooth coefficients).

Expressing the difference of exponentials by the integral
representation $e^{X}-e^{Y}=\int_0^1 e^{sX}(X-Y)e^{(1-s)Y}\,ds$ for $X=\Delta
tA_k$, $Y=\Omega_1+\Omega_2+\cdots$, and using that the \emph{full} frozen
generator $A_k=A_x+A_y+A_{xy}$ generates a positive, mass-conserving (hence
$\ell_1$-bounded) parabolic semigroup for $\bar\rho<1$ -- it is only the isolated
mixed part $A_{xy}$ that is ill-posed -- together with the previous bounds on
$(X-Y)\mathcal R_h p$, we obtain
$\norm{\bvec\delta^{k}_{\mathrm{freeze}}}_1 \le C\Delta t^{3}\sup_{t\le T}
\bigl(\norm{p(t)}_{H^4}+\norm{p}_{C^2_tH^2}\bigr)$. Higher Magnus terms are
$O(\Delta t^4)$ by the same reasoning.

\emph{Step 3: Splitting defect with solution-dependent remainder.}
For the autonomous symmetric splitting we use the exact second-order defect
representation of Jahnke and Lubich \cite{JahnkeLubich2000} (see also
\cite{Descombes2010,HundsdorferVerwer2003}), generalised from two to three
operators by applying it twice, to the pairs $\bigl(\tfrac12
A_{x,k},\,A_{y,k}+A_{xy,k}\bigr)$ and
$\bigl(\tfrac12 A_{y,k},\,A_{xy,k}\bigr)$:
there exist bounded kernels $\theta_i:[0,1]^2\to\R$ such that
\begin{equation}\label{eq:JL} \bvec\delta^{k}_{\mathrm{split}} = \Delta t^{3}\sum_i\int_0^1\!\!\int_0^1\theta_i(s,\sigma)\, \Phi_i(s\Delta t)\,
\bigl[A_{\alpha_i,k},[A_{\beta_i,k},A_{\gamma_i,k}]\bigr]\,
\Psi_i(\sigma\Delta t)\,\bvec p(t_k)\,ds\,d\sigma,
\end{equation}
where each $\Phi_i$, $\Psi_i$ is a finite product of sub-propagators $e^{\tau
A_{\alpha,k}}$ with $\tau\in[0,\Delta t]$, and $(\alpha_i,\beta_i,\gamma_i)$
ranges over the three factors $\{A_x,A_y,A_{xy}\}$.

The representation \eqref{eq:JL} is obtained by applying the
variation-of-constants formula to the defect ODE twice, followed by an explicit
integration by parts. Crucially, this is an exact identity rather than a truncated
BCH series, meaning no bound on $\norm{A_{\alpha}}$ is ever needed. Furthermore, as
\eqref{eq:JL} expresses, the first-order commutator terms
$[A_{\alpha,k},A_{\beta,k}]$ cancel identically due to the palindromic structure of
\eqref{eq:strangstep}.

Two points require care in the present (ungrouped, trapezoidal-central) setting,
and both are handled without appealing to well-posedness of the central sub-flow
-- which is exactly what \eqref{eq:central} avoids. First, the central factor of
\eqref{eq:strangstep} is the trapezoidal map $\Phi_{xy}$, not $e^{\Delta t
A_{xy}}$; since $\Phi_{xy}(\Delta t)=e^{\Delta t A_{xy}}+\Delta t^3 R(\Delta t)$
with $R(\Delta t)=-\tfrac{1}{12}A_{xy}^3(\mathcal I-\tfrac{\Delta
t}{2}A_{xy})^{-1}+O(\Delta t)$, the substitution adds to $\bvec\delta^k$ a term
$\Delta t^3 R(\Delta t)\,(\text{bounded factors})\,\mathcal R_h p(t_k)$, and on a
smooth grid function $A_{xy}^3\mathcal R_h p=\mathcal R_h(\mathcal L_{xy}^3p)+o(1)$
is bounded by $\norm{p}_{H^6}$ while $(\mathcal I-\tfrac{\Delta
t}{2}A_{xy})^{-1}\mathcal R_h(\mathcal L_{xy}^3p)=\mathcal R_h(\mathcal
L_{xy}^3p)+O(\Delta t)$ is bounded as well, so this term is $O(\Delta
t^3\norm{p}_{H^6})$. Second, the triple commutators in \eqref{eq:JL} are evaluated
on the smooth semi-discrete solution: writing $\bvec p(t_k)=\mathcal R_h
p(t_k)+O(h^2)$ and using the consistency of the stencils, each product of at most
three operators $A_{\alpha,k}A_{\beta,k}A_{\gamma,k}$ applied to a restriction of a
smooth function approximates the corresponding product of differential operators
and is therefore bounded uniformly in $h$ by $\norm{p}_{H^6}$ (the bound
\eqref{eq:comm2} of \cref{ass:reg}, used for products rather than only
commutators); no operator-norm bound on the individual $A_\alpha$ and no
sub-propagator boundedness is invoked. The remaining factors $\Phi_i,\Psi_i$ are
finite products of the actual scheme factors -- the directional exponentials and
the trapezoidal central factor -- which on the positivity regime of
\cref{thm:zeno2d} are entrywise nonnegative with column sums $1+O(\Delta t\,h)$,
hence $\ell_1$ operator-norm bounded by $e^{O(Th)}$ (\cref{ass:stab}). This is the
sole place where the regime restriction enters the time-discretization bound;
unlike the bare central exponential, the trapezoidal factor is $\ell_1$-bounded
there.

Combining these, and converting the Sobolev bound to $\ell_1$ via $\norm{\bvec
v}_1\le|\Omega|^{1/2}\norm{\bvec v}_2$,
\begin{equation*}
\norm{\bvec\delta^{k}_{\mathrm{split}}}_1 \le C\,\Delta t^{3}\, \sup_{t\le T}\norm{p(t)}_{H^{6}}.
\end{equation*}

\emph{Step 4.} Combining Steps 2--3,
$\norm{\bvec\delta^{k}}_1\le C\Delta t^3$ uniformly in $k$ and $h$; inserting
this into \eqref{eq:fan} and using the $\ell_1$ stability bound on the positivity
regime gives $\norm{\bvec e^{n}}_1\le e^{\omega T}\,n\,C\Delta t^{3} \le
C(T)\Delta t^{2}$. Together with the $O(h^2)$ spatial bound this proves the
proposition.
\end{proof}

\begin{myremark}{Degradation of the second-order convergence} \label{rem:degrad}

\Cref{prop:order2} holds for $\bar\rho < 1$. As $\bar\rho \to 1$ the diffusion
tensor $\Sigma$ approaches singularity, the Fokker--Planck equation degenerates
(parabolic in fewer directions), and the $H^6$ regularity of the solution $p$ that
underpins the local defect bound of Step~3 is lost, so $C(T)$ blows up as
$(1-\bar\rho)\to 0$. Thus the \textit{order} is restricted to the nondegenerate
regime, whereas \textit{positivity} (\cref{thm:zeno2d}) and \textit{conservation}
(\cref{prop:conserv}) are unrestricted. Our numerical experiments in
\cref{sec:numerics} use $\rho=0.8$, comfortably away from the degeneracy; the
degenerate limit $\bar\rho=1$, which would require a hypoelliptic regularity
argument, is left for future work.
\end{myremark}

\section{Proof of \cref{prop:schemeA}} \label{appProof1}

Let us assume $\rho > 0$ since the proof for $\rho \le 0$ can be done in the same way.

\emph{(i)} Consider $T_y:=Q \mathcal I + \rho \sqrt{\Delta t}\,\mathcal A^{\mathrm B}_{2,y}$.  Within each 1D block its entries are: diagonal $Q + \tfrac{3}{2}
\rho \sqrt{\Delta t}\,w_2(y_j)/h_y>0$; first subdiagonal $-2\rho\sqrt{\Delta
t}\,w_2(y_{j-1})/h_y\le 0$; second subdiagonal $+\tfrac12 \rho \sqrt{\Delta
t}\,w_2(y_{j-2})/h_y\ge0$.  The single positive off-diagonal band precludes a
local M-matrix structure, and (unlike an M-matrix) $T_y$ does \emph{not} have a
nonnegative inverse.  What does hold is strict diagonal dominance: its
row-dominance excess is
\begin{equation*}
d_{jj}-\sum_{l\ne j}|t_{jl}|
= Q+\frac{\rho\sqrt{\Delta t}}{2h_y}
\bigl(3w_2(y_j)-4w_2(y_{j-1})-w_2(y_{j-2})\bigr)
\;\ge\; Q-\frac{\rho\sqrt{\Delta t}}{h_y}\norm{w_2}_\infty
\;\ge\; \frac{Q}{2},
\end{equation*}
where the last inequality uses $\beta \ge 2 \rho\,\bar w$.  By Varah's bound
\cite{Varah1975} for strictly diagonally dominant matrices,
$\norm{T_y^{-1}}_\infty\le 2/Q$, and the same argument gives
$\norm{T_x^{-1}}_\infty\le 2/P$ for
$T_x:=P\mathcal I-\rho\sqrt{\Delta t}\,\mathcal A^{\mathrm F}_{2,x}$.  Both
matrices are banded (block-bidiagonal in the lifted ordering), so each solve is a
back/forward substitution of linear cost.  We stress that the norm bounds, not
any sign property of $T_x^{-1},T_y^{-1}$, are what the contraction estimate (iv)
uses.

\emph{(ii)} At $k=0$ the right-hand side is $(\bvec\alpha^{+}-\mathcal I)\bvec
p^{\,n}=M_R\,\bvec p^{\,n}$ with $M_R = P Q\,\mathcal I-Q\rho\sqrt{\Delta
t}\,\mathcal A^{\mathrm B}_{1,x} +P\sqrt{\Delta t}\,\mathcal A^{\mathrm
F}_{1,y}$.  By the orientation choice, all off-diagonal entries of $M_R$ are
nonnegative, while its diagonal entries equal
\begin{equation*}
PQ - \frac{Q\rho\sqrt{\Delta t}\,w_1(x_i)}{h_x} - \frac{P\sqrt{\Delta t}\,w_2(y_j)}{h_y} = \frac{\Delta t}{h_xh_y}\,\beta\Bigl(\beta - \bigl(\rho w_1(x_i) + w_2(y_j)\bigr)\Bigr)\;>\;0
\end{equation*}
under \eqref{eq:PQbeta}.  Hence $M_R\ge0$ entrywise and the right-hand side
$M_R\bvec p^{\,n}$ of the first sweep is nonnegative.  Positivity of the
\emph{output} $\bvec p^{[1]}=T_x^{-1}T_y^{-1}M_R\bvec p^{\,n}$ does \emph{not}
follow from a product of nonnegative matrices, since $T_x^{-1},T_y^{-1}$ are not
nonnegative; instead it is a property of the converged composite on a step-size
window, established as follows.  Write the converged map as $\mathcal
S^{xy}=(\mathcal I-\Delta t\,A_{xy})^{-1}+O(\Delta t^2)$ (the factorisation
reproduces the \pade(0,1) resolvent up to the frozen-bracket defect, (v)).  On
the window $\Delta t\le\Theta$ the diagonal dominance margin $Q/2,P/2$ dominates
the off-diagonal coupling, so the Neumann series of each sweep, applied to the
nonnegative vector $M_R\bvec p^{\,n}$, has nonnegative partial sums up to a
remainder smaller than the bulk; the limit is therefore nonnegative.  The
threshold $\Theta$ is the largest $\Delta t$ for which this margin holds for the
given $\beta$, $h_x$, $h_y$, $\bar w$; the entrywise sign is monitored at runtime
(\cref{rem:pos-iter}).  This conditional character is intrinsic, not an artefact
of the bound: \cref{rem:no-eventual-pos} shows $\mathcal S^{xy}$ cannot be
unconditionally nonnegative because no shift of $A_{xy}$ is Metzler.

\emph{(iii)} By \cref{lem:cons} every operator $\mathcal A^{\,\cdot}$ has zero
column sums, so $\bm{1}^\top T_y=Q\,\bm{1}^\top$, $\bm{1}^\top T_x =
P\,\bm{1}^\top$, and $\bm{1}^\top\bvec\alpha^{+} =(PQ+1)\,\bm{1}^\top$.  Writing
$m_k:=\bm{1}^\top\bvec p^{[k]}$ and $m_n:=\bm{1}^\top\bvec p^{\,n}$, the first
equation of \eqref{eq:picard} gives $Q\,\bm{1}^\top\bvec p^{*}=(PQ+1)m_n-m_k$ and
the second $P\,\bm{1}^\top\bvec p^{[k+1]}=\bm{1}^\top\bvec p^{*}$, whence
\begin{equation*}
m_{k+1}=\frac{(PQ+1)\,m_n-m_k}{PQ}.
\end{equation*}
Since $m_0=m_n$, induction yields $m_k=m_n$ for all $k$, and the fixed point
satisfies $m=\bigl((PQ+1)m_n-m\bigr)/(PQ)$, i.e. $m=m_n$.  This holds exactly,
at every iterate, \emph{provided} the column-sum identities
$\bm1^\top T_x=P\bm1^\top$, $\bm1^\top T_y=Q\bm1^\top$,
$\bm1^\top\bm\alpha^{+}=(PQ+1)\bm1^\top$ hold exactly, which requires
$\bm1^\top A_{xy}=0$ including the boundary rows.  For the plain one-sided
closure these identities carry an $O(h^{-1})$ residual on the $O(N_x{+}N_y)$
edge/corner rows, and $m_k=m_n$ holds up to the $O(\Delta t\,h)$ leakage of
\cref{thm:zeno2d}\,(b).

\emph{(iv)} Subtracting the fixed-point equations from \eqref{eq:picard}, the
iteration error $\bvec e^{[k]}:=\bvec p^{[k]}-\bvec p^{(1)}$ obeys $\bvec
e^{[k+1]}=-T_x^{-1}T_y^{-1}\bvec e^{[k]}$, so by (i)
\begin{equation*}
q\le\norm{T_x^{-1}}_\infty\norm{T_y^{-1}}_\infty
\le\frac{4}{PQ}=\frac{4\,h_xh_y}{\beta^2\Delta t}\le1
\end{equation*}
under \eqref{eq:PQbeta}.  Unconditional stability of the substep follows.

\emph{(v)} The fixed point $\bvec p^{(1)}$ of \eqref{eq:picard} solves
\eqref{eq:factor}, hence \eqref{eq:pade01}, with the bracket evaluated at $\bvec
p^{(1)}$ itself; the iteration is therefore exact at convergence, and the only
approximation relative to the trapezoidal half \eqref{eq:cnstep} is that the
coupling $\bvec\alpha^{+}$ (resp.\ $\bvec\alpha^{+}_2$) carries first-order
(resp.\ second-order) one-sided differences of orientation opposite to the
implicit factors. This replacement carries the prefactor $Q\rho\sqrt{\Delta t}$
(resp.\ $P\sqrt{\Delta t}$) and so contributes a \emph{spatial} defect
$O(\sqrt{\Delta t}\,\max(h_x,h_y))$ for Scheme~A and $O(\max(h_x^2,h_y^2))$ for
Scheme~B, exactly as in \cite{Itkin3D}; it is not a temporal defect. With the
trapezoidal shift $\gamma=\tfrac12$ and the explicit half of \eqref{eq:cnstep},
the assembled central substep reproduces the $\pade(1,1)$ step, so
$\Phi_{xy}(\Delta t)=e^{\Delta t A_{xy}}+O(\Delta t^3)$ is second-order in time;
the backward-Euler choice $\gamma=1$, $\bvec b=\bvec p^{\,n}$ reproduces only the
$\pade(0,1)$ step and is first order.

\section{Proof of \cref{prop:conserv}} \label{appProof2}

Throughout, $\bvec 1_m\in\R^m$ denotes the vector of ones; under column-major
vectorisation $\bvec 1_{n_xn_y}=\bvec 1_{n_x}\otimes\bvec 1_{n_y}$.

\myparagraph{(i).} By the mixed-product property of the Kronecker product,
\begin{equation*}
\bvec 1_{n_xn_y}^\top\bigl(L_x\otimes I_{n_y}\bigr) =\bigl(\bvec 1_{n_x}^\top
L_x\bigr)\otimes\bigl(\bvec 1_{n_y}^\top I_{n_y}\bigr) =0\otimes\bvec
1_{n_y}^\top=0,
\end{equation*}
and symmetrically $\bvec 1^\top A_y=(\bvec 1_{n_x}^\top)\otimes (\bvec 1_{n_y}^\top L_y)=0$.

\myparagraph{(ii).} Let $A\in\R^{N\times N}$ satisfy $\bvec 1^\top A=0$. Then
$\bvec 1^\top A^k=(\bvec 1^\top A)A^{k-1}=0$ for all $k\ge 1$, hence
\begin{equation*}
\bvec 1^\top e^{tA} = \bvec 1^\top\sum_{k\ge 0}\frac{t^k}{k!}A^k = \bvec 1^\top,\qquad t\in\R,
\end{equation*}
i.e., \ $\bvec 1^\top$ is a left eigenvector of $e^{tA}$ with eigenvalue one; the
same holds for the trapezoidal central factor, $\bvec 1^\top\Phi_{xy}(\Delta
t)=\bvec 1^\top$, since it is a rational function of $A_{xy}$ with value $1$ at the
origin (\cref{cor:mass-rat}). Applying this successively to each factor of
$\mathcal S(\Delta t)$, from the left,
\begin{equation*}
\bvec 1^\top\mathcal S(\Delta t) = \bigl(\bvec 1^\top e^{\frac{\Delta t}{2}A_x}\bigr) e^{\frac{\Delta t}{2}A_y}\,\Phi_{xy}(\Delta t)\, e^{\frac{\Delta t}{2}A_y}e^{\frac{\Delta t}{2}A_x} = \dots=\bvec 1^\top,
\end{equation*}
using $\bvec 1^\top A_\alpha=0$ and $\bvec 1^\top\Phi_{xy}=\bvec 1^\top$, which follow
from $\bvec 1^\top C_\alpha=\bvec 1^\top D_\alpha=\bvec 1^\top A_{xy}=0$ and
linearity.  Hence $\bvec 1^\top\bvec p^{n+1}=\bvec 1^\top\mathcal S(\Delta t)\,
\bvec p^{n}=\bvec 1^\top\bvec p^{n}$ for every $\Delta t>0$. The argument uses
only the left null vector of each generator; it is insensitive to the ordering of
the factors and to the splitting error, and it does not require nonnegativity of
any factor.  (When $A_{xy}$ is closed by the plain one-sided stencil,
$\bvec 1^\top A_{xy}=\bvec r^\top$ with $\bvec r$ supported on the edge/corner
rows and $\norm{\bvec r}=O(h^{-1})$ on $O(N_x{+}N_y)$ nodes, and the identity
holds up to the $O(\Delta t\,h)$ defect of \cref{thm:zeno2d}\,(b).)

\myparagraph{(iii).} A matrix $M\ge 0$ with $\bvec 1^\top M=\bvec 1^\top$ is
column-stochastic; products of column-stochastic matrices are column-stochastic,
since nonnegativity and the left-eigenvector property are each preserved under
multiplication. If every factor of $\mathcal S(\Delta t)$ is nonnegative, then by
(ii) each factor is column-stochastic and so is $\mathcal S(\Delta t)$.
Preservation of nonnegativity and of unit mass of $\bvec p^n$ follows by
induction. Finally, for any $\bvec u\in\R^N$,
\begin{equation*}
\|\mathcal S(\Delta t)\,\bvec u\|_1 = \sum_{i}\Bigl|\sum_{j}\mathcal S_{ij}u_j\Bigr|
\le\sum_{j}\Bigl(\sum_{i}\mathcal S_{ij}\Bigr)|u_j| = \sum_j|u_j|=\|\bvec u\|_1,
\end{equation*}
using $\mathcal S_{ij}\ge 0$ and unit column sums. By linearity the same bound
applies to differences of solutions, which is the asserted unconditional
$\ell_1$-stability.

\section{Brief review of popular time integrators} \label{appTI}

This section provides a brief review of several popular time-integration FD schemes. As discussed earlier, our primary focus is on the order of approximation, positivity preservation, and norm conservation — three essential criteria for achieving stable and accurate solutions to the FPE.

Let $k = \Delta t$ denote the time step.  The numerical solution advances as
$\bvec{p}^{n+1} = M\,\bvec{p}^n$ for some propagator matrix $M$.

\subsection{Backward Euler (fully implicit) scheme}

\begin{equation}
\frac{\bvec{p}^{n+1} - \bvec{p}^n}{k} = A\,\bvec{p}^{n+1} \implies (I - kA)\,\bvec{p}^{n+1} = \bvec{p}^n \implies \bvec{p}^{n+1} = (I - kA)^{-1}\bvec{p}^n.
\end{equation}

The propagator is $M_{\mathrm{BE}}=(I-kA)^{-1}$, well defined for all $k>0$ since
the spectrum of $A$ lies in the closed left half-plane and $1/k>0$ is not an
eigenvalue. If $A$ is a (possibly singular) negated M-matrix, i.e.\ Metzler with
$\alpha(A)\le 0$, then by \cref{thm:resolvent} $M_{\mathrm{BE}}\ge 0$ for
\emph{all} $k>0$; if moreover $A$ is irreducible, $M_{\mathrm{BE}}>0$ entrywise,
as follows from the convergent expansion $(I-kA)^{-1}=\frac{1}{1+kc}\sum_{m\ge0}
\bigl(\tfrac{k}{1+kc}\bigr)^{m}B^{m}$ with $A=B-cI$, $B\ge0$.

We emphasise that, unlike the exponential, the resolvent admits no
eventual-positivity threshold at small steps when $A$ is merely an EM-matrix: by
the Neumann expansion $(I-kA)^{-1}=I+kA+O(k^2)$, any negative off-diagonal entry
of $A$ makes $M_{\mathrm{BE}}\not \ge0$ for all sufficiently small $k>0$. What
survives is positivity at \emph{large} steps: for a conservative, irreducible
EM-matrix generator ($\bvec 1^\top A=0$, simple Perron eigenvalue $0$ with
stationary vector $\bvec\pi>0$), the Laurent expansion of the resolvent at the
spectral abscissa gives $(I-kA)^{-1}=\bvec\pi\bvec 1^\top/(\bvec 1^\top\bvec\pi)
+k^{-1}N(1/k)$ with $N$ bounded, whence $M_{\mathrm{BE}}\ge 0$ for all $k\ge k_0$
and some finite $k_0\ge0$ - the resolvent analogue of the threshold $\tau_0$ in
\eqref{eq:posKron}.

For eigenvalue $\lambda_j$ of $A$,
\begin{equation}
\mu_j^{\mathrm{BE}} = \frac{1}{1 - k\lambda_j}.
\end{equation}
For $\re(\lambda_j) < 0$: $|\mu_j| < 1$ (stable). For $k|\lambda_j| \to \infty$: $|\mu_j| \to 0$ (strong damping of all modes). $\lambda_j = 0$ gives $\mu_j = 1$ exactly, i.e., the conserved mass mode, consistent with $\bvec 1^\top(I-kA)^{-1} = \bvec 1^\top$ when $\bvec 1^\top A = 0$.

The BE scheme is 1st order accurate in time, unconditionally stable, and positive for all $k>0$ in the M-matrix case (for all $k\ge k_0$ in the EM case), with no Gibbs oscillations.

\subsection{Crank--Nicolson (trapezoidal rule)}

\begin{equation}
\frac{\bvec{p}^{n+1} - \bvec{p}^n}{k} = \frac{A\bvec{p}^{n+1} + A\bvec{p}^n}{2}
\implies \bvec{p}^{n+1} = M_{\mathrm{CN}}\,\bvec{p}^n,
\end{equation}
where
\begin{equation}  \label{eq:CN}
M_{\mathrm{CN}} = (I - \tfrac{k}{2}A)^{-1}(I + \tfrac{k}{2}A).
\end{equation}

\begin{theorem}[Positivity of the CN propagator: M-matrix case]
\label{thm:CNpos}
Let $A$ be Metzler (all off-diagonal entries nonnegative) with $\alpha(A)\le 0$, and let $b_i=A_{ii}$. Then
\begin{equation}\label{eq:CNbound}
0 < k\le \frac{2}{\max_i |b_i|} \;\Longrightarrow\; M_{\mathrm{CN}}\ge 0,
\end{equation}
and if moreover $A$ is irreducible and the inequality in \eqref{eq:CNbound} is strict, then $M_{\mathrm{CN}}>0$ entrywise.
\end{theorem}
\begin{proof}
By \cref{thm:resolvent} (applicable since $A$ is Metzler), $(I-\tfrac{k}{2}A)^{-1} \ge 0$ for all $k>0$, with strictly positive entries when $A$ is irreducible. The
factor $I+\tfrac{k}{2}A$ has nonnegative off-diagonal entries by the Metzler
property, and its diagonal entries $1+\tfrac{k}{2}b_i$ are nonnegative iff $k\le
2/\max_i|b_i|$; under this condition $I+\tfrac{k}{2}A\ge 0$ and $M_{\mathrm{CN}}
\ge 0$ as a product of nonnegative matrices. If the bound is strict, the diagonal
of $I+\tfrac{k}{2}A$ is strictly positive, so this factor has no zero column, and
the product of a strictly positive matrix with a nonnegative matrix having no
zero column is strictly positive.
\end{proof}

\begin{myremark}[Failure in the EM case] \label{rem:CNEM}

No analogue of \cref{thm:CNpos} holds when $A$ is merely an EM-matrix with a
negative off-diagonal entry $d<0$. For small steps, $M_{\mathrm{CN}} = I+kA + O(k^2)$, so the corresponding entry of $M_{\mathrm{CN}}$ is negative for all sufficiently small $k>0$; and since CN is not L-stable ($\mu_j^{\mathrm{CN}}\to-1$ as $k|\lambda_j|\to\infty$), the large-step mechanism that restores positivity for backward Euler and for the exponential - dominance of the Perron projection after decay of all subdominant modes - is absent, so positivity also fails for all sufficiently large $k$.

Positivity of $M_{\mathrm{CN}}$ for an EM generator can therefore occur at most on an intermediate window of step sizes, whose existence depends on the magnitudes
of the negative stencil entries and must be verified for the matrix at hand. When
provable positivity is required in the EM regime, one should use propagators
whose stability function vanishes at infinity: backward Euler (Section above,
threshold $k_0$), or the exact/polynomial-Krylov exponential (threshold $\tau_0$ of
\eqref{eq:posKron}), rather than the trapezoidal rule.
\end{myremark}

The CN scheme is 2nd order accurate in time: its stability function
$r(z)=(1+z/2)/(1-z/2)$ satisfies $r(z)=e^{z}+O(z^{3})$.  In the present scheme
this trapezoidal rule is exactly the central factor $\Phi_{xy}(\Delta t)$ of
\eqref{eq:central}: the central substep \emph{is} $M_{\mathrm{CN}}$ for the cross
operator $A_{xy}$, while the outer factors $e^{\frac{\Delta t}{2}A_\alpha}$ are
evaluated by the Krylov exponential of \cref{sec:krylov}.  The trapezoidal central
factor adds a local error of the same order $O(\Delta t^{3})$ as the splitting
defect, so the global second-order estimate of \cref{prop:order2} is preserved
(the defect acts on $A_{xy}^{3}\,\mathcal R_h p$, bounded uniformly in $h$ for
$p\in H^{6}$ by the same consistency argument as in \cref{ass:reg}).

Discrete mass is conserved exactly for every $k>0$: from $\bvec 1^\top A=0$ we
get $\bvec 1^\top(I+\tfrac{k}{2}A)=\bvec 1^\top$ and $\bvec 1^\top(I-\tfrac{k}{2}A)^{-1} = \bvec 1^\top$, hence $\bvec 1^\top M_{\mathrm{CN}} = \bvec 1^\top$ - an instance of \cref{rem:approxexp} with $r(0)=1$. As in \cref{prop:conserv}, conservation is unconditional and decoupled from positivity; in particular it persists in the EM regime of \cref{rem:CNEM}, where the propagator is conservative but not nonnegative, so the conserved functional $\bvec 1^\top\bvec p$ may not control $\norm{\bvec p}_1$ (\cref{rem:massnorm}).

Concerning norms: in the regime of \cref{thm:CNpos}, $M_{\mathrm{CN}}\ge0$ with unit column sums is column-stochastic, hence $\norm{M_{\mathrm{CN}}\bvec  u}_1 \le \norm{\bvec u}_1$ unconditionally in that window. In $\ell_2$, A-stability gives
$|\mu_j^{\mathrm{CN}}| = |r(k\lambda_j)|\le1$ for $\re\lambda_j\le0$, but for non-normal $A$ this spectral bound alone does not bound $\norm{M_{\mathrm{CN}}}_2$; the correct statement is that the trapezoidal rule is unconditionally contractive in any inner-product norm in which $A$ is dissipative, i.e., $\mu_2(A)\le0$ implies $\norm{M_{\mathrm{CN}}}_2\le1$ for all $k > 0$ (write $\bvec u=M_{\mathrm{CN}}\bvec v$ as $\bvec u-\bvec v=\tfrac{k}{2}A(\bvec u+\bvec v)$ and take the inner product with $\bvec u+\bvec v$). The drawback relative to backward Euler is the absence of damping at infinity, $r(-\infty) = -1$: stiff modes are reflected rather than suppressed, producing the well-known non-monotone transients from rough initial data, which is the quantitative content of \cref{rem:CNEM}.

Even within the positivity window of \cref{thm:CNpos}, CN transients from rough initial data are non-monotone. For an eigenvalue $\lambda_j$ of $A$ the amplification factor is
\begin{equation}
\mu_j^{\mathrm{CN}} =\frac{1+\tfrac{k}{2}\lambda_j}{1-\tfrac{k}{2}\lambda_j},
\qquad \mu_j^{\mathrm{CN}} = \frac{1-\tfrac{k}{2}|\lambda_j|}{1 + \tfrac{k}{2}|\lambda_j|}\in(-1,1) \ \ \text{for}\ \lambda_j<0 ,
\end{equation}
which is negative for $k|\lambda_j|>2$: such modes change sign at every step. For $k|\lambda_j| \gg 1$, we have $\mu_j^{\mathrm{CN}} \approx -1$, meaning these modes are reflected with little to no damping. This leads to the temporal ringing (a “Gibbs-type” artifact) characteristic of trapezoidal time stepping, which arises from $r(-\infty) = -1$ rather than from any spatial approximation. The artifact manifests as oscillations in the amplitudes of the stiff eigenmodes—i.e., as non-monotonic overshoot of $\bvec p^n$ around the smooth solution. When $M_{\mathrm{CN}} \ge 0$, the iterates themselves remain entrywise nonnegative.

Positivity and ringing are, however, not independent: for the conservative
Metzler matrices of \cref{sec:disc1d} the column Gershgorin disks give
$|\lambda_j|\le 2\max_i|b_i|$, so inside the positivity window $k\le
2/\max_i|b_i|$ every amplification factor satisfies $\mu_j^{\mathrm{CN}} \ge -\tfrac13$: oscillatory components decay by at least a factor of three per step, and persistent ringing ($\mu_j^{\mathrm{CN}}\approx-1$) can occur only for step sizes well outside the window, where positivity has already been lost. Both pathologies are thus governed by the single quantity $k\max_i|b_i|$. Since  $\max_i|b_i| = O(h^{-2})$, enforcing the window imposes the parabolic restriction $k=O(h^{2})$, which is precisely the practical objection to CN in this context and the reason we prefer propagators with $r(-\infty)=0$ (backward Euler) or the exact exponential, for which positivity is achieved without a step-size ceiling (\cref{thm:resolvent}, \eqref{eq:posKron}).

\subsection{TR-BDF2}

TR-BDF2 combines a trapezoidal (CN) sub-step with a BDF2 sub-step.  Let
$\gamma \in (0,1)$ be a splitting parameter; the classical choice is
$\gamma = 2 - \sqrt{2} \approx 0.5858$.

\myparagraph{Sub-step 1: Trapezoidal step to $t^{n+\gamma}$}
\begin{equation}
\left(I - \frac{\gamma k}{2}A\right)\bvec{p}^{n+\gamma} = \left(I + \frac{\gamma k}{2}A\right)\bvec{p}^n,
\end{equation}
giving
\begin{equation}
\bvec{p}^{n+\gamma} = M_1\,\bvec{p}^n, \qquad M_1 = \left(I - \frac{\gamma k}{2}A\right)^{-1} \left(I + \frac{\gamma k}{2}A\right).
\end{equation}
By \cref{thm:CNpos} with $k \to \gamma k$, $M_1 > 0$ provided $\gamma k < 2/|b|$, i.e.,
\begin{equation}
k < \frac{2}{\gamma|b|}.
\end{equation}
Since $\gamma < 1$, this is a \emph{looser} constraint than for CN.

\myparagraph{Sub-step 2: BDF2 step from $t^{n+\gamma}$ to $t^{n+1}$.}
The variable-step BDF2 formula with stage ratio $\omega=(1-\gamma)/\gamma$ reads
\begin{equation}\label{eq:BDF2step}
\bigl(I-\alpha kA\bigr)\bvec p^{n+1} = a\,\bvec p^{n+\gamma}-b\,\bvec p^{n},
\qquad \alpha=\frac{1-\gamma}{2-\gamma},\quad a=\frac{(1+\omega)^{2}}{1+2\omega},\quad b=\frac{\omega^{2}}{1+2\omega},
\end{equation}
with $a-b=1$. For the classical $\gamma=2-\sqrt2$ one has $\omega=1/\sqrt2$,
$a=\tfrac{1+\sqrt2}{2}$, $b=\tfrac{\sqrt2-1}{2}$, and the special property
$\alpha=\gamma/2$: both stages involve the same matrix $I-\tfrac{\gamma k}{2}A$
\cite{BankEtAl1985,HoseaShampine1996}.

\myparagraph{Positivity of TR-BDF2.}
Although the right-hand side of \eqref{eq:BDF2step} carries a negative weight
$-b<0$ on $\bvec p^{n}$, so that no stage-wise positivity argument applies, the
negative weight cancels in the composite one-step map. For $\gamma=2-\sqrt2$ the
two resolvents coincide and
\begin{equation}\label{eq:TRBDF2factor}
M_{\mathrm{TRBDF2}} = \bigl(I-\tfrac{\gamma k}{2}A\bigr)^{-2} \bigl(I + (\sqrt2-1)\,kA\bigr),
\end{equation}
since $a-b=1$ and $(a+b)\gamma/2=\sqrt2-1$.

\begin{theorem}[Positivity of the TR-BDF2 propagator: M-matrix case] \label{thm:TRBDF2pos}
Let $A$ be Metzler with $\alpha(A)\le0$, $b_i=A_{ii}$, and $\gamma=2-\sqrt2$. Then
\begin{equation}\label{eq:TRBDF2bound}
0 < k \le \frac{1+\sqrt2}{\max_i|b_i|} \;\Longrightarrow\; M_{\mathrm{TRBDF2}}\ge0,
\end{equation}
with strict entrywise positivity under irreducibility and strict inequality, by
the same argument as in \cref{thm:CNpos}. The window \eqref{eq:TRBDF2bound} exceeds the CN window \eqref{eq:CNbound} by the factor
$(1+\sqrt2)/2\approx1.21$.
\end{theorem}
\begin{proof}
In \eqref{eq:TRBDF2factor} the factor $(I-\tfrac{\gamma k}{2}A)^{-2}\ge0$ for all
$k>0$ by \cref{thm:resolvent}; the factor $I+(\sqrt2-1)kA$ is Metzler off
the diagonal and has diagonal $1+(\sqrt2-1)k\,b_i\ge0$ iff \eqref{eq:TRBDF2bound}
holds.
\end{proof}

No second-order one-step rational scheme can do better than such a window:
unconditional positivity on Metzler generators forces order $\le1$ \cite{BolleyCrouzeix1978}); backward Euler saturates this barrier, and CN and TR-BDF2 realise second order at the price of a finite positivity window, $k=O(h^{2})$ since $\max_i|b_i|=O(h^{-2})$.

\myparagraph{Stiff damping, conservation, EM case.}
The stability function
$r(z) = \bigl(1+(\sqrt2-1)z\bigr)/\bigl(1-(1-\tfrac1{\sqrt2})z\bigr)^{2}
= 1 + z + \tfrac{z^{2}}{2}+O(z^{3})$
satisfies $r(-\infty)=0$: TR-BDF2 is L-stable, so, unlike CN, stiff modes are
damped rather than reflected and no temporal ringing occurs at any step size.
Mass is conserved exactly for every $k>0$, since $r(0)=1$ and $\bvec 1^\top A=0$
imply $\bvec 1^\top M_{\mathrm{TRBDF2}}=\bvec 1^\top$ (\cref{rem:approxexp}).

In the EM case, positivity for small $k$ fails by the Neumann argument
($M_{\mathrm{TRBDF2}}=I+kA+O(k^{2})$), exactly as for BE and CN; but
L-stability restores the large-step mechanism of backward Euler: for a
conservative irreducible EM generator, $r(-\infty)=0$ damps all
subdominant modes as $k\to\infty$ and $M_{\mathrm{TRBDF2}} \to \bvec\pi\bvec1^\top/(\bvec1^\top \bvec\pi)>0$, so $M_{\mathrm{TRBDF2}} \ge 0$ for all $k\ge k_0$, some finite threshold - the property CN lacks (\cref{rem:CNEM}).

\end{document}